\documentclass{article}
\usepackage[utf8]{inputenc}
\usepackage{amsmath}
\usepackage{amsthm}
\usepackage{amsfonts}
\usepackage{amssymb, amsbsy}
\usepackage{color}
\usepackage{hyperref}
\usepackage{enumitem}
\usepackage{comment}

\usepackage[a4paper, textwidth=14cm, top=3cm]{geometry}

\usepackage{mathtools}
\usepackage{cancel}

\title{Extension operators and Korn inequality for variable coefficients in perforated domains with applications to homogenization of viscoelastic non-simple materials}
\author{Markus Gahn}

\date{}

\newcommand{\R}{\mathbb{R}}
\newcommand{\N}{\mathbb{N}}
\newcommand{\Z}{\mathbb{Z}}
\newcommand{\vareps}{\varepsilon}
\newcommand{\oeps}{\Omega_{\varepsilon}}
\newcommand{\geps}{\Gamma_{\varepsilon}}
\newcommand{\veps}{v_{\varepsilon}}
\newcommand{\ueps}{u_{\varepsilon}}

\newcommand{\weps}{w_{\varepsilon}}
\newcommand{\tveps}{\widetilde{v}_{\varepsilon}}
\newcommand{\tueps}{\widetilde{u}_{\varepsilon}}
\newcommand{\dueps}{\dot{u}_{\varepsilon}}
\newcommand{\dtueps}{\dot{\widetilde{u}}_{\varepsilon}}
\renewcommand{\oe}{\Omega_{\varepsilon}}

\newcommand{\E}{\mathcal{E}}

\newcommand{\hY}{\widehat{Y}}

\newcommand{\per}{\mathrm{per}}

\newcommand{\teps}{\mathcal{T}_{\varepsilon}}
\newcommand{\meps}{\mathcal{M}_{\varepsilon}}

\newcommand{\fxe}{\frac{x}{\vareps}}

\newcommand{\Meps}{\mathcal{M}_{\varepsilon}}

\newcommand{\ie}{i.\,e.,\,}

\newcommand{\rightwts}[1]{\xrightharpoonup[]{2,#1}}
\newcommand{\rightsts}[1]{\xrightarrow[]{2,#1}}

\newcommand{\uepstau}[1]{u_{\varepsilon,\tau}^{(#1)}}
\newcommand{\lepstau}[1]{l_{\varepsilon,\tau}^{(#1)}}
\newcommand{\fepstau}[1]{f_{\varepsilon,\tau}^{(#1)}}
\newcommand{\gepstau}[1]{g_{\varepsilon,\tau}^{(#1)}}
\newcommand{\tuepstau}[1]{\tilde{u}_{\varepsilon,\tau}^{(#1)}}
\newcommand{\huepstau}{\widehat{u}_{\varepsilon,\tau}}
\newcommand{\dhuepstau}{\dot{\widehat{u}}_{\varepsilon,\tau}}
\newcommand{\buepstau}{\overline{u}_{\varepsilon,\tau}}
\newcommand{\lbuepstau}{\underline{u}_{\vareps,\tau}}
\newcommand{\blepstau}{\overline{l}_{\varepsilon,\tau}}

\newtheorem{theorem}{Theorem}[section]
\newtheorem{corollary}[theorem]{Corollary}
\newtheorem{lemma}[theorem]{Lemma}
\newtheorem{proposition}[theorem]{Proposition}
\newtheorem{remark}[theorem]{Remark}
\newtheorem{definition}[theorem]{Definition}

\newcommand{\id}{\mathrm{id}}
\newcommand{\Yepsid}{\mathcal{Y}_{\mathrm{id},\vareps}}
\newcommand{\zeps}{z_{\varepsilon}}
\newcommand{\dtau}{\delta_{\tau}}
\newcommand{\Reps}{R_{\varepsilon}}
\newcommand{\ssubset}{\subset\joinrel\subset}

\newcommand{\theps}{\widetilde{h}_{\varepsilon}}

\begin{document}

\maketitle
\vspace{-2em}
\begin{center}
\begin{minipage}{25em}
\centering\textit{Institute for Mathematics, University Heidelberg, \\
Im Neuenheimer Feld 205, 69120 Heidelberg, Germany.}
\end{minipage}
\end{center}

\vspace{2em}
\begin{abstract}
In this paper we present the homogenization for nonlinear viscoelastic second-grade non-simple perforated materials at large strain in the quasistatic setting. The reference domain $\oeps$ is periodically perforated and is depending on the scaling parameter $\vareps$ which describes the ratio between the size of the whole domain and the small periodic perforations. The mechanical energy depends on the gradient and also the second gradient of the deformation, and also respects positivity of the determinant of the deformation gradient. For the viscous stresses we assume dynamic frame indifference and is therefore depending of the rate of the Cauchy-stress tensor. For the derivation of the homogenized model for $\vareps \to 0$ we use the method of two-scale convergence. For this uniform \textit{a priori} estimates with respect to $\vareps$ are necessary. The most crucial part is to estimate the rate of the deformation gradient. Due to the time-dependent frame indifference of the viscous term, we only get coercivity with respect to the rate of the Cauchy-stress tensor. To overcome this problem we derive a Korn inequality for non-constant coefficients on the perforated domain. The crucial point is to verify that the constant in this inequality, which is usually depending on the domain, can be chosen independently of the parameter $\vareps$. Further, we construct an extension operator for second order Sobolev spaces on perforated domains with operator norm independent of $\vareps$. 
\end{abstract}

\section{Introduction}

The aim of this paper is the rigorous homogenization of a microscopic model for viscoelastic  (second-grade) non-simple  periodically perforated materials at large strains in the quasistatic setting, where our model in particular includes dynamic frame indifference for the viscous stresses. For this, we develop new multi-scale techniques adapted to the underlying microscopic problem, such as a Korn inequality with non-constant coefficients in perforated domains
 and an extension operator for second order Sobolev functions defined on the perforated domain. We emphasize that these methods are not restricted to the micro-model treated in this paper, but also can be applied to more general problems including  for example additional transport processes.  Hence, the contribution of the paper is two-folded: Derivation of general multi-scale techniques arising in particular in continuum mechanics, as well as the homogenization of a viscoelastic model for non-simple materials.

Our microscopic model is motivated by \cite{mielke2020thermoviscoelasticity}, see also \cite[Chapter 9]{kruvzik2019mathematical} for an overview, where existence for a thermoviscoelastic model was shown by using a staggered time-discretization together with a regularization argument. This method was later improved in \cite{badal2023nonlinear}, where they where able to avoid the regularization argument. For the pure viscoelastic case with self-penetration we refer to \cite{Kroemer2020}. In our paper we only consider the pure viscoelastic problem (\cite{mielke2020thermoviscoelasticity} with constant temperature). More precisely, we consider the differential equation
\begin{align*}
- \nabla \cdot \left( \partial_F W_{\vareps}(\nabla \ueps) + \partial_{\dot{F}} R_{\vareps} (\nabla \ueps , \nabla \dueps) - \nabla \cdot (\partial_G H_{\vareps}(\nabla^2 \ueps)) \right) &= f_{\vareps} &\mbox{ in }& (0,T) \times \oeps
\end{align*}
for a fixed finite time $T>0$ and the microscopic periodically perforated domain $\oeps$, where the small parameter $\vareps$ describes the ratio between the macroscopic domain and the micro-cells (perforations). The system is closed by suitable boundary conditions. Here, $W_{\vareps}$ is the elastic energy potential, $H_{\vareps}$ the convex strain gradient energy potential, and $R_{\vareps}$ the potential of dissipative forces. Our aim is the derivation of a macroscopic model for $\vareps\to 0$ formulated on a fixed macroscopic domain $\Omega$.

For fixed $\vareps$ the existence of a weak solution was already established in \cite{mielke2020thermoviscoelasticity} and \cite{badal2023nonlinear}. However, for the homogenization process it is crucial to have suitable \textit{a priori} estimates uniformly with respect to the scaling parameter $\vareps$. These estimate build the basis when passing to the limit $\vareps \to 0$. Therefore, the first essential step is the derivation of such \textit{a priori} estimates. For this we follow the existence proof in \cite{badal2023nonlinear} (resp. \cite{mielke2020thermoviscoelasticity}) via a time-discretization, where in every time step a minimization problem is solved. Now, the continuous solution can be obtained as the approximation via piecewise constant/linear interpolations, where uniform \textit{a priori} estimates with respect to the time step are necessary.  To obtain $\vareps$-uniform estimates for the microscopic solution, we show that the estimates for the discrete solutions are also independent of $\vareps$. Such estimates are obtained via the assumptions on the data and general inequalities like Korn-, Poincar\'e, Sobolev- or the trace-inequality. These inequalities include constants depending on the geometry of the domain and therefore on $\vareps$, and we have to exhibit the explicit dependence on $\vareps$. We emphasize that in \cite{mielke2020thermoviscoelasticity} a balance equation for the mechanical energy was derived, which would allow to establish $\vareps$-uniform \textit{a priori} estimates without using the time-discretization. However, here we prove the estimates via the time-discretization for several reasons: First of all, for the reader who is not familiar with the results in \cite{badal2023nonlinear} and \cite{mielke2020thermoviscoelasticity} we can repeat/summarize the existence proof for the pure viscoelastic problem, which is much simpler as for the  thermoviscoelastic case. Secondly, our methods also include the techniques for the treatment of (stationary) problems for non-simple second-grade materials. Finally, our results are also a first step in the treatment of the much more complicated homogenization for the thermovisoelastic problem, where we expect to obtain \textit{a priori} estimates only via the time-discretization, since the heat equation has no gradient structure.

Our model includes dynamic frame indifference for the viscous stresses what means that the material response is unaffected by time-depending rigid body motions. Thus the viscous term including $\partial_{\dot{F}}R_{\vareps}(\nabla \ueps,\nabla \dueps)$ is depending on the Cauchy-Green tensor $\nabla \ueps \nabla\ueps^{\top}$ and its rate $\nabla \dueps^{\top} \nabla \ueps + \nabla \ueps^{\top} \nabla \dueps$. Hence, to control the rate of the deformation gradient by the rate of the Cauchy-Green tensor we need an inequality of the form
\begin{align}\label{ineq:introduction}
\int_{\oeps} | \nabla \dueps|^2 dx \le C_{\vareps} \int_{\oeps} \left| \nabla \ueps^{\top} \nabla \dueps + \nabla \dueps^{\top} \nabla \ueps \right|^2 dx.
\end{align}
This is a Korn inequality with the non-constant coefficients $\nabla \ueps$ on the perforated domain $\oeps$. Such an inequality was shown in \cite[Theorem 3.3]{mielke2020thermoviscoelasticity} for $\{\nabla \ueps\}_{\vareps >0}$ compact in the space of continuous matrix valued functions $C^0(\overline{\Omega})^{n\times n}$, positive determinant of the deformation gradient $\det (\nabla \ueps)\geq \kappa >0$, and for fixed domains $\Omega$ (not depending on $\vareps$). Hence, we first have to establish that the sequence of deformation gradients has determinant bounded away from zero uniformly with respect to $\vareps$. However, the most crucial point is to show that the Korn inequality $\eqref{ineq:introduction}$ with $A \in C^0(\overline{\Omega})^{n\times n}$ instead of $\nabla \ueps$ and $\det (A) \geq \kappa >0$ is valid with a constant $C_{\vareps}>0$ which can be chosen independently of $\vareps$. For this we approximate $A$ by  $A_{\vareps}$ constant on every micro-cell and construct an extension operator $E_{A_{\vareps}}: W^{1,p}(\oeps)^n \rightarrow W^{1,p}(\Omega)^n$ such that
\begin{align*}
\left\|A_{\vareps}^{\top}\nabla\left(E_{A_{\vareps}}\veps\right) +  \nabla\left(E_{A_{\vareps}}\veps\right)^{\top} A_{\vareps} \right\|_{L^p(\Omega)} \le C \left\| A_{\vareps}^{\top} \nabla \veps + \nabla \veps^{\top} A_{\vareps} \right\|_{L^p(\oeps)}.
\end{align*}
This allows to transfer the problem to the fixed domain $\Omega$ where we can use a perturbation argument (since $A_{\vareps}$ approximates $A$) and the Korn inequality with continuous non-constant coefficients from \cite{pompe2003korn}.

To pass to the limit $\vareps \to 0$ in the microscopic problem we use the method of two-scale convergence, see \cite{Allaire_TwoScaleKonvergenz,Nguetseng}, and the unfolding method, see \cite{Cioranescu_Unfolding1} and especially the monograph \cite{CioranescuGrisoDamlamian2018}.  
Since the microscopic deformation has values in the second order Sobolev space $W^{2,p}(\oeps)^n$ and our problem is formulated on the perforated domain $\oeps$ we have to extend these microscopic solutions to the whole macroscopic domain $\Omega$. For this we construct an extension operator from $W^{2,p}(\oeps)$ to $W^{2,p}(\Omega)$ such that the norm for the extended function can be controlled by the norm of the function itself uniformly with respect to $\vareps$. We emphasize that since the rate of the deformation gradient is only an element of $H^1(\oeps)^n$, this extension operator has to act simultaneously on $W^{2,p}(\oeps)$ and $H^1(\oeps)$, or in other words it has to be total. In this way we can treat sequences on the perforated domain $\oeps$ as sequences defined on the whole macroscopic domain $\Omega$. Another important consequence of such extension operators is the possibility to establish general inequalities (like for example Poincar\'e inequality or emeddings) with constants independent of $\vareps$, which is crucial to obtain \textit{a priori} estimates for the microscopic solution.
\\
For the derivation of the macroscopic model for $\vareps \to 0$ 
we use the (two-scale) compactness results obtained from the $\vareps$-uniform \textit{a priori} bounds for the microscopic solution given in the space
\begin{align*}
L^{\infty}((0,T),W^{2,p}(\oeps))^n\cap H^1((0,T),H^1(\oeps))^n.
\end{align*}
We obtain strong convergence results for the deformation $\ueps$ and its gradient $\nabla \ueps$, or more precisely for their extensions, and also the weak two-scale convergence for the rate of the deformation gradient $\nabla \dueps$. This allows to pass to the limit in the first order terms. The critical part is the convergence of the nonlinear second order term, since we only have the weak two-scale convergence of $\nabla^2 \ueps$. To overcome this problem we use some kind of Minty-trick. In the macroscopic model the first order terms (depending on the gradient of the deformation and its rate) occur as averaged quantities, while for the second order term depending on the second gradient we have to solve a nonlinear cell problem.
\\

\noindent\textbf{Literature overview:} As already mentioned above existence for the microscopic problem for fixed $\vareps$ was already established for the more general thermoviscoelastic case in \cite{mielke2020thermoviscoelasticity}, and later with an improved time-discretization in \cite{badal2023nonlinear}. In \cite{Kroemer2020} the pure viscoelastic problem with self-penetration was analyzed. We also refer to the monograph \cite{kruvzik2019mathematical} for an overview of such problems and also some results including inertial forces are given.
\\
There is a huge literature for  homogenization in linear elasticity, where we in particular mention the monograph \cite{Oleinik1992}.  Also for nonlinear problems in elasticity without viscosity a lot of results have been obtained. Here we have to mention the seminal works \cite{Marcellini1978} for convex integrals and \cite{Braides1985,Mueller1987} for non-convex integrals, where the method of $\Gamma$-convergence (see for example \cite{dal1993introduction}) introduced in \cite{de1975tipo} was used. A nonlinear Kelvin-Voigt model for viscoelastic materials with small deformations was considered in \cite{Visintin2006}. In contrast, to the best of the authors knowledge, for nonlinear viscoelastic problems no rigorous homogenization results exist. Here a crucial problem is the $\vareps$-uniform positive lower bound for the determinant of the deformation gradient. In \cite{GahnPop2023Mineral} and \cite{WIEDEMANN2023113168} a reaction-diffusion equation in an evolving microstructure depending on the concentration of the system was homogenized. However, in this case the evolution of the geometry (the free boundary) was given by an ordinary differential equation with strong constraints, such that the evolution of the perforated domain reduces to a one-dimensional problem, and the control of the determinant from below becomes much more simple. In \cite{nika2024effective} they homogenized a linear problem for elastic non-simple second-grade materials on a fixed domain ($\vareps$-independent).
A Korn inequality for non-constant coefficients in perforated domains was shown in \cite{MR5411390-Wiedemann-Peter-2024} for functions with vanishing traces on (a part of) the perforations in every micro-cell. This allows to use a decomposition argument and use the Korn inequality from \cite{mielke2020thermoviscoelasticity} respectively \cite{pompe2003korn} on the reference element. In our setting this procedure is not possible what makes the situation much more complicated.
Finally, we mention the work \cite{braides2000quasiconvexity} where homogenization results in the context of $\Gamma$-convergence for quasiconvex integrals under differential constraints were obtained. Especially, those results include cell problems for integrals depending on higher order derivatives.
\\

\noindent Finally, let us summarize the main contributions of the paper:
\begin{itemize}
\item Korn inequality with non-constant coefficients on perforated domains;
\item Norm-preserving extension operator for second order Sobolev spaces;
\item $\vareps$-uniform positive lower bound for the determinant of the deformation gradient for sequences with bounded mechanical energy;
\item Homogenization (including $\vareps$-uniform \textit{a priori} estimates) for the microscopic model for viscoelastic second-grade non-simple materials in perforated domains.
\end{itemize}\

The paper is organized as follows: In Section \ref{sec:Micro_Model} we formulate the microscopic model and its underlying geometry. Further, we give the assumptions on the data. In Section \ref{sec:Main_Results} we formulate the main results of the paper and the macroscopic model.  
The Korn inequality for perforated domains with non-constant coefficients is shown in Section \ref{sec:Korn_inequality} (for the case $\Omega \setminus \oeps$ disconnected). In Section \ref{sec:Extension_second_order} we construct the extension operator on $W^{2,p}(\oeps)$.
Existence for the microscopic problem with $\vareps$-uniform \textit{a priori} estimates are shown in Section \ref{sec:existence_a_priori_estimates}. Finally, in Section \ref{sec:derivation_macro_model} we derive the macroscopic model for $\vareps \to 0$. In the Appendix \ref{sec:auxiliary_results} we formulate some technical results, in particular general inequalities on perforated domains and also a parabolic embedding result. The concept of two-scale convergence and unfolding method with some basic properties is introduced in Appendix \ref{SectionTwoScaleConvergence}. The proof of the Korn inequality for $\Omega \setminus \oeps$ connected is finally given in Appendix \ref{sec:Korn_connected}.

\subsection{Notations}

We denote be $E_n\in \R^{n\times n}$ the identity matrix and $\id:\R^n\rightarrow \R^n$ with $\id(x)=x$ stands for the identity map. Further we set $GL^+(n):=\{F\in \R^{n\times n} \, : \, \det(F)>0\}$. For a bounded Lipschitz domain $\Omega \subset \R^n$ we denote by $\nu$ the outer unit normal on $\partial \Omega$. The inner products between vectors $a,b\in \R^n$, matrices $A,B \in \R^{n\times n}$, and 3rd order tensors $T,S\in \R^{n\times n\times n}$ we denote by 
\begin{align*}
a\cdot b:= \sum_{i=1}^n a_i b_i,\quad A:B:= \sum_{i,j=1}^n A_{ij}B_{ij}, \quad T\vdots S:= \sum_{i,j,k=1}^n T_{ijk}S_{ijk}.
\end{align*}

Let $n,d\in \N$, then for $\Omega\subset \R^n$ we denote by $L^p(\Omega)^d, \, W^{1,p}(\Omega)^d $ the standard Lebesgue and Sobolev spaces with $p \in [1,\infty]$. Especially, for $p=2$ we write $H^1(\Omega)^d:= W^{1,2}(\Omega)^d$. For $\Omega$ a Lipschitz domain and $\Gamma \subset \partial \Omega$ we define
\begin{align*}
W_{\Gamma}^{1,p}(\Omega) : = \left\{ u \in W^{1,p}(\Omega) \, : \, u=0 \mbox{ on } \Gamma \right\}.
\end{align*}
 For the norms we neglect the upper index $d$, for example we write $\|\cdot\|_{L^p(\Omega)}$ instead of $\|\cdot \|_{L^p(\Omega)^d}$. For a Banach space $X$ and $p \in [1,\infty]$ we denote the usual Bochner spaces by $L^p(\Omega,X)$. For the dual space of $X$ we use the notation $X'$. Further, we consider the following periodic function spaces. Let $Y:= (0,1)^n$ be the unit cell in $\R^n$, then $C^{\infty}_{\per}(Y)$ is the space of smooth functions on $\R^n$ which are $Y$-periodic. $W^{k,p}_{\per}(Y)$ for $k\in \N_0$ is the closure of $C^{\infty}_{\per}(Y)$ with respect to the norm on $W^{k,p}(Y)$. For a subset $Y_s\subset Y$ we denote by $W^{k,p}_{\per}(Y_s)$ the space of  functions from  $W^{k,p}_{\per}(Y)$ restricted to $Y_s$.

The (weak) gradient of a function $f: \R^n\supset\Omega \rightarrow \R^m$ is denoted by $\nabla f$ with $(\nabla f)_{ij} = \partial_j f_i$ for $i=1,\ldots,m$ and $j=1,\ldots,n$. We emphasize that this is a slight abuse of notation since the gradient of a scalar valued function is usually the transpose of the Jacobi matrix, while here the gradient is equal to the Jacobi matrix. However, we would like to keep the notation close to the related literature. We also identify the gradient of a scalar function with a vector in $\R^n$. Further, we denote the second gradient by $\nabla^2 f: \Omega \rightarrow \R^{m\times n \times n}$ with $(\nabla^2 f)_{ijk} = \partial_{jk} f_i = \partial_{kj} f_i$ for $j,k\in \{1,\ldots,n\} $ and $i \in \{1,\ldots,m\}$. 
For scalar valued functions $W:\R^{n \times n} \rightarrow \R$ and $H:\R^{n\times n \times n } \rightarrow \R$ we identify their Fr\'echet derivatives $\partial_F W(F)$ for $F\in \R^{n\times n}$ and $\partial_G H(G) $ for $G \in \R^{n\times n \times n}$ with their related transformation matrix respectively 3rd order tensor with respect to the standard basis in $\R^{n\times n}$ and $\R^{n\times n \times n}$. 
The symmetric gradient for a (weakly differentiable) function $u:\R^n \supset \Omega \rightarrow \R^n$ is denoted by $e(u):=\frac12 (\nabla u + \nabla u^{\top})$.

Next,  we introduce some notations which are necessary for the formulation of the strong microscopic and macroscopic model. For the analysis we only work with the associated weak formulation, for which the following notations are not necessary. Hence, the reader who is only interested in the analysis and not in the modelling may skip this part.
For $G \in \R^{n\times n\times n} $ and a vector $a \in \R^n$ we define
\begin{align*}
Ga\in \R^{n\times n}, \quad (Ga)_{ij} = \sum_{k=1}^n G_{ijk}a_k \quad\mbox{for } i,j=1,\ldots,n,
\end{align*}
and for $A \in \R^{n\times n}$ we set
\begin{align*}
G:A \in \R^n, \quad (G:A)_i := \sum_{j,k=1}^n G_{ijk} A_{jk} \quad\mbox{ for } i =1, \ldots,n.
\end{align*}
The divergence of a matrix valued function $A:\R^n \supset \Omega \rightarrow \R^{n\times n}$ (smooth enough) is defined  for every row
\begin{align*}
(\nabla \cdot A)_i = \nabla \cdot (A^{\top}e_i) = \sum_{j=1}^n \partial_j A_{ij} \quad\mbox{ for } i=1,\ldots,n.
\end{align*}
The divergence of 3rd-order tensor valued functions $G:\R^n \supset \Omega \rightarrow \R^{n\times n \times n }$ is defined by
\begin{align*}
(\nabla \cdot G)_{ij} = \sum_{k=1}^n \partial_k G_{ijk} \mbox{ for } i,j=1,\ldots,n.
\end{align*}
For a smooth vector field $f$ in the neighborhood of $\partial \Omega$ we define its surface divergence by
\begin{align*}
\nabla_s \cdot f := \nabla \cdot f - \nu \cdot (\nabla f)\nu.
\end{align*}
For a smooth matrix valued function $A$ we define its surface divergence by
\begin{align*}
(\nabla_s \cdot A)_i:= \nabla_s \cdot (A^{\top} e_i).
\end{align*}
In other words, it is defined for every row.

\section{Microscopic model}
\label{sec:Micro_Model}

In this section we introduce the microscopic model. We start with the underlying microscopic geometry.
For $n \in \N$, $n \geq 2$ and $a,b \in \Z^n$ such that $a_i < b_i$ for all $i = 1, \dots n$, we consider the hyper-rectangle $\Omega = (a,b)\subset \R^n$ as the macroscopic domain. Further $\epsilon >0$ is a sequence of small parameters such that $\epsilon^{-1} \in \N$. 
$\oeps \subset \Omega$ is the periodically perforated microscopic domain constructed as follows. With $Y:=(0,1)^n$, we let $Y_s \subset Y$ be a connected subdomain with opposite faces matching each other, \ie for $i=1,\ldots, n$ it holds that
\begin{align*}
\partial Y_s \cap \{x_i = 0\} + e_i = \partial Y_s \cap \{x_i = 1\}.
\end{align*}
We define $\Gamma:= \partial Y_s \setminus \partial Y$ and assume that $\Gamma $ is a Lipschitz boundary. Further, let 
$K_{\epsilon}:= \left\{k \in \Z^n\, : \, \epsilon(Y + k) \subset \Omega \right\}$. Clearly, $\Omega = \mathrm{int} \left(\bigcup_{k \in K_{\epsilon}}  \epsilon \big(\overline{Y} + k \big) \right)
$, where \textit{int} denotes the interior of a set. Now we define $\oe$ by
\begin{align*}
\oe:= \mathrm{int} \left(\bigcup_{k \in K_{\epsilon}}  \epsilon \big(\overline{Y_s} + k \big) \right)
\end{align*}
and the oscillating  (inner) boundary $\geps$ as 
\begin{align*}
\geps := \partial \oe \setminus \partial \Omega.
\end{align*}
We assume that $\oeps $ is connected and $\geps$ is a Lipschitz boundary. Hence, we have $|\partial Y_s \cap \{x_i=0\}| >0$ for $i=1,\ldots,n$.  We emphasize that the complement $\Omega \setminus \oeps$ may be connected (if $n\geq 3$) or disconnected. 
Further, let $\Gamma^N,\Gamma^D\subset \partial \Omega$ with $\Gamma^N \cap \Gamma^D = \emptyset$, $\overline{\Gamma^D} \cup \overline{\Gamma^N} = \partial \Omega $, and $|\Gamma^D|>0$. We now define 
\begin{align*}
\geps^D := \mathrm{int}\left( \partial \oeps \cap  \Gamma^D\right),\qquad \geps^N:= \mathrm{int}\left( \partial \oeps \cap \Gamma^N\right).
\end{align*}
If $Y \setminus Y_s $ is strictly included in $Y$ (in other words $\Omega \setminus \oeps$ is disconnected), we have $\geps^D = \Gamma^D$ and $\geps^N = \Gamma^N$.
We assume that for every $\vareps$ the sets $\Gamma^D$ and $\Gamma^N$ can be decomposed into small micro-cells $\vareps \left(S\left((0,1)^{n-1} \times \{0\}\right) + k\right) $ with a rotation $S\in \R^{n\times n}$ and $k \in \Z^n$. A simple example would be that $\Gamma^N$ and $\Gamma^D$ are sides of the hyper-rectangle $\Omega$.

\begin{remark}
We made the assumption that $\Omega$ is a rectangle. This could be easily generalized to domains which are unions of rectangles, and a sequence $\vareps$ such that this domain can be decomposed exactly in shifted and scaled micro-cells $\vareps (k + Y)$ with $k \in \Z^n$. In other words, no micro-cells intersect with the boundary $\partial \Omega$. 
\end{remark}

\subsection{Microscopic equations}

For $T>0$ we are looking for a deformation $\ueps : (0,T)\times \oeps \rightarrow \R^n$ which solves the microscopic problem
\begin{align}
\begin{aligned}\label{MicroModel}
- \nabla \cdot \left( \partial_F W_{\vareps}(\nabla \ueps) + \partial_{\dot{F}} R_{\vareps} (\nabla \ueps , \nabla \dueps) - \nabla \cdot (\partial_G H_{\vareps}(\nabla^2 \ueps)) \right) &= f_{\vareps} &\mbox{ in }& (0,T) \times \oeps,
\\
\ueps &= \id &\mbox{ on }& (0,T)\times \geps^D,
\\
\left(\partial_F W_{\vareps}(\nabla \ueps ) + \partial_{\dot{F}} R(\nabla\ueps , \nabla \dueps ) - \nabla \cdot (\partial_G H_{\vareps}(\nabla^2 \ueps))\right) \nu 
\\
-\nabla_s\cdot \left(\partial_G H_{\vareps}(\nabla^2 \ueps) \nu\right)  &= g_{\vareps} &\mbox{ on }& (0,T)\times \big( \geps \cup \geps^N\big),
\\
\partial_G H_{\vareps} (\nabla^2 \ueps) : (\nu \otimes \nu) &= 0 &\mbox{ on }&  (0,T) \times \partial\oeps,
\\
\ueps(0) &= \ueps^{in} &\mbox{ in }& \oeps.
\end{aligned}
\end{align}
We assume that $p>n$ (in the microscopic model) and define the solution space of the microscopic deformation $\ueps$  with respect to the spatial variable as
\begin{align*}
\Yepsid:= \left\{ \veps \in W^{2,p}(\oeps)^n \, : \, \ueps = \id \mbox{ on } \geps^D, \, \det(\nabla \veps ) >0 \mbox{ in } \oeps\right\}.
\end{align*}
Additionally, for given $\kappa >0$ we define the subspace
\begin{align*}
\Yepsid^{\kappa}:= \left\{\veps \in \Yepsid \, : \, \det(\nabla \veps) \geq \kappa \right\}.
\end{align*}
In the following we define a weak solution for the microscopic problem. We refer to \cite[Theorem 2.5.2]{kruvzik2019mathematical} for the relation between the strong and the weak formulation.
\begin{definition}\label{def:weak_sol_micro_model} We say $\ueps :(0.T)\times \oeps \rightarrow \R^n$ is a weak solution of the microscopic problem $\eqref{MicroModel}$ if $\ueps \in L^{\infty}((0,T),\Yepsid)\cap H^1((0,T),H^1(\oeps))^n$ with $\ueps(0) = \ueps^{in}$ and it holds 
\begin{align}
\begin{aligned}\label{eq:micro_model_var}
\int_0^T \int_{\oeps} &\left(\partial_F W_{\vareps}(\nabla \ueps) + \partial_{\dot{F}}R_{\vareps}(\nabla \ueps,\nabla \dueps) \right) : \nabla \zeps +  \partial_G H_{\vareps}(\nabla^2\ueps) \vdots \nabla^2 \zeps dx dt
\\
&= \int_0^T \int_{\oeps} f_{\vareps} \cdot \zeps dx d t + \int_0^T \int_{\geps} g_{\vareps} \cdot \zeps d\sigma dt
\end{aligned}
\end{align}
for all $\zeps \in  L^2((0,T),W_{\geps^D}^{2,p}(\oeps))^n $.
\end{definition}

\begin{remark}
We emphasize that due to the assumptions below the integrals in $\eqref{eq:micro_model_var}$ are well-defined. In particular, we have with $\partial_{\dot{F}} R(F,\dot{F}) = 2 F \left(D(C)\dot{C}\right)$ (see  assumption \ref{ass:diss_R} below) that
\begin{align*}
\partial_{\dot{F}} R_{\vareps} (\nabla \ueps , \nabla \dueps)  \in L^2((0,T)\times \oeps)^{n\times n}
\end{align*}
and with assumption \ref{AssumptionLowerBoundH}
\begin{align*}
\partial_G H_{\vareps}(\nabla^2 \ueps) \in L^{\infty}((0,T),L^{p'}(\oeps))^{n\times n \times n}.
\end{align*}
\end{remark}

\subsection{Assumptions on the data}

The elastic energy potential $W_{\vareps}$ is given by $W_{\vareps}(x,F):= W\left(\fxe, F\right)$ with $W: \R^n \times GL^+(n)  \rightarrow [0,\infty)$ is $Y$-periodic with respect to the first variable  and fulfills:
\begin{enumerate}
[label = (W\arabic*)]
\item\label{ass:W_regularity} $W$ is continuous and  continuously differentiable in the second variable, \ie $\partial_F W$ is continuous.
\item\label{AssumptionLowerBoundW} Lower bound: For all $F\in GL^+(n)$ and $y \in \R^n$ it holds that
\begin{align*}
W(y,F) \geq c_0 \left(|F|^2 + \det(F)^{-q}\right) - C_0,
\end{align*}
for constants $C_0,c_0>0$ and $q\geq \frac{pn}{p-n}$.
\end{enumerate}

\noindent For the strain gradient energy potential $H_{\vareps}$ we assume that $H_{\vareps}(x,G):= H\left( \fxe , G\right)$ with $H: \R^n \times \R^{n\times n \times n} \rightarrow [0,\infty) $ is $Y$-periodic with respect to the first variable and fulfills:
\begin{enumerate}
[label = (H\arabic*)]
\item\label{AssumpHConvexity} $H$ is continuous. Additionally it is convex and continuously differentiable with respect to the second variable ($\partial_G H$ is continuous).
\item\label{AssumptionLowerBoundH} For every $y \in \R^n$ and $G \in \R^{n\times n\times n}$ it holds that ($p>n$)
\begin{align*}
c_0 |G|^p \le H(y,G) &\le C_0\left(1 + |G|^p\right),
\\
|\partial_G H(y,G)| &\le C_0 |G|^{p-1}.
\end{align*}
\end{enumerate}

\noindent The potential of dissipative force $\Reps$ is given by $\Reps(x,F,\dot{F}):= R\left(\fxe,F,\dot{F}\right)$, where $R: \R^n \times \R^{n\times n} \times \R^{n\times n} \rightarrow [0,\infty)$ fulfills:
\begin{enumerate}
[label = (D\arabic*)]
\item\label{ass:diss_R} For every $y \in \R^n$ and $F,\dot{F} \in \R^{n\times n}$ it holds that
\begin{align*}
R(y,F,\dot{F}) = \frac12 \dot{C} : D(y,C) \dot{C}
\end{align*}
with $C=F^{\top} F $ and $\dot{C}= \dot{F}^{\top} F + F^{\top} \dot{F}$ and $D \in C^0\left(\R^n \times \R^{n\times n}_{sym}\right)^{n\times n\times n\times n}$ is $Y$-periodic with respect to the first variable and fulfills $D_{ijkl} = D_{jikl} = D_{klij}$ for $i,j,k,l \in \{1,\ldots,n\}$.
\item\label{ass:diss_lower_bound}  For every $C,\dot{C} \in R^{n\times n}_{sym}$ and every $y \in \R^n$ it holds that
\begin{align*}
c_0|\dot{C}|^2 \le \dot{C} : D(y,C)\dot{C} \le C_0 |\dot{C}|^2
\end{align*}
for constants $c_0, C_0>0$.
\end{enumerate}
Finally, we make the following assumptions on the forces and the initial data:
\begin{enumerate}
[label = (A\arabic*)]
\item\label{ass:f_eps} We assume $f_{\vareps} \in W^{1,1}((0,T),L^2(\oeps))^n$ with 
\begin{align*}
\|f_{\vareps}\|_{W^{1,1}((0,T),L^2(\oeps))} \le C
\end{align*}
for a constant $C>0$ independent of $\vareps$.
Further, there exists $f_0 \in W^{1,1}((0,T),L^2(\Omega \times Y_s))^n$  such that $f_{\vareps}$ converges (weakly) in the two-scale sense to $f_0$ (see Section \ref{SectionTwoScaleConvergence} in the appendix for the definition of the two-scale convergence). Additionally, we assume that for every interval $I \subset (0,T)$ it holds that 
\begin{align*}
\sup_{\vareps >0}\int_I  \|\dot{f}_{\vareps}\|_{L^2(\oeps)} dt \rightarrow 0 \quad\mbox{ for } |I| \rightarrow 0.
\end{align*}

\item\label{ass:g_eps} We assume $g_{\vareps} = 0$ on $\geps^N$ and  $g_{\vareps} \in W^{1,1}((0,T),L^2(\geps))^n$ with 
\begin{align*}
\|g_{\vareps}\|_{W^{1,1}((0,T),L^2(\geps))} \le C\sqrt{\vareps}
\end{align*}
for a constant $C>0$ independent of $\vareps$.
Further, there exists $g_0 \in W^{1,1}((0,T),L^2(\Omega \times \Gamma))^n$  such that $g_{\vareps}$ converges (weakly) in the two-scale sense to $g_0$. Additionally, we assume that for every interval $ I \subset (0,T)$ it holds that 
\begin{align*}
\sup_{\vareps >0}\int_I  \|\dot{g}_{\vareps}\|_{L^2(\geps)} dt \rightarrow 0 \quad\mbox{ for } |I| \rightarrow 0.
\end{align*}

\item\label{ass:ueps_in} For the initial condition $\ueps^{in}$ we assume $\ueps^{in}\in \Yepsid^{\kappa_{in}}$ with $\kappa_{in}>0$ independent of $\vareps$ and such that (see below for the definition of $\meps$)
\begin{align*}
\meps(\ueps^{in}) \le C.
\end{align*}
Further we assume that there exists $u_0^{in} \in W^{2,p}(\Omega)^n$ such that $\chi_{\oeps} \ueps^{in}$ converges (weakly) in the two-scale sense to $u_0^{in}$.
\end{enumerate}
We also define the mechanical energy for $\veps \in \Yepsid$ by 
\begin{align*}
\meps(\ueps) &:= \int_{\oeps} W_{\vareps} (\nabla \ueps) + H_{\vareps}(\nabla^2 \ueps) dx,
\end{align*}

\begin{remark}\
\begin{enumerate}
[label = (\roman*)]
\item Lower regularity assumptions for $\partial_F W$ and $\partial_G H$ with respect to the oscillating variable are possible (for example $L^\infty$ is enough). To keep the notation in the assumptions simpler, we just assume continuity.
\item In \cite{mielke2020thermoviscoelasticity} they assumed that $W$ is twice continuously differentiable (with respect to $F$). This is necessary to obtain some kind of $\Lambda$-convexity (see the proof of \cite[Proposition 3.2]{mielke2020thermoviscoelasticity}) and necessary for to guarantee the existence for the thermoviscoelastic problem. So we can drop this assumption.

\item For the sake of simplicity we assume $g_{\vareps} = 0$ on the lateral boundary $\geps^N$. Some obvious generalizations are possible. For example we can assume $g_{\vareps} = g$ on $\geps^N$ with $g \in W^{1,1}((0,T),L^2(\Gamma^N))^n)$. In this case we obtain in the macroscopic model $\eqref{eq:macro_model}$ in inhomogeneous boundary condition on $\Gamma^N$ with $g$ instead of $0$.

\item The existence of $\kappa_{in}$ in assumption \ref{ass:ueps_in} follows directly from the uniform bound of the mechanical energy $\meps(\ueps^{in})$ and Lemma \ref{lem:Lower_bound_det}. Further, we also obtain $u_0^{in} \in Y_{\id}^{\kappa_{in}}$.
\end{enumerate}

\end{remark}

\section{Main results}
\label{sec:Main_Results}

In this section we formulate the main results of the paper. First of all, we have general multi-scale results including a Korn inequality with non-constant coefficients on perforated domains, as well as extension and two-scale compactness results for Sobolev functions of second order. Secondly, we give the existence result for the microscopic problem and the associated homogenization result for $\vareps \to 0$.

\subsection{Korn inequality and extension operator}

The following results are crucial for the derivation of the \textit{a priori} estimates for the microscopic solution and the derivation of the macroscopic model. However, they are independent of the microscopic model and therefore of interest on their own and important for other applications in the homogenization theory, in particular for problems including second order gradients.
We start with a Korn inequality with non-constant coefficients (which can depend on $\vareps$) on the perforated domain $\vareps$. For our application this inequality is crucial to control the strain rate $\nabla \dueps$ of the microscopic solution. In this case, the non-constant coefficients are given by the deformation gradient For this, we introduce a generalization of the symmetric gradient with non-constant coefficients. Let $U\subset \R^n$ open and bounded. For  $A: U \rightarrow \R^{n\times n}$ and $v\in W^{1,1}(U)^n$ we use  the notation 
\begin{align}\label{def:sym_gradient_nonconstant}
e_A(v):= \frac12\left( A\nabla v + \left(A\nabla v\right)^{\top}\right).
\end{align}
We emphasize that for $A\in \R^{n\times n}$ constant, we have $e_A(v) = e(Av)$.
\begin{theorem}\label{thm:main_thm_Korn_inequality}
Let $\mathcal{F} \subset C^0(\overline{\Omega})^{n\times n}$ be a compact subset, such that for all $A \in \mathcal{F}$ it holds that $\det A> \mu_0$ for a constant  $\mu_0>0$ independent of $\vareps$. Further, let $p \in (1,\infty)$. Then, there exists a constant $C_{\mathcal{F}}>0$ independent of $\vareps$, such that for all $\veps \in W^{1,p}_{\geps^D}(\oeps)^n$ and all $A \in \mathcal{F}$ it holds that
\begin{align*}
\|\veps\|_{W^{1,p}(\oeps)} \le C_{\mathcal{F}} \|e_A(\veps)\|_{L^p(\oeps)} = C_{\mathcal{F}} \|A\nabla \veps + (A \nabla \veps)^T \|_{L^p(\oeps)}.
\end{align*}
\end{theorem}
For the proof we refer to Section \ref{sec:Korn_inequality} for the case $\Omega \setminus \oeps$ and the (more technical) general case is treated Section \ref{sec:Korn_connected} in the appendix.
For Sobolev functions including second order derivatives we have the following extension result, which allows to treat functions on perforated domains as functions on the whole domain. This allows to generalize many results from the fixed domain $\Omega$ to the perforated domain $\oeps$. The proof can be found in Section \ref{sec:Extension_second_order}.
\begin{theorem}\label{thm:main_thm_extension_operator}
Let $p \in [1,\infty)$. There exists an extension operator $E_{\vareps}: W^{2,p}(\oeps) \rightarrow W^{2,p}(\Omega)$ such that for all $\ueps \in W^{2,p}(\oeps)$ it holds
\begin{align*}
\|E_{\vareps}\ueps\|_{L^p(\Omega)} &\le C \left( \|\ueps\|_{L^p(\oeps)} + \vareps \|\nabla \ueps \|_{L^p(\oeps)} \right),
\\
\|\nabla E_{\vareps} \ueps \|_{L^p(\Omega)} &\le C  \|\nabla \ueps \|_{L^p(\oeps)} ,
\\
\|\nabla^2 E_{\vareps} \ueps\|_{L^p(\Omega)} &\le C \|\nabla^2 \ueps\|_{L^p(\oeps)},
\end{align*}
for a constant $C>0$ independent of $\vareps$. Further, $E_{\vareps}$ also acts as an extension operator $E_{\vareps}:W^{1,s}(\oeps) \rightarrow W^{1,s}(\Omega)$ for $s \in [1,\infty)$ such that for all $\veps \in W^{1,s}(\oeps)$
\begin{align*}
\|E_{\vareps}\veps\|_{L^s(\Omega)} \le C \|\veps\|_{W^{1,s}(\oeps)},\qquad \|\nabla E_{\vareps}\veps\|_{L^s(\Omega)} \le C \|\nabla \veps \|_{L^s(\oeps)}.
\end{align*}
\end{theorem}

\subsection{Existence for the micro-model and homogenization}

We formulate the existence result for the microscopic problem together with $\vareps$-uniform \textit{a priori} estimates, which build the basis for the compactness results for the microscopic solutions and thus for the derivation of the macroscopic model.
\begin{theorem}\label{thm:main_thm_existence}
There exists a weak solution $\ueps$ of the microscopic problem $\eqref{MicroModel}$ in the sense of Definition \ref{def:weak_sol_micro_model}, which fulfills the following \textit{a priori} estimates:
\begin{align*}
\|\ueps\|_{L^{\infty}((0,T),W^{2,p}(\oeps))} +\|\ueps\|_{H^1((0,T),H^1(\oeps))} \le C,
\end{align*}
for a constant $C>0$ independent of $\vareps$. Further $\ueps \in L^{\infty}((0,T),\Yepsid^{\kappa_M})$ for a constant $\kappa_M>0$ independent of $\vareps$.
\end{theorem}
The proof can be found in Section \ref{sec:existence_a_priori_estimates}.
Next, we formulate the macroscopic model. We are looking for a function $u_0: (0,T)\times \Omega \rightarrow \R^n$ which solves the problem
\begin{align}
\begin{aligned}\label{eq:macro_model}
- \nabla \cdot \left[ \partial_F \overline{W}(\nabla u_0) + \partial_{\dot{F}}\overline{R}(\nabla u_0, \nabla \dot{u}_0) - \nabla \cdot \partial_G H_{hom} (\nabla^2 u_0) \right] &= \bar{f}_0 + \bar{g}_0 &\mbox{ in }& (0,T)\times \Omega,
\\
u_0 &= \id &\mbox{ on }& (0,T)\times \Gamma^D,
\\
\left[ \partial_F \overline{W} (\nabla u_0) + \partial_{\dot{F}} \overline{R}(\nabla u_0,\nabla \dot{u}_0)  - \nabla \cdot \partial_G H_{hom} (\nabla^2 u_0)\right]\nu
\\
 - \nabla_S \cdot \left[\partial_G H_{hom}(\nabla^2 u_0)\nu\right] &=0 &\mbox{ on }& (0,T)\times \Gamma^N,
\\
\partial_G H_{hom}(\nabla^2 u_0) : (\nu \times \nu) &= 0 &\mbox{ on }& (0,T)\times \partial \Omega,
\\
u_0(0) &= u^{in} &\mbox{ in }& \Omega.
\end{aligned}
\end{align}
Here, $\overline{W}$ and $\overline{R}$ denote the averaged values with respect to $Y_s$ of $W$ and $R$ respectively, this means we have for all $F,\dot{F} \in \R^{n\times n}$
\begin{align*}
\overline{W}(F):= \int_{Y_s} W(y,F)dy,\quad \overline{R}(F,\dot{F}) := \int_{Y_s} R(y,F,\dot{F}) dy.
\end{align*}
In a similar way we define the averaged quantities $\bar{f}_0$ and $\bar{g}_0$ for almost every $t \in (0,T)\times \Omega$ via 
\begin{align*}
\bar{f}_0(t,x):= \int_{Y_s} f_0(t,x,y)dy ,\qquad \bar{g}_0(t,x) := \int_{Y_s} g_0(t,x,y)dy.
\end{align*}
Further, the homogenized strain gradient energy potential $H_{hom}$ is given by
\begin{align}
H_{hom}(G):= \inf_{v_2 \in W_{\per}^{2,p}(Y_s)^n} \int_{Y_s} H(y, G + \nabla_y v_2) dy .
\end{align}
for $G \in \R^{n\times n \times n}$. We also introduce the spaces
\begin{align*}
\mathcal{Y}_{\id}:= \left\{ u \in W^{2,p}(\Omega)^n \, : \, u = \id \mbox{ on } \Gamma^D,\, \det(\nabla u)>0 \right\},
\end{align*}
and for $\kappa>0$
\begin{align*}
\mathcal{Y}_{\id}^{\kappa}:=\left\{ u \in \mathcal{Y}_{\id} \, : \, \det(\nabla u) \geq \kappa\right\}.
\end{align*}
\begin{definition}\label{def:weak_macro_solution} We call $u_0$ a weak solution of the macroscopic problem $\eqref{eq:macro_model}$, if
\begin{align*}
u_0 \in L^{\infty}((0,T),W^{2,p}(\Omega))^n\cap H^1((0,T),H^1(\Omega))^n,
\end{align*}
and for all $v_0 \in L^2((0,T),W^{2,p}_{\Gamma^D}(\Omega))^n$ it holds that
\begin{align}
\begin{aligned}\label{eq:macro_model_var}
\int_0^T& \int_{\Omega} \left[ \partial_F \overline{W}(\nabla u_0) + \partial_{\dot{F}} \overline{R}(\nabla u_0 , \nabla \dot{u}_0) \right] : \nabla v_0 dx dt
\\
&+ \int_0^T\int_{\Omega} \partial_G H_{hom} (\nabla^2 u_0) \vdots \nabla^2 v_0 dx dt
= \int_0^T \int_{\Omega} \left[\bar{f}_0  + \bar{g}_0\right] \cdot v_0  dx dt .
\end{aligned}
\end{align}
\end{definition}
Now, we are able to formulate the main homogenization result of the paper. For the definition of the two-scale convergence we again refer to Section \ref{SectionTwoScaleConvergence} in the appendix.
\begin{theorem}\label{thm:main_thm_conv_and_macro_model}
There exist $u_0 \in L^{\infty}((0,T),\mathcal{Y}_{\id}^{\kappa_M})\cap H^1((0,T),H^1(\Omega))^n $ for some $\kappa_M>0$ and $u_2 \in  L^{\infty}((0,T),L^p(\Omega,W^{2,p}_{\per}(Y)/\R))$, such that the sequence of weak solutions $\ueps$ from Theorem \ref{thm:main_thm_existence} fulfills 
\begin{align*}
\tueps &\rightharpoonup^{\ast} u_0 &\mbox{ weakly$^{\ast}$ in }& L^{\infty}((0,T),W^{2,p}(\Omega))^n,
\\
\tueps &\rightarrow u_0 &\mbox{ in }& C^0([0,T],C^1(\overline{\Omega}))^n,
\\
\chi_{\oeps} 
\nabla^2 \ueps &\rightwts{s,p}  \chi_{Y_s}\nabla_x^2 u_0 + \nabla_y^2 u_2 ,
\\
\tueps &\rightharpoonup u_0 &\mbox{ weakly in }& H^1((0,T),H^1(\Omega))^n,
\\
\nabla\dtueps &\rightwts{2} \nabla \dot{u}_0
\end{align*}
with $\tueps:= E_{\vareps}\ueps$ and the extension operator $E_{\vareps}$ from Theorem \ref{thm:main_thm_extension_operator}. Further, the limit function $u_0$ is a weak solution of the macroscopic model $\eqref{eq:macro_model}$.
\end{theorem}

\section{Korn inequality with non-constant coefficients in perforated domains}
\label{sec:Korn_inequality}
In this section we prove the Korn inequality from Theorem \ref{thm:main_thm_Korn_inequality}. Such a result was obtained for fixed domains (independent of $\vareps$) in \cite[Theorem 3.3]{mielke2020thermoviscoelasticity} using a Korn inequality from \cite{pompe2003korn} and a perturbation argument. In our case the perforated domain $\oeps$ is depending on $\vareps$ and the crucial point is to show that the constant in the Korn inequality can be chosen independent of $\vareps$. An essential step in our proof is to show a similar result for $A_{\vareps} \in L^{\infty}(\Omega)^{n\times n}$ constant on every micro-cell $\vareps (Y + k)$ with $k \in K_{\vareps}$. For this we construct an (global)  extension operator  $E_{A_{\vareps}}: W^{1,p}(\oeps)^n \rightarrow W^{1,p}(\Omega)^n$ for $p\in (1,\infty)$ such that
\begin{align}\label{ineq:aux_korn_constant_micro}
\left\| e_{A_{\vareps}} \left(E_{A_{\vareps}}\veps\right) \right\|_{L^p(\Omega)} \le C \left\| e_{A_{\vareps}} (\veps) \right\|_{L^p(\oeps)}
\end{align}
for a constant $C>0$ independent of $\vareps$. We refer to $\eqref{def:sym_gradient_nonconstant}$ for the definition of the symmetric gradient with non-constant coefficients $e_{A_{\vareps}}$. With inequality $\eqref{ineq:aux_korn_constant_micro}$ we reduce the problem to the case of a domain $\Omega$ independent of $\vareps$, such that we can use the results from \cite{pompe2003korn}.

\begin{remark}
In this section we only give the proof for $Y\setminus Y_s $ strictly included in $Y$, in particular $\Omega \setminus \oeps$ is disconnected. This allows to avoid some technical aspects and to focus on the main ideas of the proof. The general case is treated in Section \ref{sec:Korn_connected} in the appendix.
\end{remark}

We start with the construction of a local extension operator:
\begin{lemma}\label{lem:local_extension_A}
Let $p\in (1,\infty)$ and $A \in \R^{n\times n}$ such that $\det(A) >0$. Then there exists an extension operator $\tau_A:W^{1,p}(Y_s)^n \rightarrow W^{1,p}(Y)^n$ such that for all $v \in W^{1,p}(Y_s)^n$ it holds that
\begin{align*}
\|\tau_A(v)\|_{L^p(Y)} &\le C \left|A^{-1}\right| |A|  \|v\|_{W^{1,p}(Y_s)},
\\
\|\nabla \tau_A(v) \|_{L^p(Y)} &\le C \left|A^{-1}\right| |A| \|\nabla v\|_{L^p(Y_s)},
\\
\|e_A(\tau_A(v))\|_{L^p(Y)} &\le C \|e_A(v)\|_{L^p(Y_s)},
\end{align*}
for a constant $C>0$ independent of $A$.
\end{lemma}

\begin{proof}
In \cite[Lemma 2]{GahnJaegerTwoScaleTools} the existence of an extension operator $S:W^{1,p}(Y_s)^n \rightarrow W^{1,p}(Y)^n$ with the following properties was established:
\begin{align*}
\Vert Sv\Vert_{L^p(Y)} &\le C\left(\Vert v \Vert_{L^p(Y_s)} + \Vert \nabla v \Vert_{L^p(Y_s)}\right),
\\
\Vert\nabla S v \Vert_{L^p(Y)} &\le C \Vert \nabla v \Vert_{L^p(Y_s)},
\\
\Vert e(S v)\Vert_{L^p(Y)} &\le C \Vert e(v) \Vert_{L^p(Y_s)}.
\end{align*}
Now, we define $\tau_A$ for $v \in W^{1,p}(Y_s)^n$ by
\begin{align*}
\tau_A(v):&= A^{-1} S\left(Av\right) .
\end{align*}
Obviously, $\tau_A$ is an extension operator from $W^{1,p}(Y_s)^n$ to $W^{1,p}(Y)^n$. It remains to show the estimates. First of all, we have
\begin{align*}
\|\tau_A(v)\|_{L^p(Y)} &\le C  |A^{-1}|\|S(Av)\|_{L^p(Y)}
\le C |A^{-1}|\|Av\|_{W^{1,p}(Y_s)} .
\end{align*}
Next, we consider the gradient $\nabla \tau_A(v)$. We have
\begin{align*}
\|\nabla \tau_A(v)\|_{L^p(Y)} &= \left\|A^{-1} \nabla S(Av)\right\|_{L^p(Y)}
\le |A^{-1}|\|\nabla S(Av) \|_{L^p(Y)}
\\
&\le C |A^{-1}| |A| \|\nabla v \|_{L^p(Y_s)}.
\end{align*}
Finally, for $e_A(\tau_A(v))$ we get
\begin{align*}
\|e_A(\tau_A(v))\|_{L^p(Y)} &= \left\|e(S(Av)) \right\|_{L^p(Y)}
\\
&\le C \|e(Av)\|_{L^p(Y_s)}
= C \|e_A(v)\|_{L^p(Y_s)}.
\end{align*}
\end{proof}
Next, we construct a global extension operator on $\oeps$ with a uniform bound for $e_{A_{\vareps}}$, where $A_{\vareps}$ is a matrix valued function constant on every micro-cell $\vareps(Y + k)$ with $k \in K_{\vareps}$.

\begin{proposition}\label{TheoremGlobalExtension}
Let $\Omega \setminus \oeps$ be disconnected, \ie $Y\setminus Y_s \ssubset Y$ strictly included. Let $A_{\vareps} \in L^{\infty}(\Omega)^{n\times n}$ such that $A_{\vareps}$ is constant on every micro-cell $\vareps (Y+ k)$ and fulfills
\begin{align*}
\|A_{\vareps}\|_{L^{\infty}(\Omega)} \le M, \qquad \det A_{\vareps} \geq \mu >0
\end{align*}
for constants $\mu>0$ and $ M>0$ independent of $\vareps$. Then for $p\in (1,\infty)$ there exists an extension operator $E_{A_{\vareps}} : W^{1,p}(\oeps)^n \rightarrow W^{1,p}(\Omega)^n$, such that  for all $\veps \in W^{1,p}(\oeps)^n$ it holds that
\begin{align*}
\|E_{A_\vareps}(\veps)\|_{L^p(\Omega)} &\le C_G \frac{M^n}{\mu} \left( \|\veps\|_{L^p(\oeps)}  + \vareps \|\nabla \veps \|_{L^p(\oeps)} \right),
\\
\|\nabla E_{A_{\vareps}}(\veps) \|_{L^p(\Omega)} &\le C_G \frac{M^n}{\mu} \|\nabla \veps \|_{L^p(\oeps)} ,
\\
\|e_{A_{\vareps}}(E_{A_{\vareps}}(\veps))\|_{L^p(\Omega)} &\le C_G \|e_{A_{\vareps}} (\veps)\|_{L^p(\oeps)},
\end{align*}
for a constant $C_G>0$ independent of $\vareps$.
\end{proposition}
\begin{proof}
For $k \in K_{\vareps}$ and $\veps \in W^{1,p}(\oeps)^n$ we define for almost every $y \in Y_s$
\begin{align*}
\veps^k(y):= \veps(\vareps(y + k))\in W^{1,p}(Y_s)^n, \qquad A_{\vareps}^k:=A_{\vareps}|_{\vareps(Y + k)}\in \R^{n\times n}.
\end{align*}
Now, we define $E_{A_{\vareps}}$ for almost every $x \in \vareps (Y + k)$ by
\begin{align}\label{def:global_extension_E_Aeps}
E_{A_{\vareps}}(\veps)(x) := \tau_{A_{\vareps}^k} (\veps^k) \left(\dfrac{x}{\vareps} - k\right).
\end{align}
Obviously, $E_{A_{\vareps}}$ defines an extension operator on $W^{1,p}(\oeps)^n$. To show the $\vareps$-uniform estimates, we first notice that for almost every $x\in \Omega$ it holds that
\begin{align*}
\left| A_{\vareps}(x)^{-1} \right| \le C \frac{|A_{\vareps}(x)|^{n-1}}{\det A_{\vareps}(x)} \le C \dfrac{M^{n-1}}{\mu},
\end{align*}
and therefore $\|A_{\vareps}^{-1}\|_{L^{\infty}(\Omega)} \le C \frac{M^{n-1}}{\mu}$. Hence, we obtain with Lemma \ref{lem:local_extension_A}
\begin{align*}
\|E_{A_{\vareps}}(\veps)\|_{L^p(\Omega)}^p &= \sum_{k \in K_{\vareps}} \int_{\vareps(Y + k)} \left| \tau_{A_{\vareps}^k} (\veps^k)\left(\dfrac{x}{\vareps} - k\right)\right|^p dx
\\
&= \sum_{k \in K_{\vareps}} \vareps^n \int_Y \left| \tau_{A_{\vareps}^k} (\veps^k)\right|^p dy
\\
&\le C \sum_{k \in K_{\vareps}} |A_{\vareps}^k|^p |(A_{\vareps}^k)^{-1}|^p \vareps^n \int_{Y_s} |\veps^k|^p + |\nabla_y \veps^k|^p dy
\\
&\le C \dfrac{M^{np}}{\mu^p}  \sum_{k \in K_{\vareps}} \int_{\vareps(Y_s + k)} |\veps|^p + \vareps^p |\nabla \veps|^p dx
\\
&\le C \dfrac{M^{np}}{\mu^p} \left(\|\veps \|_{L^p(\oeps)}^p + \vareps^p \|\nabla \veps\|_{L^p(\oeps)}^p \right).
\end{align*}
Next, for the estimate of $\nabla E_{A_{\vareps}}(\veps)$ we first notice that for almost every $x \in \vareps(Y + k)$ it holds that 
\begin{align*}
\nabla E_{A_{\vareps}}(\veps)(x) = \vareps^{-1} \left( \nabla_y \tau_{A_{\vareps}^k}(\veps^k)\right) \left(\fxe - k\right).
\end{align*}
This implies 
\begin{align*}
\left\| \nabla E_{A_{\vareps}}(\veps) \right\|_{L^p(\Omega)}^p &= \sum_{k \in K_{\vareps}} \vareps^{n-p} \int_Y \left| \nabla_y \tau_{A_{\vareps}^k}(\veps^k) \right|^p dy 
\\
&\le C \sum_{k \in K_{\vareps}} \vareps^{n-p} |A_{\vareps}^k|^p |(A_{\vareps}^k)^{-1}|^p \int_{Y_s} |\nabla_y \veps^k|^p dy 
\\
&\le C \frac{M^{np}}{\mu^p} \sum_{k \in K_{\vareps}} \int_{\vareps (Y_s + k)} |\nabla \veps|^p dx
\\
&= C\frac{M^{np}}{\mu^p}  \|\nabla \veps\|_{L^p(\oeps)}^p.
\end{align*}
Finally, we consider the norm of $e_{A_{\vareps}}(E_{A_{\vareps}}(\veps))$. 
In the following we use the same notation $e_A$ (for arbitrary $A\in \R^{n\times n}$) for derivatives with respect to $x$ and $y$. More precisely, for functions defined on the microscopic (respectively macroscopic) domain $\oeps$ (respectively $\Omega$) the operator $e_A$ includes derivatives with respect to $x$, and for functions on  the reference element $Y_s$ (respectively $Y$) we have derivatives with respect to $y$.
For almost every $x \in \vareps(Y  + k)$ it holds that
\begin{align*}
e_{A_{\vareps}}\left(E_{A_{\vareps}}(\veps)\right)(x) = \dfrac{1}{\vareps} e_{A_{\vareps}^k}\left(\tau(\veps^k)\right) \left(\fxe - k\right).
\end{align*}
Using again  Lemma \ref{lem:local_extension_A} we get
\begin{align*}
\left\|e_{A_{\vareps}}\left(E_{A_{\vareps}}(\veps)\right)\right\|_{L^p(\Omega)}^p &= \sum_{k \in K_{\vareps}} \vareps^{n-p} \int_Y \left| e_{A_{\vareps}^k}\left(\tau_{A_{\vareps}^k}(\veps^k)\right)\right|^p dy
\\
&\le C \sum_{k \in K_{\vareps}} \vareps^{n-p} \int_{Y_s} \left| e_{A_{\vareps}^k}(\veps^k)\right|^p dy 
\\
&= C \sum_{k \in K_{\vareps}} \int_{\vareps (Y_s +k)} |e_{A_{\vareps}^k}(\veps)|^p dx
\\
&= C \|e_{A_{\vareps}}(\veps)\|_{L^p(\oeps)}^p.
\end{align*}
\end{proof}

Now we are able to show a $\vareps$-uniform Korn inequality for non-constant continuous coefficients on the perforated domain $\oeps$. We emphasize that this result is valid for general domains $\oeps$ from our assumptions ($\Omega \setminus \oeps$ connected or disconnected). However, here we only give the proof for  the disconnected case, see Section \ref{sec:Korn_connected} for the general setting.
\begin{theorem}\label{TheoremKornContFunct}
For every $A\in C^0(\overline{\Omega})^{n\times n}$ with $\det A \geq \mu_0 >0$ and $\mu_0$ independent of $\vareps$, there exists a constant $C_A>0$ independent of $\vareps$ such that for all $\veps \in W_{\Gamma_\vareps^D}^{1,p}(\oeps)^n$ with $p\in (1,\infty)$ it holds that
\begin{align}\label{KornInequalityContFunct}
\|\veps \|_{W^{1,p}(\oeps)} \le C_A \|e_A(\veps)\|_{L^p(\oeps)}.
\end{align}
\end{theorem}
\begin{proof}
Let $\Omega \setminus \oeps$ be disconnected (see end of Section \ref{sec:Korn_connected} for the connected case), especially we have $\Gamma^D:= \Gamma_{\vareps}^D$ is independent of $\vareps$.
There exists a sequence $A_{\vareps} \in L^{\infty}(\Omega)^{n\times n}$ constant on every micro-cell $\vareps (Y + k)$ for $k \in K_{\vareps}$, such that $A_{\vareps} \rightarrow A $ in $L^{\infty}(\Omega)^{n\times n}$. More precisely, we can define for $x \in \vareps(Y + k)$ the function $A_{\vareps}(x):= A(x_k)$ with  $x_k \in \vareps(Y + k)$ arbitrary but fixed. Hence, there exists $\vareps_1 >0$, and $M,\mu>0$, such that for all $\vareps \le \vareps_1$ it holds almost everywhere in $\Omega$ that
\begin{align*}
\|A_{\vareps}\|_{L^{\infty}(\Omega)} \le M, \qquad \det A_{\vareps} \geq \mu >0.
\end{align*}
For given $\veps \in W_{\Gamma^D}^{1,p}(\oeps)^n$ we denote the extension from Proposition \ref{TheoremGlobalExtension} by $\tveps:= E_{A_{\vareps}}(\veps)$. Hence, from the uniform estimates in Proposition \ref{TheoremGlobalExtension} we obtain with \cite[Corollary 4.1]{pompe2003korn} for a constant $C_{A,1}>0$ (independent of $\vareps$)
\begin{align}
\begin{aligned}\label{ineq:thm_korn_C0_aux}
\|\nabla \veps \|_{L^p(\oeps)} &\le \|\nabla \tveps\|_{L^p(\Omega)} \le C_{A,1} \|e_A(\tveps)\|_{L^p(\Omega)}
\\
&\le C_{A,1} \left( \|e_{A_{\vareps}}(\tveps)\|_{L^p(\Omega)} + 2 \|A_{\vareps} - A \|_{L^{\infty}(\Omega)} \|\nabla \tveps \|_{L^p(\Omega)}\right)
\\
&\le C_G C_{A,1} \|e_{A_{\vareps}}(\veps)\|_{L^p(\oeps)} + 2C_G C_{A,1} \dfrac{M^n}{\mu} \|A_{\vareps} - A\|_{L^{\infty}(\Omega)} \|\nabla \veps \|_{L^p(\oeps)}. 
\end{aligned}
\end{align}
Choosing $\vareps_2>0$ so small that for every $\vareps < \min\{\vareps_1,\vareps_2\}$ it holds that
\begin{align*}
\|A_{\vareps} - A \|_{L^{\infty}(\Omega)} \le \frac{\mu}{4 C_G C_{A,1} M^n},
\end{align*}
we get 
\begin{align*}
\|\nabla \veps \|_{L^p(\oeps)} &\le \underbrace{2 C_G C_{A,1} }_{=: C_{A,2}} \|e_{A_{\vareps}}(\veps)\|_{L^p(\oeps)}
\\
&\le C_{A,2}\|e_A(\veps)\|_{L^p(\oeps)} + 2C_{A,2} \|A_{\vareps} - A\|_{L^{\infty}(\Omega)} \|\nabla \veps\|_{L^p(\oeps)}. 
\end{align*}
Now, choosing $\vareps_3>0$ such that for every $\vareps < \min\{\vareps_1,\vareps_2,\vareps_3\}$ it holds that
\begin{align*}
\|A_{\vareps} - A\|_{L^{\infty}(\Omega)} \le \frac{1}{4C_{A,2}},
\end{align*}
we obtain with the Poincar\'e inequality, see Lemma \ref{lem:Poincare-inequality}, for a constant $C_{A,0}>0$ that
\begin{align*}
\| \veps \|_{W^{1,p}(\oeps)} \le C_{A,0} \| e_A(\veps)\|_{L^p(\oeps)}
\end{align*}
for all $\vareps < \widetilde{\vareps}_0 := \min\{\vareps_1,\vareps_2,\vareps_3\}$. Since $\vareps^{-1} \in \N$, there are only finitely many elements $\vareps\geq \widetilde{\vareps}_0 $, which we denote by $\vareps^{1}, \ldots,\vareps^l$ with $l\in \N$. Using again \cite[Corollary 4.1]{pompe2003korn} and the Poincar\'e inequality, we obtain for every $\vareps^i$ with $i \in \{1,\ldots,l\}$ for a constant $\widetilde{C}_{A,i}$ that 
\begin{align*}
\|v_{\vareps^i}\|_{W^{1,p}(\Omega_{\vareps^i})} \le \widetilde{C}_{A,i} \|e_A(v_{\vareps^i})\|_{L^p(\Omega_{\vareps^i})},
\end{align*}
for all $v_{\vareps^i} \in W^{1,p}(\Omega_{\vareps^i})$. Choosing $C_A:= \max\{C_{A,0}, \widetilde{C}_{A,1},\ldots,\widetilde{C}_{A,l}\}$ we obtain the desired result.
\end{proof}

\begin{remark}
From the proof of Theorem \ref{TheoremKornContFunct} we also obtain that inequality $\eqref{KornInequalityContFunct}$ is valid (possibly after a change of $C_A$) for all $A_{\vareps} \in L^{\infty}(\Omega)^{n\times n}$ constant on $\vareps(Y + k)$ for $k \in K_{\vareps}$, such that $A_{\vareps}$ lies in a small neighborhood of a continuous matrix-valued function $A\in C^0(\overline{\Omega})^{n\times n}$ with $det A> \mu_0>0$.
\end{remark}

For applications, where for example $A$ is given as the (transposed) deformation gradient, it is often not enough to consider the inequality $\eqref{KornInequalityContFunct}$ for a fixed element $A$, but for example for a sequence $A_{\vareps}$. In Theorem \ref{thm:main_thm_Korn_inequality} a sufficient condition is given such that the inequality is still valid uniformly with respect to $\vareps$ for all such elements $A_{\vareps}$.
\begin{proof}[Proof of Theorem \ref{thm:main_thm_Korn_inequality}]
The proof follows the same lines as the proof of \cite[Theorem 3.3]{mielke2020thermoviscoelasticity}, and we only sketch some details. Due to Theorem \ref{TheoremKornContFunct}, for given $A \in \mathcal{F}$ there exists a constant $C_A>0$ (independent of $\vareps$) such that for all $\veps \in W^{1,p}(\oeps)^n$ we have
\begin{align*}
\|\veps\|_{W^{1,p}(\oeps)} \le C_A \|e_A(\veps)\|_{L^p(\oeps)}.
\end{align*}
The mapping $A\mapsto C_A$ is constant and due to the compactness of $\mathcal{F}$ attains its minimum on $\mathcal{F}$ for some element $A^{\ast} \in \mathcal{F}$. Hence, the desired result is valid for $C_{\mathcal{F}}:= C_{A^{\ast}}>0$.
\end{proof}

\section{Extension operator for second order Sobolev spaces}
\label{sec:Extension_second_order}
In the following we construct a norm preserving extension operator for functions in $W^{2,p}(\oeps)$ for $p\in [1,\infty)$ which allows to treat functions defined on perforated domains as functions defined on the whole domain. Extension operators for perforated domains including holes (not touching each other) can be found for functions in $W^{1,p}(\oeps)$ in \cite{CioranescuSJPaulin}. The situation gets more complicated if the perforations are connected. This problem was solved in \cite{Acerbi1992} for functions in $W^{1,p}(\oeps)$. Here, we want to extend the latter results to second order Sobolev spaces, The crucial point is to construct suitable local extension operators on  reference elements. 

For the sake of simplicity and to avoid technical issues we only give the proof for strict inclusions, \ie $\Omega \setminus \oeps$ is disconnected. For the general case we refer to the methods in \cite{Acerbi1992} and also the arguments in Section \ref{sec:Korn_connected} in the appendix.

In the extension Theorem \ref{thm:main_thm_extension_operator} we need some kind of total extension operator, which means that the operator acts simultaneously on $W^{2,p}(\oeps)$ and $W^{1,s}(\oeps)$ with suitable estimates. Hence, also our local extension operator has to be total. For this in \cite{Acerbi1992} an arbitrary strong-1-extension acting on $L^p$ and $W^{1,p}$, see \cite[5.17]{AdamsSobelevSpaces2003} for the precise definition of such operators, was used for the extension of $W^{1,p}$-functions. Such operators can be obtained via a mirror argument for Lipschitz domains. However, for higher-order Sobolev spaces this mirror argument fails for the construction of strong-2-extension operators. Therefore we use here the Stein-extension operator from \cite{Stein1970}, see also \cite[Theorem 5.24]{AdamsSobelevSpaces2003}. More precisely, there exists an extension operator $S: W^{2,p}(Y_s)\rightarrow W^{2,p}(Y)$ with $p \in [1,\infty)$, such that for all $u \in W^{2,p}(Y_s)$ it holds that
\begin{align}
\begin{aligned}\label{ineq:SteinExtensionLocal}
\|Su\|_{L^p(Y)} &\le C \|u\|_{L^p(Y_s)},
\\
 \|Su\|_{W^{1,p}(Y)} &\le C \|u\|_{W^{1,p}(Y_s)},
 \\
 \|Su\|_{W^{2,p}(Y)} &\le C\|u\|_{W^{2,p}(Y_s)}.
\end{aligned}
\end{align}
We start with the construction of the local extension operator.
\begin{lemma}\label{LemmaLocalExtension}
Let $p\in [1,\infty)$. There exists an extension operator $\tau : W^{2,p}(Y_s) \rightarrow W^{2,p}(Y)$ such that for all $u \in W^{2,p}(Y_s)$ it holds that
\begin{align*}
\|\tau u\|_{L^p(Y)} &\le C \| u\|_{W^{1,p}(Y_s)},
\\
\|\nabla \tau u \|_{L^p(Y)} &\le C \|\nabla u \|_{L^p(Y_s)},
\\
\|\nabla^2 \tau u\|_{L^p(Y)} &\le C \|\nabla^2 u\|_{L^p(Y_s)}.
\end{align*}
Further, for $s\in [1,\infty)$ the operator $\tau$ also defines an extension operator $\tau: W^{1,s}(Y_s)\rightarrow W^{1,s}(Y)$ such that for all $v\in W^{1,s}(Y_s)$ it holds that
\begin{align}\label{ineq:tau_W1s}
\|\tau v\|_{L^s(Y)} \le C \|v\|_{W^{1,s}(Y_s)},\quad \|\nabla \tau v \|_{L^s(Y)} \le C \|\nabla v\|_{L^s(Y_s)}.
\end{align}
\end{lemma}
\begin{proof}
Let $u \in W^{2,p}(Y_s)$  and define 
\begin{align*}
F:= \frac{1}{|Y_s|} \int_{Y_s} \nabla u(y) dy, \quad b:= \frac{1}{|Y_s|} \int_{Y_s} u(y)- Fy dy,\quad m_u(y):= Fy + b.
\end{align*}
We define the local extension operator $\tau$ by
\begin{align*}
\tau u := S\left(u - m_u\right) + m_u.
\end{align*}
Obviously $\tau u$ is an extension of $u$. It remains to show the estimates. Since the mean value of $m_u$ is equal to the mean value of $u$, and $\nabla m_u=F$ is equal to the mean value of $\nabla u$, we obtain using Poincar\'e's inequality twice ($Y_s$ is connected) 
\begin{align*}
\|u - m_u\|_{L^p(Y_s)} \le C \|\nabla u - F \|_{L^p(Y_s)} \le C \| \nabla^2 u\|_{L^p(Y_s)},
\end{align*}
and therefore
\begin{align*}
\|u - m_u\|_{W^{1,p}(Y_s)} &\le C \|\nabla u\|_{L^p(Y_s)},
\\
\|u-m_u\|_{W^{2,p}(Y_s)} &\le C \|\nabla^2 u\|_{L^p(Y_s)}.
\end{align*}
We have
\begin{align}
\begin{aligned}\label{InequalityTauLp}
\|\tau u\|_{L^p(Y)} &\le \| S(u - m_u)\|_{L^p(Y)} + \| m_u \|_{L^p(Y)}
\\
&\le C \| u - m_u \|_{L^p(Y_s)}  + C \|u\|_{W^{1,p}(Y_s)}
\\
&\le C \|u\|_{W^{1,p}(Y_s)}.
\end{aligned}
\end{align}
Next, for the gradient of $\tau u$ we have with similar arguments
\begin{align}
\begin{aligned}\label{InequalityGradTauLp}
\|\nabla \tau u\|_{L^p(Y)} &\le \|\nabla S(u - m_u)\|_{L^p(Y)} + \|F\|_{L^p(Y)}
\\
 &\le C\|u - m_u\|_{W^{1,p}(Y_s)} + \|F\|_{L^p(Y)}
 \\
&\le C \|\nabla u\|_{L^p(Y_s)}.
\end{aligned}
\end{align}
Finally, for the second order derivatives we get
\begin{align*}
\|\nabla^2 \tau u\|_{L^p(Y_s)} &= \|\nabla^2 S(u - m_u)\|_{L^p(Y_s)} 
\\
&\le C\|u - m_u \|_{W^{2,p}(Y_s)}
\\
&\le C\|\nabla^2 u \|_{L^p(Y_s)}.
\end{align*}
Inequalities $\eqref{InequalityTauLp}$ and $\eqref{InequalityGradTauLp}$ especially show that $\tau$ also acts as an operator on $W^{1,s}(Y_s)$ (change $p$ and $s$) with the inequalities $\eqref{ineq:tau_W1s}$. This finishes the proof.
\end{proof}
Now we are able to give to proof for the extension Theorem \ref{thm:main_thm_extension_operator}:
\begin{proof}[Proof of Theorem \ref{thm:main_thm_extension_operator}]
We only sketch the proof for the case of strict inclusions, since the general case follows by the more technical arguments from \cite{Acerbi1992}, see also Section \ref{sec:Korn_connected} in the appendix. Let $\ueps \in W^{2,p}(\oeps)$. We define for almost every $x \in \vareps (Y + k)$ for $k\in K_{\vareps}$ the extension (with $\ueps^k:= \ueps(\epsilon(k + \cdot))$)
\begin{align*}
E_{\vareps}\ueps(x):= \tau \ueps^k \left(\fxe - k\right)
\end{align*}
with the local extension operator from Lemma \ref{LemmaLocalExtension}.
An elemental calculation shows 
\begin{align*}
\nabla E_{\vareps} \ueps = \vareps^{-1} [\nabla \tau \ueps^k]\left(\fxe - k\right), \qquad \nabla^2 E_{\vareps} \ueps = \vareps^{-2} [\nabla^2 \tau \ueps^k] \left(\fxe - k\right).
\end{align*}
With a straightforward calculation we get with the estimates from Lemma \ref{LemmaLocalExtension}
\begin{align*}
\|E_{\vareps} \ueps\|_{L^p(\Omega)}^p &= \sum_{k \in K_{\vareps}} \int_{\vareps(Y + k)} \left|\tau \ueps^k \left( \fxe  - k\right)\right|^p dx 
\\
&=\sum_{k \in K_{\vareps}} \vareps^n \int_Y |\tau \ueps^k|^p dy
\\
&\le C \sum_{k \in K_{\vareps}} \|\ueps^k\|^p_{W^{1,p}(Y_s)}
\\
&\le C\left( \|\ueps\|^p_{L^p(\oeps)} + \vareps^p \|\nabla \ueps\|^p_{L^p(\oeps)} \right).
\end{align*}
The estimates for the first and second gradient of $E_{\vareps} \ueps$ follow by similar arguments. For $\veps \in W^{1,s}(\oeps)$ the estimates follow by the same arguments.
\end{proof}

\section{Existence of a microscopic solution and \textit{a priori} estimates}
\label{sec:existence_a_priori_estimates}
The aim of this section is to show existence of a weak microscopic solution of $\eqref{MicroModel}$ together with $\vareps$-uniform \textit{a priori} estimates. For the existence we follow the approach of \cite{badal2023nonlinear} and use similar notations, where the solution is obtained via a time-discretization. We emphasize that the existence of a weak microscopic solution $\ueps$ for fixed $\vareps$ is already guaranteed by the results in \cite{badal2023nonlinear} and \cite{mielke2020thermoviscoelasticity}. The crucial point in this section is to obtain the $\vareps$-uniform \textit{a priori} bounds and $\vareps$-independent lower bound for the determinant of the deformation gradient. However, for the sake of completeness and since the proof of the pure viscoelastic problem is much simpler compared to the thermovisoelastic problem considered in  \cite{badal2023nonlinear} and \cite{mielke2020thermoviscoelasticity}, we also give a detailed existence proof. We emphasize that in the following we assume $p>n$.

\subsection{The time-discretized microscopic model}
 We fix a time step $\tau \in (0,1)$ and assume that $\tau $ evenly divides the time interval $[0,T]$. Hence, there exists $N_{\tau} \in \N$, such that $\frac{T}{\tau} = N_{\tau}$. Further, for an arbitrary sequence $(a_l)_{l\in \N_0}$ we use the following notation for discrete differences:
\begin{align*}
\dtau a_l := \frac{a_l - a_{l-1}}{\tau}, \qquad l\in \N.
\end{align*}
As initial step we choose 
\begin{align*}
\uepstau{0}:= \ueps^{in}.
\end{align*}
Now, let $k \in \{1,\ldots, N_{\tau}\}$ and  $\uepstau{k-1} \in \Yepsid$ given. We are looking for a solution $\uepstau{k}$ of the minimization problem
\begin{align}\label{Min_Problem_discr}
\min_{\veps \in \Yepsid} \left\{ \Meps(\veps) + \frac{1}{\tau} \Reps\left(\uepstau{k-1} , \veps - \uepstau{k-1}\right) - \langle \lepstau{k},\veps \rangle \right\},
\end{align}
with
\begin{align*}
\langle \lepstau{k},\veps\rangle &:= \int_{\oeps} \fepstau{k} \cdot \veps dx + \int_{\geps} \gepstau{k} \cdot \veps d\sigma,
\end{align*}
where $\fepstau{k}$ and $\gepstau{k}$ are given by
\begin{align*}
\fepstau{k} &:= \frac{1}{\tau} \int_{(k-1)\tau}^{k\tau} f_{\vareps} (s) ds,
\\
\gepstau{k} &:= \frac{1}{\tau} \int_{(k-1)\tau}^{k\tau} g_{\vareps}(s)ds.
\end{align*}

In a first step, we show that the discrete minimization problem $\eqref{Min_Problem_discr}$ has a solution in $\Yepsid$. In Proposition \ref{prop:minimizer_discrete} below we give no explicit bound with respect to $\vareps$ and $\tau$.

\begin{proposition}\label{prop:minimizer_discrete}
Let $\uepstau{k-1} \in \Yepsid^{\kappa_{\vareps}}$  with $k \in \{1,\ldots, N_{\tau}\}$ and $\Meps(\uepstau{k-1}) \le M_{\vareps}$, for constants $M_{\vareps} , \kappa_{\vareps}>0$ which might depend on $\vareps$. Then the  minimization problem $\eqref{Min_Problem_discr}$ admits a solution $\uepstau{k} \in \Yepsid$.
\end{proposition}

\begin{remark}
In the formulation of Proposition \ref{prop:minimizer_discrete} we allowed the constants $M_{\vareps}$ and $\kappa_{\vareps}$ to depend on $\vareps$ (and also on $\tau$). This causes no problems here, since this Proposition is a pure existence result without giving additional \textit{a priori} bounds. This will be done later in Lemma \ref{lem:uniform_bound_dis_energy} below, where we show that under our assumptions on $\ueps^{in}$, $f_{\vareps}$, and $g_{\vareps}$ we can choose $M_{\vareps}$ and $\kappa_{\vareps}$ independent of $\vareps$ and $\tau$ (for $\tau$ small enough).
\end{remark}

\begin{proof}
We follow the proof of \cite[Proposition 4.1]{mielke2020thermoviscoelasticity} and \cite[Proposition 3.5]{badal2023nonlinear}. Since we have no coupling to thermal transport, the proof simplifies, and for the sake of completeness we sketch the crucial steps.
In the following $C_{\vareps},c_{\vareps}>0$ are constants which might depend on $\vareps$. From our assumptions we have $\det(\nabla \uepstau{k-1} )\geq \kappa_{\vareps}>0$ and the embedding $\Yepsid^{\kappa_{\vareps}} \hookrightarrow C^{1,\gamma}(\overline{\oeps})^n$ with $\gamma \in \left(0,1 -\frac{n}{p}\right)$ for $p>n$ implies $\nabla \uepstau{k-1} \in C^0(\overline{\oeps})^{n\times n}$. Hence, we can apply the Korn inequality from \cite[Corollary 4.1]{pompe2003korn} to obtain for a constant $c_{\vareps}>0$ and all $\veps \in H^1_{\geps^D}(\oeps)^n$
\begin{align}\label{aux_Korn_inequality}
c_{\vareps}\|\veps \|_{H^1(\oeps)} \le \left\| \nabla \veps^{\top} \nabla \uepstau{l} + (\nabla \uepstau{l})^{\top} \nabla \veps \right\|_{L^2(\oeps)}.
\end{align}

Let $(\veps^n)_{n\in \N} \subset \Yepsid$ for $n\in \N$ be a minimizing sequence for $\eqref{Min_Problem_discr}$. Hence, for every $\delta>0$ we obtain for $n$ large enough
\begin{align}
\begin{aligned}\label{eq:aux_dis_inequ}
\Meps(\veps^n) + \frac{1}{\tau} &\Reps\left(\uepstau{k-1}, \veps^n - \uepstau{k-1} \right) - \langle \lepstau{k} , \veps^n\rangle
\\
&\le \Meps(\uepstau{k-1}) + \frac{1}{\tau} \Reps(\uepstau{k-1},0) - \langle \lepstau{k}, \uepstau{k-1}\rangle + \delta. 
\end{aligned}
\end{align}
Without loss of generality we put $\delta  = 0$. Now, with $\eqref{aux_Korn_inequality}$  (remember $\uepstau{k-1} - \veps^n =0$ on $\geps^D$) we get from the assumption \ref{ass:diss_lower_bound}
\begin{align}
\begin{aligned}\label{ineq:diss_korn_exist_dis_min}
\frac{1}{\tau}\Reps&\left(\uepstau{k-1},\veps^n - \uepstau{k-1}\right) \\
&\geq \frac{c_0}{\tau} \left\|\nabla \left(\veps^n - \uepstau{k-1}\right)^{\top} \nabla \uepstau{k-1} + (\nabla \uepstau{k-1})^{\top} \nabla \left(\veps^n - \uepstau{k-1}\right)\right\|^2_{L^2(\oeps)}
\\
&\ge \frac{c_{\vareps}}{\tau} \left\|\nabla \veps^n - \nabla \uepstau{k-1}\right\|^2_{L^2(\oeps)}.
\end{aligned}
\end{align}
To estimate the terms including $\lepstau{k}$, we use the Poincar\'e inequality and the trace inequality, to obtain for arbitrary $\theta >0$ and a constant $C_{\vareps}(\theta)>0$
\begin{align*}
\left|\langle \lepstau{k}, \uepstau{k-1} - \veps^n \rangle\right| \le& \|\fepstau{k}\|_{L^2(\oeps)} \|\uepstau{k-1} - \veps^n \|_{L^2(\oeps)} 
\\
&\hspace{4em}+ \|\gepstau{k}\|_{L^2(\geps)} \|\uepstau{k-1} - \veps^n\|_{L^2(\geps)}
\\
\le& \tau C_{\vareps}(\theta) \|\fepstau{k}\|_{L^2(\oeps)}^2 + \frac{\theta}{\tau} \|\nabla \uepstau{k-1} - \nabla \veps^n \|_{L^2(\oeps)}
\\
&+  \tau C_{\vareps}(\theta)\|\gepstau{k}\|_{L^2(\geps)} + \frac{\theta}{\tau} 
 \|\nabla \uepstau{k-1} - \nabla \veps^n\|_{L^2(\oeps)}^2.
\end{align*}
Altogether, for $\theta >0$ small enough we obtain from $\eqref{eq:aux_dis_inequ}$
\begin{align}
\begin{aligned}\label{eq:aux_energy_discrete}
\Meps(\veps^n) + \frac{c_{\vareps}}{\tau} & \|\nabla \uepstau{k-1} - \nabla \veps^n\|_{L^2(\oeps)}^2 
\\
&\le \Meps(\uepstau{k-1}) + C_{\vareps} \tau \left(\left\|\gepstau{k}\right\|^2_{L^2(\geps)} + \left\|\fepstau{k}\right\|^2_{L^2(\oeps)}\right).
\end{aligned}
\end{align}
Since the right-hand side is independent of $n$, we obtain with the assumptions \ref{AssumptionLowerBoundW}, \ref{AssumptionLowerBoundH}, \ref{ass:f_eps}, and \ref{ass:g_eps} that $\veps^n$ fulfills
\begin{align*}
\|\veps^n\|_{H^1(\oeps)} + \|\nabla^2 \veps^n\|_{L^p(\oeps)} \le C_{\vareps}
\end{align*}
for a constant $C_{\vareps}>0$ independent of $n$. From the Poincar\'e inequality in Lemma \ref{lem:Poincare-inequality} in the appendix we also obtain
\begin{align*}
\|\veps^n\|_{W^{2,p}(\oeps)} \le C_{\vareps}.
\end{align*}
Since $W_{\vareps}$ is continuous and $H_{\vareps} $ is convex (see \ref{ass:W_regularity} and \ref{AssumpHConvexity} ),  and the embedding $W^{2,p}(\oeps)^n$ into $W^{1,\infty}(\oeps)^n$ is compact since $p>n$, we obtain that $\Meps$ is lower semicontinuous on $W^{2,p}(\oeps)^n$. This guarantees the existence of a minimizer $\uepstau{k}$ with $\veps^n \rightharpoonup \uepstau{k}$ weakly in $W^{2,p}(\oeps)^n$. From the calculations above we also see that a minimizer has bounded energy $\meps(\uepstau{k}) \le C_{\vareps}$ and \cite[Theorem 3.1]{mielke2020thermoviscoelasticity} implies the existence of $\kappa_{\vareps}>0$ such that $\det(\nabla \uepstau{k}) \geq \kappa_{\vareps}$ in $\oeps$. This finishes the proof.
\end{proof}

\begin{remark}\label{rem:Euler_Lagrange_discrete}
Following the ideas in \cite[Proposition 3.2 and 4.1]{mielke2020thermoviscoelasticity} we obtain that $\uepstau{k}$ from Proposition \ref{prop:minimizer_discrete} fulfills for all $\zeps \in W^{2,p}_{\geps^D}(\oeps)^n$ the equation
\begin{align*}
\int_{\oeps}\bigg\{ \left[ \partial_F W_{\vareps} (\nabla \uepstau{k}) + \partial_{\dot{F}} R_{\vareps}(\nabla \uepstau{k-1},\delta_{\tau} \nabla \uepstau{k})\right]& : \nabla \zeps 
\\
& \hspace{-5em}+ \partial_G H_{\vareps} (\nabla^2 \ueps) \vdots \nabla^2 \zeps \bigg\} dx - \langle \lepstau{k},\zeps\rangle = 0.
\end{align*}
Here it is crucial that the function $\partial_{\dot{F}} R_{\vareps}$ is linear in the $\dot{F}$-variable.
\end{remark}

Proposition \ref{prop:minimizer_discrete} implies the existence of a sequence $\{\uepstau{k}\}_{k=0}^{N_{\tau}}$ with $N_{\tau} = \frac{T}{\tau}$ and $\uepstau{0} = \ueps^{in}$ of minimizers of $\eqref{Min_Problem_discr}$. However, we have no uniform bound of the energies $\meps(\uepstau{k})$ with respect to $\vareps$ and $\tau$. The reason is the constant $c_{\vareps}$ in $\eqref{ineq:diss_korn_exist_dis_min}$. In the following we will refine the estimates from the proof above, where in particular we use the higher regularity for the data $f_{\vareps}$ and $g_{\vareps}$. This gives a uniform bound for the mechanical energies $\meps(\uepstau{k})$, which allows to control the determinant $\det (\nabla \uepstau{k})$ from below uniformly with respect to $\vareps$ and $\tau$, see Lemma \ref{lem:Lower_bound_det} below.

\subsection{$\vareps$-uniform lower bound for the determinant of the deformation gradient for bounded energies}

We show an $\vareps$-uniform positive lower bound for functions with bounded energy on the perforated domain $\oeps$. More precisely, we define 
\begin{align*}
\mathcal{F}_{\vareps}^M:=\left\{ \veps \in W^{2,p}(\oeps)^n\, : \, \Meps(\veps) \le M\right\},
\end{align*}
and show that the determinant of the deformation gradient of such functions are bounded away from zero with a constant independent of $\vareps$. Such a result for fixed domains (not depending on $\vareps$) was shown in \cite[Theorem 3.1]{mielke2020thermoviscoelasticity} with similar arguments as in \cite{healey2009injective}. We start with a Lemma giving a lower bound depending on $\vareps$ by repeating the proof of \cite{mielke2020thermoviscoelasticity}.

\begin{lemma}\label{lem:lower_bound_det_eps}
For every $\veps \in \mathcal{F}_{\vareps}^M$  for a constant $M>0$ independent of $\vareps$  and $p>n$ it holds that
\begin{align*}
\det(\nabla \veps) \geq \vareps^{\frac{n}{q}} \tilde{\kappa}_M\quad \mbox{ in } \oeps,
\end{align*}
for a constant $ \tilde{\kappa}_M>0$ independent of $\vareps$.
\end{lemma}
\begin{proof}
The proof follows by similar arguments as in \cite[Theorem 3.1]{mielke2020thermoviscoelasticity} and we  use a similar notation, where here we have to analyze in detail the dependence on the parameter $\vareps$.
Let $\veps \in \mathcal{F}_{\vareps}^M$. From Lemma \ref{lem:Poincare-inequality} and \ref{cor:embedding_Sobolev_hoelder} we obtain for a constant $C>0$ independent of $\vareps$ and $\lambda = 1 -\frac{n}{p}$ for $p>n$ that
\begin{align*}
\|\veps\|_{C^{1,\lambda}(\overline{\oeps})} \le C \|\veps\|_{W^{2,p}(\oeps)} \le \meps(\veps) \le C.
\end{align*}
We define $\delta_{\vareps}:= \det(\nabla \veps) \in C^{0,\lambda}(\overline{\oeps})$.
 Since $\oeps$ is a Lipschitz domain, it also fulfills the uniform cone condition, as well as the reference element $Y_s$. Especially, we can choose cones with height $\vareps$. More precisely we have cones for every $x \in \oeps$
\begin{align*}
C_{\vareps}(x) = \left\{x + z \, : \, 0 < |z| < \vareps r_{\ast}, \, \frac{z}{|z|} \in A \right\},
\end{align*}
for $r_{\ast}>0$  and $A \subset \mathbb{S}^{n-1}$ with $\int_A d\omega = \alpha_{\ast}$. Using the H\"older continuity of $\delta_{\vareps}$ we get for a constant $L>0$ independent of $\vareps$ for all $x,y \in \overline{\Omega}$ that 
\begin{align*}
\delta_{\vareps}(y) \le \delta_{\vareps}(x) + L |x - y|^{\lambda}. 
\end{align*}
Now, we have for a constant $C_0>0$ 
\begin{align*}
C &\geq \int_{\oeps} \frac{1}{\delta_{\vareps}(y)^q} dy \geq \int_{y \in C_{\vareps}(x)} \frac{1}{\left(\delta_{\vareps}(x) + L | x -y |^{\lambda} \right)^q } dy 
\\
&\geq \int_{y \in C_{\vareps}(x)} \frac{1}{\left(\delta_{\vareps}(x) + L | x -y |^{\frac{n}{q}} \right)^q } dy  = \int_{\omega \in A} \int_{r=0}^{\vareps r_{\ast}} \frac{r^{n-1}dr}{\left(\delta_{\vareps}(x) + L|x - y|^{\frac{n}{q}}\right)^q} d\omega
\\
&\geq  \frac{\alpha_{\ast}}{C_0} \int_{r=0}^{\vareps r_{\ast}} \frac{r^{n-1} dr }{\left(\delta_{\vareps}(x)^q + r^n\right)} = \frac{\alpha_{\ast}}{nC_0} \log\left(1 + \frac{(\vareps r_{\ast})^n}{\delta_{\vareps}(x)^q}\right).
\end{align*}
From this we obtain for every $x \in \overline{\oeps}$
\begin{align*}
\delta_{\vareps} \geq \vareps^{\frac{n}{q}} r_{\ast}^{\frac{n}{q}}\left(\exp\left(\frac{C C_0 n}{\alpha_{\ast}}\right) - 1 \right)^{-\frac{1}{q}} =: \vareps^{\frac{n}{q}} c_0.
\end{align*}
\end{proof}

Now, we show that in fact the lower bound can be chosen independently of $\vareps$ by using a contradiction argument.

\begin{lemma}\label{lem:Lower_bound_det}
Let $M>0$ independent of $\vareps$. There exists $\kappa_M>0$ such that for all $\veps \in \mathcal{F}_{\vareps}^M$ ($p>n$) it holds everywhere in $\oeps$
\begin{align}\label{LowerBound_Determinant}
\det(\nabla \veps) \geq \kappa_M.
\end{align}
In particular for $\vareps$ small enough and with the extension operator $E_{\vareps} $ from Theorem \ref{thm:main_thm_extension_operator}, for all $\veps \in \mathcal{F}_{\vareps}^M$ it holds everywhere in $\Omega$
\begin{align*}
\det(\nabla E_{\vareps}\veps) \geq \frac{\kappa_M}{2}.
\end{align*}
\end{lemma}

\begin{remark}
The result of Lemma \ref{lem:Lower_bound_det} is still valid, if we replace in the definition of $\mathcal{F}_{\vareps}^M$ the energy $\Meps(\veps)$ by
\begin{align*}
\|\veps\|_{W^{2,p}(\oeps)} + \|(\det(\nabla \veps)^{-1}\|_{L^q(\oeps)} .
\end{align*}
\end{remark}

\begin{proof}
We argue by contradiction and assume that $\eqref{LowerBound_Determinant}$ is not valid. Hence, there exists a subsequence $(\vareps_k)_{k\in \N} \subset \{\vareps\}_{\vareps>0}$ and a sequence $(x_{\vareps_k})_{k \in \N} \subset\oeps$, such that
\begin{align}\label{lem:pos_det_Ass}
0<\det(\nabla v_{\vareps_k}(x_{\vareps_k})) < \frac{1}{k}.
\end{align}
Especially, it holds that $\vareps_k \to 0$ for $k\to \infty$. This is a consequence of Lemma \ref{lem:lower_bound_det_eps}. To simplify the notation we shortly write $\vareps = \vareps_k$. Since $\Omega$ is bounded, we obtain the existence of  $x \in \overline{\Omega}$, such that up to a subsequence $x_{\vareps} \to x$.
We define $\tveps:=E_{\vareps} \veps$. The properties of the extension operator from Theorem \ref{thm:main_thm_extension_operator} and the bound for $\veps \in \mathcal{F}_{\vareps}^M$ in the assumption imply (see also Lemma \ref{lem:Poincare-inequality})
\begin{align*}
\|\tveps\|_{W^{2,p}(\Omega)} \le C \|\veps\|_{W^{2,p}(\oeps)} \le C.
\end{align*}
Especially, there exists $v_0 \in W^{2,p}(\Omega)^n$ such that up to a subsequence
\begin{align*}
\tveps &\rightharpoonup v_0 &\mbox{ in }& W^{2,p}(\Omega)^n,
\\
\tveps &\rightarrow v_0 &\mbox{ in }& C^{1,\gamma}(\overline{\Omega})^n
\end{align*}
with  $0 < \gamma <1 - \frac{n}{p}$, where we used the compact embedding of $W^{2,p}(\Omega)^n$ into $C^{1,\gamma}(\overline{\Omega})^n$ for $n<p$. We immediately obtain 
\begin{align*}
\det(\nabla \tveps) \rightarrow \det(\nabla u_0) \quad\mbox{ in } C^0(\overline{\Omega}).
\end{align*}
Now, Lemma \ref{lem:unfolding_c0_linfty} implies 
\begin{align*}
\teps (\det(\nabla \tveps)) = \det(\teps(\nabla \tveps)) \rightarrow \det (\nabla u_0) \quad\mbox{in } L^{\infty}(\Omega \times Y)^{n\times n},
\end{align*}
and therefore $\det(\teps(\nabla \veps)) \rightarrow \det(\nabla u_0) $ in $L^{\infty}(\Omega \times Y_s)^{n\times n}$. 
Let us assume $\det(\nabla u_0) = 0$ on a set $A\subset \Omega$ with positive measure, i.e., $\det(\teps(\nabla \veps)) \rightarrow 0$ almost everywhere in $A\times Y_s$. Fatou's lemma and $\veps \in \mathcal{F}_{\vareps}^M$ imply (with the assumption \ref{AssumptionLowerBoundW})
\begin{align*}
C \geq \liminf_{\vareps \to 0} \int_{\oeps}\frac{1}{(\det(\nabla \veps)^q} dx &= \liminf_{\vareps \to 0} \int_{\Omega}\int_{Y_s}\frac{1}{(\det(\teps(\nabla \veps))^q} dy dx
\\
&\geq  \int_A \int_{Y_s} \liminf_{\vareps \to 0} \frac{1}{(\det(\teps(\nabla \veps))^q}dy dx
\\
&= +\infty.
\end{align*}
This contradicts the boundedness of the energy and we obtain $\det(\nabla u_0) >0$ almost everywhere in $\Omega$. Therefore, the function $\eta_0:= (\det(\nabla u_0))^{-1}$ is well-defined almost everywhere in $\Omega$ and for $\eta_{\vareps}:= (\det(\nabla \veps))^{-1}$ it holds that up to a subsequence
\begin{align*}
\teps(\eta_{\vareps})\rightarrow \eta_0 \quad \mbox{a.e. in } \Omega \times Y_s.
\end{align*}
Further, we have using again $\veps \in \mathcal{F}_{\vareps}^M$
\begin{align*}
\|\teps(\eta_{\vareps})\|_{L^q(\Omega \times Y_s)} = \|\eta_{\vareps}\|_{L^q(\oeps)} \le C.
\end{align*}
From \cite[Theorem 13.44]{HewittStromberg} we obtain the existence of a subsequence such that
\begin{align*}
(\det(\teps(\nabla \veps))^{-1} \rightharpoonup (\det(\nabla u_0))^{-1} \quad\mbox{ in } L^q(\Omega \times Y_s).
\end{align*}
Then the lower semicontinuity of the norm with respect to the weak convergence implies 
\begin{align*}
\left\|(\det(\nabla u_0))^{-1}\right\|_{L^q(\Omega)} \leq \liminf_{\vareps \to 0} \|\teps(\eta_{\vareps})\|_{L^q(\Omega \times Y_s)} \le C.
\end{align*}
Hence, we have 
\begin{align*}
\|u_0\|_{W^{2,p}(\Omega)} +  \left\|(\det(\nabla u_0))^{-1}\right\|_{L^q(\Omega)} \le C,
\end{align*}
and \cite[Theorem 3.1]{mielke2020thermoviscoelasticity} implies the existence of $\kappa>0$ independent of $\vareps$ such that
\begin{align*}
\det(\nabla u_0) \geq \kappa.
\end{align*}
Hence, from $\eqref{lem:pos_det_Ass}$ and the convergence of $\det(\nabla \tveps)$ in $C^0(\overline{\Omega})$, as well as the convergence $x_{\vareps} \to x$, we obtain
\begin{align*}
0 = \lim_{\vareps \to 0}  \det\left(\nabla \tveps(x_{\vareps})\right) = \det(\nabla v_0(x)) > \kappa,
\end{align*}
leading to a contradiction. 

Now, we show that also the determinant of $\tveps$ is uniformly bounded from below. From $\tveps \in C^{1,\gamma}(\overline{\Omega})$ we obtain for every $x \in \vareps(k + Y)$ for $k \in K_{\vareps}$ and fixed $x_k \in \vareps(k + Y_s)$ that
\begin{align*}
|\det(\nabla \tveps (x)) - \underbrace{\det(\nabla \tveps(x_k))}_{= \det(\nabla \veps(x_k)) } | \le C|x - x_k| \le C\vareps.
\end{align*}
Since $\det(\nabla \veps(x_k))\geq \kappa_M$ we obtain the desired result for $\vareps $ small enough.
\end{proof}

\subsection{Uniform bounds for the discrete mechanical energy and velocity}

We show a uniform bound with respect to $\vareps$ and $\delta$ for the discrete mechanical energies $\meps(\uepstau{k})$ and the discrete velocities $\dtau \nabla \uepstau{k}$.
We start with a simple estimate for functions in $H_{\geps^D}^1(\oeps)$:

\begin{lemma}\label{lem:bound_veps_H1}
There exists a constant $C>0$ independent of $\vareps$, such that for all $\veps \in H^1_{\geps^D}(\oeps)$ it holds that
\begin{align*}
\|\veps \|_{H^1(\oeps)} \le C \left( 1 + \|\nabla \veps\|_{L^2(\oeps)}\right).
\end{align*}
\end{lemma}
\begin{proof}
The result follows from the Poincar\'e inequality in Lemma \ref{lem:Poincare-inequality}. In fact, we have
\begin{align*}
\|\veps\|_{H^1(\oeps)} &\le \| \veps - \id\|_{H^1(\oeps)} + \|\id\|_{H^1(\oeps)}
\\
&\le C \|\nabla \veps - E_n\|_{L^2(\oeps)} + \|\id\|_{H^1(\oeps)}.
\end{align*}
The result follows from the triangle inequality and the fact that $\|\id\|_{H^1(\oeps)} \le C |\oeps|^{\frac12}$.
\end{proof}
In the next Lemma we show uniform \textit{a priori} estimates (with respect to $\vareps$ and $\tau$) for the energies of $\uepstau{k}$ and the discrete velocities $\dtau \nabla \uepstau{k}$. These \textit{a priori} estimates are a crucial step to show existence for the continuous model, but in particular also for the homogenization $\vareps \to 0$. First of all, let us define the following interpolations: Let $\tau \in (0,1)$ with $N_{\tau}= \frac{T}{\tau} \in \N$ and $\uepstau{0}, \ldots, \uepstau{N_{\tau}}$ the sequence of minimizers from Proposition \ref{prop:minimizer_discrete}. Now, for $k\in \{1,\ldots,N_{\tau})$ and $t \in [(k-1)\tau,k\tau)$ we define 
\begin{align*}
\buepstau(t):= \uepstau{k},\quad \lbuepstau(t) := \uepstau{k-1}, \quad \huepstau(t) := \frac{k\tau - t}{\tau} \uepstau{k-1} + \frac{t - (k-1)\tau}{\tau} \uepstau{k}.
\end{align*}
We emphasize that we have $\dhuepstau(t) = \dtau \uepstau{k}$.
\begin{lemma}\label{lem:uniform_bound_dis_energy}
Let $\tau \in (0,1)$ with $N_{\tau}:= \frac{T}{\tau} \in \N$ and  $\{\uepstau{k}\}_{k=1}^{N_{\tau}}$ the sequence of minimizers of $\eqref{Min_Problem_discr}$ from Proposition \ref{prop:minimizer_discrete} There exists $\tau_0>0$ such that for all $\tau \in (0,\tau_0)$ and $k=0,\ldots,N_{\tau}$ it holds for a constant $M>0$ independent of $\vareps$ and $\tau$
\begin{subequations}
\begin{align}
\label{lem:uniform_bound_dis_energy_in_meps}\meps(\uepstau{k}) &\le M,
\\
\label{lem:uniform_bound_dis_energy_in_dtau_ueps}\sum_{l=1}^{k} \tau \int_{\oeps} \left|\dtau \nabla \uepstau{l}\right|^2 dx &\le M.
\end{align}
\end{subequations}
Moreover, there exists a constant $\kappa_M>0$ independent of $\vareps$ such that $\uepstau{k} \in \Yepsid^{\kappa_M}$.
\end{lemma}
\begin{proof}
Since $\uepstau{k}$ is a minimizer of $\eqref{Min_Problem_discr}$, we obtain
\begin{align}\label{lem:bound_disc_energies_aux_ineq}
\meps(\uepstau{k}) - \meps(\uepstau{k-1}) + \frac{1}{\tau} \Reps\left(\uepstau{k-1},\uepstau{k} - \uepstau{k-1}\right) \le \tau \langle \lepstau{k} ,\dtau \uepstau{k}\rangle.
\end{align}
Using the piecewise continuous interpolation $\huepstau$, we obtain for the right-hand side above
\begin{align*}
\tau\langle \lepstau{k}, \dtau \uepstau{k} \rangle = \int_{(k-1)\tau}^{k\tau} \left[ \int_{\oeps} f_{\vareps}  \cdot \dhuepstau dx + \int_{\geps} g_{\vareps} \cdot \dhuepstau d\sigma \right] dt=: A_{\vareps} + B_{\vareps}.
\end{align*}
We only estimate in detail the second term $B_{\vareps}$, since $A_{\vareps}$ can be treated in a similar way. Using integration by parts in time we obtain
\begin{align*}
B_{\vareps} =-\int_{(k-1)\tau}^{k\tau} \int_{\geps} \dot{g}_{\vareps} \cdot \huepstau d\sigma dt + \int_{\geps} g_{\vareps} (k\tau) \cdot \uepstau{k} d\sigma - \int_{\geps} g_{\vareps}((k-1)\tau) \cdot \uepstau{k-1} d\sigma.
\end{align*}
To estimate the first term, we use the following bound for $\huepstau$ obtained from the scaled trace inequality in Lemma \ref{Lem:TraceInequality} and Lemma \ref{lem:bound_veps_H1}: For all $t \in ((k-1)\tau,k\tau)$ it holds for a constant $C>0$ independent of $\vareps$ and $\tau$ (for $\vareps <1$)
\begin{align*}
\sqrt{\vareps} \|\huepstau(t)\|_{L^2(\geps)} &\le C \left( \|\huepstau(t)\|_{L^2(\oeps)} + \vareps \|\nabla \huepstau(t)\|_{L^2(\oeps)} \right)
\\
&\le C \left( \|\uepstau{k}\|_{H^1(\oeps)}  +\|\uepstau{k-1}\|_{H^1(\oeps)}  \right)
\\
&\le C \left(1 + \|\nabla \uepstau{k}\|_{L^2(\oeps)} + \|\nabla \uepstau{k-1}\|_{L^2(\oeps)} \right)
\\
&\le C \left( 1 + \|\nabla \uepstau{k}\|_{L^2(\oeps)}^2 + \|\nabla \uepstau{k-1}\|_{L^2(\oeps)}^2 \right)
\\
&\le C \left(1 + \meps(\uepstau{k}) + \meps(\uepstau{k-1})\right).
\end{align*}
Hence, we obtain 
\begin{align*}
\left|\int_{(k-1)\tau}^{k\tau} \int_{\geps} \dot{g}_{\vareps} \cdot \huepstau d\sigma dt \right| \le C\left(1 + \meps(\uepstau{k}) + \meps(\uepstau{k-1})\right) \cdot \frac{1}{\sqrt{\vareps}} \int_{(k-1)\tau}^{k\tau} \|\dot{g}_{\vareps} \|_{L^2(\geps)} dt.
\end{align*}
Summing up inequality $\eqref{lem:bound_disc_energies_aux_ineq}$ from $1$ to $k$, we obtain (for a constant $\tilde{C}>0$ independent of $\tau$ and $\vareps$)
\begin{align}
\begin{aligned}\label{ineq:bound_dis_energy_aux}
\meps&(\uepstau{k})  + \sum_{l=1}^k \frac{1}{\tau} \Reps\left(\uepstau{l-1},\uepstau{l} - \uepstau{l-1}\right) 
\\
\le& \meps(\ueps^{in}) + \tilde{C}\sum_{l=1}^k \left(\meps(\uepstau{l})  + \meps(\uepstau{l-1})\right)\int_{(l-1)\tau}^{l\tau} \|\dot{f}_{\vareps} \|_{L^2(\oeps)} +\frac{1}{\sqrt{\vareps}} \|\dot{g}_{\vareps}\|_{L^2(\geps)} dt
\\
&+ C \int_0^{k\tau} \|\dot{f}_{\vareps}\|_{L^2(\oeps)} +\frac{1}{\sqrt{\vareps}}\|\dot{g}_{\vareps}\|_{L^2(\geps)} dt 
\\
&+ \int_{\oeps} f_{\vareps}(k\tau) \cdot \uepstau{k} dx + \int_{\geps} g_{\vareps}(k\tau) \cdot \uepstau{k} d\sigma
\\
&- \int_{\oeps} f_{\vareps}(0) \cdot \ueps^{in} dx - \int_{\geps} g_{\vareps} (0) \cdot \ueps^{in} dx.
\end{aligned}
\end{align}
Due to the assumptions on $f_{\vareps}$ and $g_{\vareps}$, see \ref{ass:f_eps} and \ref{ass:g_eps}, we obtain for every $t \in [0,T]$
\begin{align*}
\|f_{\vareps}(t)\|_{L^2(\oeps)} &\le C \|f_{\vareps}\|_{W^{1,1}((0,T),L^2(\oeps))} \le C,
\\
\|g_{\vareps}(t)\|_{L^2(\geps)} &\le C \|g_{\vareps}\|_{W^{1,1}((0,T),L^2(\geps))} \le C\sqrt{\vareps},
\end{align*}
for a constant $C>0$ independent of $\vareps$, see \cite[Section 5.9.2, Theorem 2]{EvansPartialDifferentialEquations}. Hence, we obtain for $\theta >0$ arbitrary (with similar arguments as above)
\begin{align*}
 \int_{\oeps} f_{\vareps}(k\tau) \cdot \uepstau{k} dx + \int_{\geps} g_{\vareps}(k\tau) \cdot \uepstau{k} d\sigma \le C \|\uepstau{k}\|_{H^1(\oeps)} \le C(\theta) + \theta \meps(\uepstau{k})
\end{align*}
for a constant $C(\theta)>0$ independent of $\vareps$. Due to the assumptions on the initial values \ref{ass:ueps_in}, we get that the terms including in initial states in $\eqref{ineq:bound_dis_energy_aux}$ are bounded by a constant $C>0$ independent of $\vareps$. Further, using again the assumptions on $f_{\vareps}$ and $g_{\vareps}$, we can choose $\tau_0$ so small, such that for all $\tau \in (0,\tau_0)$ it holds for $l=0,\ldots, k$ that  ($\tilde{C}$ from $\eqref{ineq:bound_dis_energy_aux}$)
\begin{align*}
b_l := \tilde{C} \int_{(l-1)\tau}^{l\tau} \|\dot{f}_{\vareps}\|_{L^2(\oeps)} + \frac{1}{\sqrt{\vareps}} \|\dot{g}_{\vareps}\|_{L^2(\geps)} dt \le \frac14.
\end{align*}
Altogether, we obtain choosing $\theta = \frac14$ with some elemental calculations from $\eqref{ineq:bound_dis_energy_aux}$
\begin{align}\label{inqu:lower_bound_dtau_ueps_aux}
\frac12\meps(\uepstau{k})+ \sum_{l=1}^k \frac{1}{\tau} \Reps\left(\uepstau{l-1},\uepstau{l} - \uepstau{l-1}\right)  \le C + \sum_{l=1}^{k-1} (b_l + b_{l+1}) \meps(\uepstau{l}).
\end{align}
Now, using the non-negativity of the dissipation terms, the discrete Gronwall inequality implies
\begin{align*}
\meps(\uepstau{k}) &\le C \exp\left( 2 \sum_{l=1}^{k-1} (b_l + b_{l+1})\right) \le C \exp\left(4 \sum_{l=1}^k b_l\right) 
\\
&\le C\exp\left( 4 \left(\|f_{\vareps}\|_{W^{1,1}((0,T),L^2)(\oeps))} + \|g_{\vareps}\|_{W^{1,1}((0,T),L^2(\geps))} \right)\right) =:M,
\end{align*}
since
\begin{align*}
\sum_{l=1}^k b_l = \tilde{C} \int_0^{k\tau} \|\dot{f}_{\vareps}\|_{L^2(\oeps)} + \frac{1}{\sqrt{\vareps}} \|\dot{g}_{\vareps}\|_{L^2(\geps)} dt .
\end{align*}
The constant $M$ is independent of $\vareps$ and $\tau$. This implies the uniform bound $\eqref{lem:uniform_bound_dis_energy_in_meps}$.

It remains to check $\eqref{lem:uniform_bound_dis_energy_in_dtau_ueps}$. First of all, from $\eqref{inqu:lower_bound_dtau_ueps_aux}$  and the bounds on $\meps(\uepstau{k})$ already obtained, we get
\begin{align}\label{ineq:lower_bound_dtau_ueps_aux2}
\sum_{l=1}^k \frac{1}{\tau} \Reps\left(\uepstau{l-1},\uepstau{l} - \uepstau{l-1}\right) \le C.
\end{align}
Since $\uepstau{k} \in \mathcal{F}_{\vareps}^M$ for every $k \in \{1,\ldots,N_{\tau}\}$, we can apply Lemma \ref{lem:Lower_bound_det} to obtain a constant $\kappa_M>0$ independent of $\vareps$ (and $\tau$), such that 
\begin{align*}
\det(\nabla \tuepstau{k})>\kappa_M,
\end{align*}
where $\tuepstau{k}$ denotes the extension of $\uepstau{k}$ with the extension operator from Theorem \ref{thm:main_thm_extension_operator}. Hence, we have ($p>n$)
\begin{align*}
\left\{\nabla \tuepstau{k}\right\}_{k=0}^{N_{\tau}} \subset \left\{F \in W^{1,p}(\Omega)^{n\times n }\, : \, \det(F)>\kappa_M\right\} \subset C^0(\overline{\Omega})^{n\times n}
\end{align*}
and the last embedding is compact. An application of the Korn inequality from Theorem \ref{thm:main_thm_Korn_inequality} implies the existence of a constant $c_M>0$ independent of $\vareps$, such that (see also assumption \ref{ass:diss_lower_bound})
\begin{align*}
\frac{1}{\tau} \Reps&\left(\uepstau{l-1}, \uepstau{l} - \uepstau{l-1}\right)
\\
&\geq \frac{c_0}{\tau} \left\| \nabla (\uepstau{l} - \uepstau{l-1})^{\top} \nabla \uepstau{l-1} + (\nabla \uepstau{l-1})^{\top} \nabla (\uepstau{l}-\uepstau{l-1})\right\|^2_{L^2(\oeps)}
\\
&\geq \frac{c_M}{\tau}\left\|\nabla \uepstau{l} - \nabla \uepstau{l-1}\right\|_{L^2(\oeps)}^2 = \tau c_M \left\| \nabla \dtau \uepstau{l}\right\|^2_{L^2(\oeps)}.
\end{align*}
From $\eqref{ineq:lower_bound_dtau_ueps_aux2}$ we obtain the desired result.
\end{proof}

As an immediate consequence we obtain the following \textit{a priori} estimates for the interpolations:
\begin{corollary}\label{cor:apriori_interpolations}
Let $\{\uepstau{k}\}_{k=1}^{N_{\tau}}$ be the sequence of minimizers from Proposition \ref{prop:minimizer_discrete} for $\tau \in (0,\tau_0)$ with $\tau_0$ from Lemma \ref{lem:uniform_bound_dis_energy}. We have $\buepstau ,\lbuepstau \in L^{\infty}((0,T),\Yepsid^{\kappa_M})$ with $\kappa_M $ from Lemma \ref{lem:uniform_bound_dis_energy} and $\huepstau   \in L^{\infty}((0,T),\Yepsid^{\kappa_M/2}) \cap H^1((0,T),H_{\geps^D}^1(\oeps))^n$. Then there exists $\alpha \in (0,1)$ and a constant $C>0$ independent of $\vareps $ and $\tau$, such that
\begin{subequations}
\begin{align}
\|\buepstau\|_{L^{\infty}((0,T),W^{2,p}(\oeps))} + \|\lbuepstau\|_{L^{\infty}((0,T),W^{2,p}(\oeps))} + \|\huepstau\|_{L^{\infty}((0,T),W^{2,p}(\oeps))}  &\le C
\\
\label{ineq:apriori_inter_time_der}
\|\dhuepstau\|_{L^2((0,T),H^1(\oeps))} &\le C,
\\
\label{ineq:apriori_inter_hoelder}
\|\huepstau\|_{C^{0,\alpha}([0,T],C^{1,\alpha}(\overline{\oeps}))} &\le C.
\end{align}
\end{subequations}
\end{corollary}
\begin{proof}
From the bound $\eqref{lem:uniform_bound_dis_energy_in_meps}$ in Lemma \ref{lem:uniform_bound_dis_energy} on the discrete mechanical energies we obtain for every $k\in \{0,\ldots,N_{\tau}\}$ and the assumptions on $W_{\vareps}$ and $H_{\vareps}$ in \ref{AssumptionLowerBoundW} and \ref{AssumptionLowerBoundH}
\begin{align*}
\|\uepstau{k}\|_{H^1(\oeps)} + \|\nabla^2 \uepstau{k}\|_{L^p(\oeps)} \le C
\end{align*}
for a constant $C>0$ independent of $\vareps$ and $\tau$. Now, the Poincar\'e inequality in Lemma \ref{lem:Poincare-inequality} applied to $\uepstau{k}$ and $\nabla \uepstau{k}$ implies
\begin{align*}
\|\uepstau{k}\|_{W^{2,p}(\oeps)} \le C.
\end{align*}
Hence, we obtain for every $t \in ((k-1)\tau,k\tau)$
\begin{align*}
\|\huepstau(t)\|_{W^{2,p}(\oeps)} \le \|\uepstau{k}\|_{W^{2,p}(\oeps)} + \|\uepstau{k-1}\|_{W^{2,p}(\oeps)} \le C.
\end{align*}
This implies the bound for $\huepstau$ in $L^{\infty}((0,T),W^{2,p}(\oeps))^n$. The bounds for $\buepstau$ and $\lbuepstau$ follow by the same arguments. To estimate the time derivative of $\huepstau$ we use inequality $\eqref{lem:uniform_bound_dis_energy_in_dtau_ueps}$ from Lemma \ref{lem:uniform_bound_dis_energy} together with the Poincar\'e inequality from Lemma \ref{lem:Poincare-inequality} in the appendix
\begin{align*}
\|\dhuepstau\|_{L^2((0,T)\times \oeps)}^2 &= \sum_{l=1}^{N_{\tau}} \int_{(l-1)\tau}^{l\tau} \|\dtau \uepstau{k} \|^2_{L^2(\oeps)} dt
= \tau  \sum_{l=1}^{N_{\tau}} \|\dtau \uepstau{k}\|_{L^2(\oeps)}^2
\\
&\le C  \sum_{l=1}^{N_{\tau}}\tau \|\dtau \nabla \uepstau{k}\|^2_{L^2(\oeps)} \le C.
\end{align*} 
In the same way we can estimate $\nabla \dhuepstau$, which gives $\eqref{ineq:apriori_inter_time_der}$. Now, the estimate in the H\"older-norm follows from the \textit{a priori} estimates already obtained and Lemma \ref{lem:embedding_hoelder}. Finally, since $\nabla\huepstau \in C^{0,\alpha}([0,T],C^{0,\alpha}(\overline{\oeps}))^n$ and $\nabla \huepstau (k\tau) = \uepstau{k}$ for $k=0,\ldots,N_{\tau}$, we obtain from $\uepstau{k} \in \Yepsid^{\kappa_M}$ and the Lipschitz continuity of the determinant on compact sets (remember $\|\nabla \huepstau\|_{L^{\infty}((0,T)\times\oeps)}$ is bounded) that for $\tau$ small enough we have $\huepstau(t,\cdot) \in \Yepsid^{\kappa_M /2}$ for all $t \in [0,T]$.
\end{proof}

\subsection{Existence of a weak solution and \textit{a priori} estimates}

Finally, we are able to show the existence of a weak solution of the microscopic model $\eqref{MicroModel}$ and give the $\vareps$-uniform \textit{a priori} estimate for such solutions. We start with some immediate compactness results for the interpolations, which are a consequence of the estimates obtained in Corollary \ref{cor:apriori_interpolations} in the previous section.

\begin{proposition}\label{prop:comp_interpolations}
Let $\{\uepstau{k}\}_{k=1}^{N_{\tau}}$ be the sequence of minimizers from Proposition \ref{prop:minimizer_discrete} for $\tau \in (0,\tau_0)$ with $\tau_0$ from Lemma \ref{lem:uniform_bound_dis_energy}.
There exists $\ueps \in L^{\infty}((0,T), \Yepsid^{\kappa_M})\cap H^1((0,T),H^1(\oeps))^n$, such that up to a subsequence of $\tau$ (still denoted by $\tau$) it holds for $\tau \to 0$ that
\begin{align*}
\buepstau,\lbuepstau,\huepstau &\rightharpoonup^{\ast} \ueps &\mbox{ weakly$^{\ast}$ in }& L^{\infty}((0,T),W^{2,p}(\oeps))^n,
\\
\huepstau &\rightharpoonup \ueps &\mbox{ weakly in }& H^1((0,T),H^1(\oeps))^n,
\\
\nabla \huepstau &\rightarrow \nabla \ueps &\mbox{ in }& C^0([0,T]\times \overline{\oeps})^{n\times n},
\\
\nabla \lbuepstau, \nabla \buepstau &\rightarrow \nabla \ueps &\mbox{ in }& L^{\infty}((0,T)\times \oeps)^{n\times n}.
\end{align*}
\end{proposition}
\begin{proof}
We first consider the sequence $\huepstau$. From the \textit{a priori} estimates in Corollary \ref{cor:apriori_interpolations} we immediately obtain the existence of $\ueps \in L^{\infty}((0,T),W^{2,p}(\oeps))^n \cap H^1((0,T),H^1(\oeps))^n$ such that up to a subsequence  (for $\tau \to 0$)
\begin{align*}
\huepstau &\rightharpoonup^{\ast} \ueps &\mbox{ weakly$^{\ast}$ in }& L^{\infty}((0,T),W^{2,p}(\oeps))^n,
\\
\huepstau &\rightharpoonup \ueps &\mbox{ weakly in }& H^1((0,T),H^1(\oeps))^n.
\end{align*}
Further, since $\huepstau$ is bounded in $C^{0,\alpha}([0,T],C^{1,\alpha}(\overline{\oeps}))^n$, see $\eqref{ineq:apriori_inter_hoelder}$ in Corollary \ref{cor:apriori_interpolations}, the Arzel\`a-Ascoli theorem implies that up to a subsequence $\nabla\huepstau$ converges strongly to $\ueps$ in $C^0([0,T]\times \overline{\oeps})^n$. It remains to check the convergences of $\buepstau$ and $\lbuepstau$. We only consider $\buepstau$. Again from the \textit{a priori} estimates in Corollary \ref{ineq:apriori_inter_hoelder} we get (up to a subsequence) that $\buepstau$ converges weakly$^{\ast}$ in $L^{\infty}((0,T),W^{2,p}(\oeps))^n$. From Lemma \ref{lem:error_interpolations} below we obtain that the limit is equal to $\ueps$, and we also obtain the strong convergence of $\nabla \buepstau$ in $L^{\infty}((0,T)\times \oeps)^{n\times n}$.
\end{proof}

\begin{lemma}\label{lem:error_interpolations}
Under the assumptions of Proposition \ref{prop:comp_interpolations} there exists a constant $C>0$  independent of $\vareps$ and $\theta \in (0,1)$, such that
\begin{align*}
\|\buepstau - \huepstau\|_{L^{\infty}((0,T),W^{1,\infty}(\oeps))} + \|\lbuepstau - \huepstau\|_{L^{\infty}((0,T),W^{1,\infty}(\oeps))} \le C\tau^{\frac{\theta}{2}}.
\end{align*}

\end{lemma}
\begin{proof}
We only give the proof for $\buepstau$. Since the embedding  $W^{2,p}(\oeps) \hookrightarrow W^{1,\infty}(\oeps)$ is continuous, we obtain that the functions $\buepstau,\huepstau: (0,T)\rightarrow W^{1,\infty}(\oeps))^n$ are Bochner-measurable (we emphasize that the range of the embedding is separable in $W^{1,\infty}(\oeps)$) and therefore  $\buepstau,\huepstau \in L^{\infty}((0,T),W^{1,\infty}(\oeps))^n$.
%
 Further, we have for $t \in ((k-1)\tau,k\tau)$
\begin{align*}
\buepstau(t) - \huepstau(t) = (k\tau- t) \dtau \uepstau{k}.
\end{align*}
Estimate $\eqref{lem:uniform_bound_dis_energy_in_dtau_ueps}$ implies (with the Poincar\'e inequality from Lemma \ref{lem:Poincare-inequality} in the appendix)
\begin{align*}
\|\buepstau(t)-\huepstau(t)\|_{H^1(\oeps)} \le \tau \|\dtau \uepstau{k}\|_{H^1(\oeps)} \le C\tau^{\frac12},
\end{align*}
and therefore
\begin{align*}
\|\buepstau - \huepstau\|_{L^{\infty}((0,T),H^1(\oeps))} \le C\tau^{\frac12}.
\end{align*}
Now, let $\theta,q,s,$ and $\beta$ be as in Lemma \ref{lem:embedding_hoelder} in the appendix, and $\bar{s} = \theta + 2(1-\theta) = 2 - \theta = s+1$. Then the Gagliardo-Nirenberg inequality implies 
\begin{align*}
\|\buepstau (t) - \huepstau(t)\|_{C^{1,\beta}(\overline{\oeps})} &\le C \|\uepstau{k} - \uepstau{k-1}\|_{W^{\bar{s},q}(\oeps)}
\\
&\le C \|\uepstau{k} - \uepstau{k-1}\|_{W^{2,p}(\oeps)}^{1-\theta} \|\uepstau{k}-\uepstau{k-1}\|_{H^1(\oeps)}^{\theta}
\\
&\le C\|\uepstau{k}-\uepstau{k-1}\|_{H^1(\oeps)}^{\theta}
\\
&\le C\tau^{\frac{\theta}{2}},
\end{align*}
for a constant $C>0$ independent of $\vareps$ (for this we can use again the extension operator as in the proof of Lemma \ref{lem:embedding_hoelder}).  This implies the desired result.
\end{proof}

From the convergence results in Proposition \ref{prop:comp_interpolations} we get the following \textit{a priori} estimates for the limit functions $\ueps$:
\begin{corollary}\label{cor:apriori_ueps}
The limit function $\ueps$ from Proposition \ref{prop:comp_interpolations} fulfills the following \textit{a priori} estimates:
\begin{align*}
\|\ueps\|_{L^{\infty}((0,T),W^{2,p}(\oeps))} 
+ \|\ueps\|_{H^1((0,T),H^1(\oeps))} &\le C,
\end{align*}
for a constant $C>0$ independent of $\vareps$. Addtionally, it holds that $\ueps \in L^{\infty}((0,T),\Yepsid^{\kappa_M})$ with $\kappa_M>0$ independent of $\vareps$ obtained in Lemma \ref{lem:uniform_bound_dis_energy}.
\end{corollary}
\begin{proof}
The estimates follow directly from the weak and weak$^{\ast}$ lower semicontinuity of the norms and the convergence results in Proposition \ref{prop:comp_interpolations}. To show that $\ueps$ maps into $\Yepsid^{\kappa_M}$ it is enough to show that almost everywhere in $(0,T)\times \oeps$ it holds that $\det(\nabla \ueps)\geq \kappa_M$. For every $t \in((k-1)\tau , k\tau)$ with $k \in \{1,\ldots,N_{\tau}\}$ and $x \in \oeps$ we have 
\begin{align*}
\det(\nabla \buepstau(t,x) ) = \det(\nabla \uepstau{k}(x)) \geq \kappa_M.
\end{align*}
Now, the desired result follows from the strong convergence of $\buepstau$ in $L^{\infty}((0,T)\times \oeps)$  in Proposition \ref{prop:comp_interpolations}.
\end{proof}
Now, we are able to prove the existence of a weak solution of the microscopic model $\eqref{MicroModel}$.
\begin{proof}[Proof of Theorem \ref{thm:main_thm_existence}]
We follow the same arguments as in \cite[Porposition 5.2]{mielke2020thermoviscoelasticity} and \cite[Proposition 4.4]{badal2023nonlinear}, but for the sake of completeness we give the proof. Let $X:= L^p((0,T),W^{2,p}(\oeps))^n$ and define the operator $A:X \rightarrow X'$ (we suppress the index $\vareps$) by 
\begin{align*}
\langle A\veps, \zeps\rangle_{X',X} := \int_0^T \int_{\Omega} \partial_G H_{\vareps} (\nabla^2 \veps) \vdots \nabla^2 \zeps dx dt.
\end{align*}
We emphasize that the $L^p$-integrability in time in the choice of $X$ is necessary, since this implies $\partial_G H_{\vareps} (\nabla^2 \veps) \in L^{p'}((0,T)\times \oeps)^{n\times n \times n}$ (see assumption \ref{AssumptionLowerBoundH}), and therefore the right-hand side above is well-defined. We want to apply Minty's trick to $A$ and the sequence $\buepstau$. First of all we notice that the convexity of $H$ implies the monotonicity of $A$, and using again \ref{AssumptionLowerBoundH} we obtain the hemicontinuity of $A$. Proposition \ref{prop:comp_interpolations} implies the weak convergence of $\buepstau$ in $X$. Now, using the weak equation in Remark \ref{rem:Euler_Lagrange_discrete}, we obtain after an elemental calculation for all $\zeps \in X$ with $\blepstau(t):= \lepstau{k}$ for $t \in [(k-1)\tau,k\tau)$ with $k =1,\ldots, N_{\tau}$
\begin{align*} 
\langle A \buepstau , \zeps \rangle_{X',X} =&  -\int_0^T \int_{\oeps} \left[ \partial_F W_{\vareps} (\nabla \buepstau) + \partial_{\dot{F}}R_{\vareps} (\nabla \lbuepstau ,  \nabla \dhuepstau ) \right] : \nabla \zeps dx dt 
\\
&+ \int_0^T \langle \blepstau ,\zeps \rangle dt. 
\end{align*}
We show
\begin{align*}
\langle A \buepstau ,\zeps \rangle_{X',X} \overset{\tau \to 0}{\longrightarrow}& -\int_0^T \int_{\oeps} \left[\partial_F W_{\vareps} (\nabla \ueps) + \partial_{\dot{F}} R_{\vareps}(\nabla \ueps, \nabla\dueps ) \right] : \nabla \zeps dx dt 
\\
&+ \int_0^T \int_{\oeps} f_{\vareps} \cdot \zeps dx dt + \int_0^T \int_{\geps } g_{\vareps} \cdot \zeps d\sigma dt 
=: \langle b_{\vareps} ,\zeps\rangle_{X',X}
\end{align*}
and
\begin{align}\label{conv:Minty_trick}
\langle A \buepstau ,\buepstau \rangle_{X',X} \overset{\tau\to 0}{\longrightarrow} \langle b_{\vareps} , \ueps\rangle_{X',X}.
\end{align}
Then the Minty trick implies $A\ueps = b_{\vareps}$ which implies the desired result. We only show the convergence in $\eqref{conv:Minty_trick}$, since the other one follows by similar (simpler) arguments.
For the term including $\partial_F W_{\vareps}(\nabla \buepstau)$ we use the continuity of $\partial_F W$, see assumption \ref{ass:W_regularity}, and the strong convergence of $\nabla \buepstau$ in $L^{\infty}((0,T)\times \oeps)^{n\times n}$ from Proposition \ref{prop:comp_interpolations}. For the dissipation term we use the continuity of $\partial_{\dot{F}} R_{\vareps}$ and its linearity with respect to the $\dot{F}$-variable together with the strong convergence of $\nabla \lbuepstau$ in $L^{\infty}((0,T)\times \oeps)^{n \times n}$ and the weak convergence of $\nabla \dhuepstau$ in $L^2((0,T)\times \oeps)^{n\times n}$  from Proposition \ref{prop:comp_interpolations}. Further, by an elemental calculation we obtain 
\begin{align*}
\bigg|  \sum_{k=1}^{N_{\tau}} \int_{\geps} \gepstau{k} \cdot \buepstau d\sigma &- \int_0^T \int_{\geps } g_{\vareps} \cdot \buepstau d\sigma dt  \bigg|
\\
&\le \sum_{k=1}^{N_{\tau}} \left\|\dot{g}_{\vareps}\right\|_{L^1(((k-1)\tau,k\tau),L^2(\geps))}\|\ueps\|_{L^1(((k-1)\tau,k\tau),L^2(\geps))} 
\\
&\hspace{5em}+ \|g_{\vareps}\|_{L^2((0,T)\times \geps)} \|\ueps - \buepstau\|_{L^2((0,T)\times \geps)}
\\
&\le C\tau\left(\left\|\dot{g}_{\vareps}\right\|_{L^2((0,T)\times \oeps)}^2 + \|\ueps\|^2_{L^2((0,T)\times \oeps)}\right)
\\
&\hspace{5em} + \|g_{\vareps}\|_{L^2((0,T)\times \geps)} \|\ueps - \buepstau\|_{L^2((0,T)\times \geps)}.
\end{align*}
For fixed $\vareps$ the right-hand side converges to $0$ for $\tau \to 0$. A similar argument holds for $\fepstau{k}$ and $f_{\vareps}$, which implies 
\begin{align*}
\int_0^T \langle \blepstau,\buepstau \rangle dt \overset{\tau \to 0}{\longrightarrow} \int_0^T \int_{\oeps} f_{\vareps} \cdot \ueps dx dt + \int_0^T \int_{\geps } g_{\vareps} \cdot \ueps d\sigma dt.
\end{align*}
Hence, we obtain $\eqref{conv:Minty_trick}$ and therefore the desired result.
\end{proof}

\section{Derivation of the macroscopic model}
\label{sec:derivation_macro_model}
Based on the \textit{a priori} estimates from Section 7 in Corollary \ref{cor:apriori_ueps} we are able to pass to the limit $\vareps$ in the microscopic problem.
We denote by $\tueps$ the extension of $\ueps$ to the whole domain $\Omega$ via the extension operator from Theorem \ref{thm:main_thm_extension_operator}. More precisely, for  almost every $t \in (0,T)$ we define $\tueps(t):= E_{\vareps}\ueps(t)$. Then we have the following bounds for $\tueps$:
\begin{lemma}\label{lem:apriori_tueps}
For a constant $C>0$ independent of $\vareps$ it holds that
\begin{align*}
\|\tueps\|_{L^{\infty}((0,T),W^{2,p}(\Omega))}
+\|\tueps\|_{H^1((0,T),H^1(\Omega))} &\le C,
\end{align*}

\end{lemma}
\begin{proof}
For the extension operator $E_{\vareps}$ from Theorem \ref{thm:main_thm_extension_operator} it holds that $E_{\vareps}:H^1(\oeps)^n \rightarrow H^1(\Omega)^n$ and $E_{\vareps}:W^{2,p}(\oeps)^n \rightarrow W^{2,p}(\Omega)^n$, and for all $\veps \in H^1(\oeps)^n$ and $ \weps \in W^{2,p}(\oeps)^n$ we have
\begin{align*}
\|E_{\vareps}\veps\|_{H^1(\Omega)}\le C\|\veps\|_{H^1(\oeps)},\quad \|E_{\vareps}\weps\|_{W^{2,p}(\Omega)}\le C\|\weps\|_{W^{2,p}(\oeps)}.
\end{align*}
We immediately obtain the bound of $\tueps$ in $L^{\infty}((0,T),W^{2,p}(\Omega))^n$. To estimate the time derivative we introduce for $0<h\ll 1$ and $t \in (0,T-h)$ the difference quotient
\begin{align*}
\delta_h \tueps(t):= \frac{\tueps(t+h) - \tueps(t)}{h},
\end{align*}
and in a similar way $\delta \ueps(t)$.
Since the operator $E_{\vareps}$ is linear and acts pointwise on $\ueps$ with respect to $t$, we get for almost every $t\in (0,T-h)$
\begin{align*}
\|\delta_h \tueps(t)\|_{H^1(\Omega)} = \|E_{\vareps} \delta_h \ueps(t) \|_{H^1(\Omega)} \le C \|\delta_h \ueps(t)\|_{H^1(\oeps)} \le C \|\dueps(t)\|_{H^1(\oeps)}.
\end{align*}
With the \textit{a priori} bound for $\dueps$ from Corollary \ref{cor:apriori_ueps} we get
\begin{align*}
\|\delta_h \tueps\|_{L^2((0,T-h),H^1(\Omega))} \le C,
\end{align*}
which implies $\dtueps \in L^2((0,T),H^1(\Omega))^n$ and 
\begin{align*}
\|\dtueps\|_{L^2((0,T),H^1(\Omega))} \le C.
\end{align*}
\end{proof}

For the sequence of microscopic solutions we have the following (two-scale) compactness result:
\begin{proposition}\label{prop:conv:tueps}
There exist limit functions $u_0 \in L^{\infty}((0,T),\mathcal{Y}_{\id}^{\kappa_M}) \cap H^1((0,T),H^1(\Omega))^n$  and $u_2 \in L^{\infty}((0,T),L^p(\Omega,W^{2,p}_{\per}(Y)/\R))$  such that up to a subsequence
 for all $s \in [1,\infty)$
\begin{subequations}
\begin{align}
\label{conv:tueps_linftyw2p}\tueps &\rightharpoonup^{\ast} u_0 &\mbox{ weakly$^{\ast}$ in }& L^{\infty}((0,T),W^{2,p}(\Omega))^n,
\\
\label{conv:tueps_c0c1}\tueps &\rightarrow u_0 &\mbox{ in }& C^0([0,T],C^1(\overline{\Omega}))^n,
\\
\label{conv:tueps_ts_nabla2}\chi_{\oeps} 
\nabla^2 \ueps &\rightwts{s,p}  \chi_{Y_s}\left(\nabla_x^2 u_0 + \nabla_y^2 u_2\right) ,
\\
\label{conv:tueps_H^1H^1}\tueps &\rightharpoonup u_0 &\mbox{ weakly in }& H^1((0,T),H^1(\Omega))^n,
\\
\label{conv:tueps_time_2scale}
\nabla\dtueps &\rightwts{2} \nabla \dot{u}_0.
\end{align}
\end{subequations}
\end{proposition}
\begin{proof}
$\eqref{conv:tueps_linftyw2p}$ follows directly from the \textit{a priori} estimates in Lemma \ref{lem:apriori_tueps}. For the second inequality we use the embedding (which was shown in a slightly different version for perforated domains in Lemma \ref{lem:embedding_hoelder})
\begin{align*}
L^{\infty}((0,T),W^{2,p}(\Omega)) \cap H^1((0,T),H^1(\Omega)) \hookrightarrow  C^{0,\alpha}([0,T],C^{1,\alpha}(\overline{\Omega}))
\end{align*}
for a suitable $\alpha \in (0,1)$, where we used $p>n$. Now, using again the estimates from Lemma \ref{lem:apriori_tueps} and the Arzel\`a-Ascoli theorem, we obtain $\eqref{conv:tueps_c0c1}$. The two-scale convergence in $\eqref{conv:tueps_ts_nabla2}$ follows from Corollary \ref{cor:two_scale_conv_second_order}. The  convergence $\eqref{conv:tueps_H^1H^1}$ is again a direct consequence of Lemma \ref{lem:apriori_tueps}. For the last convergence $\eqref{conv:tueps_time_2scale}$ we first notice that there exists $U_1 \in L^2((0,T)\times \Omega, H^1_{\per}(Y)/\R)^n$ such that up to a subsequence  (see Lemma \ref{BasicTwoScaleCompactness} in the appendix for the stationary case)
\begin{align*}
\nabla \dtueps \rightwts{2} \nabla \dot{u}_0 + \nabla_y U_1.
\end{align*}
Hence, for every $\phi \in C_0^{\infty}(\Omega,C_{\per}^{\infty}(Y))^{n\times n}$ and $\psi \in C_0^{\infty}(0,T)$ we obtain by integration by parts:
\begin{align*}
\int_0^T \int_{\Omega} \int_Y \nabla \dot{u}_0 : \phi \psi  dy dx dt &= -\int_0^T \int_{\Omega} \int_Y \nabla u_0  : \phi \psi' dy dx dt 
\\
&= \lim_{\vareps \to 0} -\int_0^T \int_{\Omega} \nabla \tueps : \phi\left(x,\fxe\right) \psi'dx dt
\\
&= \lim_{\vareps \to 0} \int_0^T \int_{\Omega} \nabla \dtueps : \phi\left(x,\fxe\right) \psi dx dt
\\
&= \int_0^T \int_{\Omega} \int_Y \left[\nabla \dot{u}_0 + \nabla_y U_1\right] : \phi \psi dy dx dt.
\end{align*}
This implies $\nabla_y U_1 = 0$ and therefore $U_1 = 0$, which gives $\eqref{conv:tueps_time_2scale}$. The lower bound for the determinant $\det(\nabla u_0) \geq \kappa_M$ follows by similar arguments as at the end of the proof of Lemma \ref{lem:Lower_bound_det} and therefore $u_0 \in L^{\infty}((0,T),\mathcal{Y}_{\id}^{\kappa_M})$.
This finishes the proof.
\end{proof}
From the (strong) convergence results above we obtain the following two-scale compactness results for the nonlinear terms of lower order:
\begin{corollary}\label{cor:conv_nonlinear}
The subsequence from Proposition \ref{prop:conv:tueps} fulfills for every $s \in [1,\infty)$
\begin{align*}
\partial_F W\left(\cdot_y , \teps(\nabla \tueps) \right) &\rightarrow  \partial_F W(\cdot_y, \nabla u_0) &\mbox{ in }& L^s((0,T)\times \Omega \times Y)^{n\times n},
\\
\partial_{\dot{F}} R\left(\cdot_y , \teps(\nabla \tueps) , \teps(\nabla \dtueps)\right) &\rightharpoonup \partial_{\dot{F}} R(\cdot_y \nabla u_0 , \nabla \dot{u}_0) &\mbox{ weakly in }& L^2((0,T)\times \Omega \times Y)^{n\times n}.
\end{align*}
In other words, we have $\partial_F W_{\vareps}(\nabla \tueps) \rightsts{s} \partial_F W(\nabla u_0)$ and $\partial_{\dot{F}} \Reps(\nabla \tueps,\nabla \dtueps) \rightwts{2} \partial_{\dot{F}} R(\nabla u_0 , \nabla \dot{u}_0)$.
\end{corollary}
\begin{proof} 
From Lemma \ref{lem:unfolding_c0_linfty} we get with the strong convergence of $\nabla \tueps$  in $C^0([0,T]\times \overline{\Omega})^{n\times n}$ from $\eqref{conv:tueps_c0c1}$ that 
\begin{align*}
\teps(\nabla\tueps) \rightarrow \nabla u_0 \quad \mbox{ in } L^{\infty}((0,T)\times \Omega \times Y)^{n\times n}.
\end{align*}
Especially, the sequence $\teps(\nabla \ueps)$ is pointwise bounded. We have almost everywhere in $(0,T)\times \Omega \times Y$ (remember that $\partial_F W$ is continuous on $GL^+(n)$)
\begin{align*}
\partial_F W(\cdot_y, \teps(\nabla \tueps)) \rightarrow \partial_F W(\cdot_y, \nabla u_0).
\end{align*}
Now, the convergence theorem of Lebesgue implies the convergence of $\partial_F W(\cdot_y,\teps(\nabla \tueps))$ in $L^s((0,T)\times \Omega \times Y)^{n\times n}$ for every $s\in [1,\infty)$.
Next, we notice that
\begin{align*}
\partial_{\dot{F}} R &\left(\cdot_y , \teps(\nabla \tueps) , \teps(\nabla \dtueps)\right) 
\\
&= 2 \teps(\nabla \tueps) \left( D\left(\cdot_y,\teps(\nabla \tueps)\teps(\nabla \tueps)^{\top}\right) \left[\teps(\nabla \dtueps)^{\top} \teps(\nabla \tueps) + \teps(\nabla \tueps)^{\top} \teps(\nabla \dtueps) \right]\right).
\end{align*}
With similar arguments as above we can show that every factor in this product except $\teps(\nabla \dtueps)$ and $\teps(\nabla \dtueps)^{\top}$ converge strongly in $L^{\infty}((0,T)\times \Omega \times Y)$. Since the product between a strong converging sequence in $L^{\infty}$ and a weakly convergent sequence in $L^2$ converges weakly in $L^2$, we obtain the desired result from the two-scale convergence of $\nabla \dtueps$ in $\eqref{conv:tueps_time_2scale}$, which implies the weak convergence of $\teps(\nabla \dtueps)$ (and its transpose) in $L^2((0,T)\times \Omega \times Y)^{n\times n}$.
\end{proof}
Now, we are able to pass to the limit $\vareps \to 0$ in the microscopic equation $\eqref{eq:micro_model_var}$ for a suitable choice of test-functions. We will see that the only critical term remains the term including the higher derivatives, more precisely $\partial_G H_{\vareps} (\nabla^2 \ueps)$. For this we use the convexity of $H$ and some kind of Minty-trick, which was already used (for a lower order problem) in \cite[Proof of Theorem 3.5]{Allaire_TwoScaleKonvergenz}. Let us introduce the following notation
\begin{align*}
\theps := \partial_G H_{\vareps}(\nabla^2 \tueps).
\end{align*}
From the upper bound of $\partial_G H_{\vareps}$ in assumption \ref{AssumptionLowerBoundH} we obtain almost everywhere in $(0,T)$ using $p' (p-1) = p$
\begin{align*}
\|\theps\|_{L^{p'}(\Omega)}^{p'} &\le C \int_{\Omega} \left(1 + |\nabla^2 \tueps|^{p-1} \right)^{p'} dx
\\
&\le C \left( 1 + \|\nabla^2 \tueps\|^p_{L^p(\Omega)} \right) \le C,
\end{align*}
where the last estimate follows from the \textit{a priori} estimates in Lemma \ref{lem:apriori_tueps}. Hence, due to Lemma \ref{BasicTwoScaleCompactness} in the appendix, there exists $h_0 \in L^{\infty}((0,T),L^{p'}(\Omega\times Y))^{n\times n \times n}$,  such that up to a subsequence and arbitrary $s \in [1,\infty)$
\begin{align}\label{conv:heps_ts}
\theps \rightwts{s,p} h_0.
\end{align}
We emphasize that in a first step we only get $h_0 \in L^s((0,T),L^p(\Omega))^{n\times n \times n}$ for $s\in [1,\infty)$. However, then the lower semicontinuity of the norm with respect to the two-scale convergence implies almost everywhere in $(0,T)$
\begin{align*}
\|h_0\|_{L^{p'}(\Omega)} \le \liminf_{\vareps \to 0} \|\theps\|_{L^{p'}(\Omega)} \le C,
\end{align*}
and therefore $h_0 \in L^{\infty}((0,T),L^{p'}(\Omega))^{n\times n \times n}$. Next, we choose test-functions of the form
\begin{align*}
\zeps(t,x) := v_0(t,x) + \vareps^2 v_2\left(t,x,\fxe\right)
\end{align*}
with $v_0 \in C^{\infty}_0([0,T]\times \overline{\Omega}\setminus \Gamma^D)^n$ and $v_2 \in C^{\infty}_0([0,T]\times \Omega , C_{\per}^{\infty}(Y))^n$. Since $\geps^D \subset \Gamma^D$ we have $\zeps = 0 $ on $\geps^D$. Further it holds that
\begin{align*}
\nabla \zeps(t,x) &= \nabla v_0(t,x) + \vareps \nabla_y v_2\left(t,x,\fxe\right) + \vareps^2 \nabla_x v_2 \left(t,x,\fxe\right),
\\
\nabla^2 \zeps(t,x) &= \nabla^2 v_0(t,x) + \nabla_y^2 v_2\left(t,x,\fxe\right) + 2\vareps \nabla_x \nabla_y v_2 \left(t,x,\fxe\right) + \vareps^2 \nabla_x^2 v_2\left(t,x,\fxe\right). 
\end{align*}
Using $\zeps$ as a test-function in $\eqref{eq:micro_model_var}$ we obtain
\begin{align}
\begin{aligned}
\label{eq:deriv_macro_micro_aux}
\int_0^T & \int_{\oeps} \left[\partial_F W_{\vareps}(\nabla\ueps) + \partial_{\dot{F}} R_{\vareps}(\nabla \ueps,\nabla \dueps)\right] : \left[\nabla v_0 + \vareps \nabla_y v_2\left(t,x,\fxe\right) + \vareps^2 \nabla_x v_2 \left(t,x,\fxe\right) \right] dx dt
\\
&+ \int_0^T \int_{\oeps} \underbrace{\partial_G H_{\vareps}(\nabla^2 \ueps)}_{= \theps} \vdots \left[ \nabla^2 v_0 + \nabla_y^2 v_2\left(t,x,\fxe\right) + 2\vareps \nabla_x \nabla_y v_2 \left(t,x,\fxe\right) + \vareps^2 \nabla_x^2 v_2\left(t,x,\fxe\right)\right] dx dt
\\
=& \int_0^T \int_{\oeps} f_{\vareps} \cdot \left[v_0 + \vareps^2 v_2 \left(t,x,\fxe\right) \right] dx dt 
+ \int_0^T \int_{\geps} g_{\vareps} \cdot \left[ v_0 + \vareps^2 v_2 \left(t,x,\fxe\right) \right] d\sigma dt.
\end{aligned}
\end{align}
From the assumptions for $f_{\vareps}$ and $g_{\vareps}$ in \ref{ass:f_eps} and \ref{ass:g_eps} we immediately obtain that the right-hand side converges to
\begin{align*}
\int_0^T \int_{\Omega} \underbrace{\int_{Y_s} f_0(t,x,y) dy}_{=:\bar{f}_0(t,x)} \cdot v_0(t,x)  dx dt + \int_0^T \int_{\Omega} \underbrace{\int_{\Gamma} g_0(t,x,y) d\sigma_y}_{=: \bar{g}_0} \cdot v_0(t,x) dx dt.
\end{align*}
For the first term in $\eqref{eq:deriv_macro_micro_aux}$ we get with the convergence results in Corollary \ref{cor:conv_nonlinear} that it converges to 
\begin{align*}
\int_0^T \int_{\Omega} \left[ \partial_F \overline{W}(\nabla u_0) + \partial_{\dot{F}} \overline{R}(\nabla u_0 , \nabla \dot{u}_0) \right] : \nabla v_0 dx dt
\end{align*}
with 
\begin{align*}
\overline{W}(F):= \int_{Y_s} W(y,F)dy,\quad \overline{R}(F,\dot{F}) := \int_{Y_s} R(y,F,\dot{F}) dy.
\end{align*}
Finally, using the two-scale convergence of $\theps$ we obtain for the second term on the left-hand side in $\eqref{eq:deriv_macro_micro_aux}$ the limit 
\begin{align*}
\int_0^T\int_{\Omega} \int_{Y_s} h_0 \vdots \left[\nabla v_0 + \nabla_y^2 v_2(t,x,y) \right] dy dx dt.
\end{align*}
Altogether, we have the limit equation
\begin{align}
\begin{aligned}\label{eq:deriv_macro_aux}
\int_0^T& \int_{\Omega} \left[ \partial_F \overline{W}(\nabla u_0) + \partial_{\dot{F}} \overline{R}(\nabla u_0 , \nabla \dot{u}_0) \right] : \nabla v_0 dx dt
\\
&+ \int_0^T\int_{\Omega} \int_{Y_s} h_0 \vdots \left[\nabla^2 v_0 + \nabla_y^2 v_2(t,x,y) \right] dy dx dt
\\
=& \int_0^T \int_{\Omega}  \bar{f}_0(t,x)  \cdot v_0(t,x)  dx dt + \int_0^T \int_{\Omega}\bar{g}_0(t,x)  \cdot v_0(t,x) dx dt.
\end{aligned}
\end{align}
By density this equation holds for all $v_0 \in L^2((0,T),W^{2,p}_{\Gamma^D}(\Omega))^n$ and $v_2 \in L^1((0,T),L^p(\Omega,W_{\per}^{2,p}(Y)))^n$.
Choosing $v_0 = 0$ we directly obtain
\begin{align}
\label{eq:minty_aux}
\int_0^T \int_{\Omega} \int_{Y_s} h_0 \vdots \nabla_y^2 v_2 dy dx dt  =0
\end{align}
for all $v_2 \in L^1((0,T),L^p(\Omega,W_{\per}^{2,p}(Y)))^n$. We 
want to identify the function $h_0$ with $\partial_G H(\cdot_y,\nabla^2_x u_0 + \nabla_y^2 u_2)$. For this we define for almost every $(t,x)\in (0,T)\times \Omega$
\begin{align*}
\mu_{\vareps} (t,x):= \nabla^2_x u_0(t,x) + \nabla_y^2 v_2\left(t,x,\fxe\right) + \delta \phi\left(t,x,\fxe\right)
\end{align*}
with $\delta \in \R$,  $\phi \in C^{\infty}_0 ((0,T)\times \Omega \times Y_s)^{n\times n \times n}$  and $v_2 \in  C^{\infty}_0([0,T]\times \Omega , C_{\per}^{\infty}(Y))^n $ as above. Since $H$ is convex, see assumption \ref{AssumpHConvexity}, we obtain
\begin{align}\label{ineq:minty_aux}
\int_0^T \int_{\oeps} \left[ \partial_G H_{\vareps}(\nabla^2 \ueps) - \partial_G H_{\vareps}(\mu_{\vareps})\right] \vdots \left[\nabla^2 \ueps - \mu_{\vareps} \right] dx dt \geq 0.
\end{align}
We want to pass to the limit on the left-hand side. Here the most critical term is the one including $\partial_G H_{\vareps}(\nabla^2 \ueps)\vdots\nabla^2 \ueps$, which will be considered later. For the other ones we can argue in the following way: The oscillation lemma, see \cite[Lemma 5.2]{Allaire_TwoScaleKonvergenz} for the stationary case, implies that for arbitrary $s\in [1,\infty)$ it holds that
\begin{align*}
\mu_{\vareps} \rightsts{s,p} \mu_0 := \nabla^2_x u_0 + \nabla_y^2 v_2 + \delta \phi.
\end{align*}
In others words, we have $\teps(\mu_{\vareps}) \rightarrow \mu_0$ in $L^s((0,T),L^p(\Omega \times Y))^{n\times n \times n}$. 
Especially, the convergence holds up to a subsequence almost everywhere in $(0,T)\times \Omega \times Y$ and using assumption \ref{AssumptionLowerBoundH} we get
\begin{align*}
\left|\partial_G H\left(y,\teps(\mu_{\vareps})(t,x,y)\right)\right|^{p'} \le C |\teps(\mu_{\vareps})|^p. 
\end{align*}
Hence, the (generalized) dominated convergence theorem of Lebesgue, see \cite[Theorem 3.25]{AltLFAenglisch}, implies 
\begin{align*}
\partial_G H(\cdot_y,\teps(\mu_{\vareps})) \rightarrow \partial_G H(\cdot_y , \mu_0) \quad \mbox{ in } L^{p'}((0,T)\times \Omega \times Y)^{n\times n \times n}.
\end{align*}
Since the sequence $\partial_G H(\cdot_y,\teps(\mu_{\vareps}))$ is bounded in $L^{\infty}((0,T),L^{p'}(\Omega \times Y))^{n\times n \times n}$ and $\partial_G H(\cdot_y , \mu_0) \in L^{\infty}((0,T),L^{p'}(\Omega \times Y))^{n\times n \times n}$, we obtain that the convergence above is valid in $L^s((0,T),L^{p'}(\Omega \times Y))^{n\times n \times n}$ for arbitrary $s \in [1,\infty)$. Hence, with the two-scale convergence for $\nabla^2 \ueps$ from Proposition \ref{prop:conv:tueps} we get 
\begin{align*}
\lim_{\vareps \to 0} \int_0^T \int_{\oeps} \partial_G H_{\vareps} (\mu_{\vareps}) \vdots &\left[\nabla^2 \ueps - \mu_{\vareps}\right] dx dt
\\
&= \int_0^T \int_{\Omega}\int_{Y_s} \partial_G H(y,\mu_0) \vdots \left[\nabla^2_x u_0 + \nabla_y^2 u_2 - \mu_0\right]dy dx dt. 
\end{align*}
From the two-scale convergence of $\theps=\partial_G H_{\vareps}(\nabla^2\ueps)$ to $h_0$ obtained in $\eqref{conv:heps_ts}$, we get with the strong two-scale convergence of $\mu_{\vareps}$ that
\begin{align*}
\lim_{\vareps \to 0}\int_0^T \int_{\oeps} \partial_G H_{\vareps} (\nabla^2 \ueps) \vdots \mu_{\vareps} dx dt = \int_0^T \int_{\Omega} \int_{Y_s} h_0 \vdots \mu_0 dy dx dt .
\end{align*}
Finally, for the remaining term in $\eqref{ineq:minty_aux}$ we use the microscopic equation $\eqref{eq:micro_model_var}$ to get with $\eqref{eq:deriv_macro_aux}$
\begin{align*}
\int_0^T& \int_{\oeps} \partial_G H_{\vareps}(\nabla^2 \ueps) \vdots \nabla^2 \ueps dx dt 
\\
=&- \int_0^T \int_{\oeps} \left[\partial_F W_{\vareps} (\nabla \ueps) + \partial_{\dot{F}}R_{\vareps}(\nabla \ueps ,\nabla \dueps) \right] : \nabla \ueps dx dt
\\
&+ \int_0^T \int_{\oeps} f_{\vareps} \cdot \ueps dx dt + \int_0^T \int_{\geps} g_{\vareps} \cdot \ueps d\sigma dt
\\
\rightarrow& - \int_0^T \int_{\Omega} \left[ \partial_F \overline{W}(\nabla u_0) + \partial_{\dot{F}} R(\nabla u_0 , \nabla \dot{u}_0)\right] : \nabla u_0 dx dt
\\
&+ \int_0^T \int_{\Omega} \bar{f}_0 \cdot u_0 dx dt + \int_0^T \int_{\Omega} \bar{g}_0 \cdot u_0 dx dt
\\
=& \int_0^T \int_{\Omega} \int_{Y_s} h_0 \vdots \nabla^2 u_0 dy dx dt.
\end{align*}
where for the convergence we used the results from Proposition \ref{prop:conv:tueps} and Corollary \ref{cor:conv_nonlinear}. Altogether, we obtain from $\eqref{ineq:minty_aux}$ and using $\eqref{eq:minty_aux}$
\begin{align*}
0 \le& \int_0^T \int_{\Omega} \int_{Y_s} h_0 \vdots \left[\nabla_x^2 u_0 - \mu_0 \right] dy dx dt 
\\
&- \int_0^T \int_{\Omega} \int_{Y_s} \partial_G H(y,\mu_0) \vdots \left[\nabla_x^2 u_0 + \nabla_y^2 u_2 - \mu_0 \right] dy dx dt
\\
=& \int_0^T \int_{\Omega} \int_{Y_s} \left[h_0 - \partial_G H(y,\mu_0)\right] \vdots \left[ \nabla_y^2 u_2 - \nabla_y^2 v_2 + \delta \phi \right] dy dx dt.
\end{align*}
By density we can choose $v_2 = u_2$ and obtain
\begin{align*}
0 \le \int_0^T \int_{\Omega} \int_{Y_s} \left[h_0 - \partial_G H(y,\nabla_x^2 u_0 + \nabla_y^2 u_2 + \delta \phi ) \right] \vdots \delta \phi dy dx dt.
\end{align*}
We divide by $\delta$ for sequences $\delta \nearrow 0$ and $\delta \searrow 0$ to obtain
\begin{align*}
0 = \int_0^T \int_{\Omega} \int_{Y_s} \left[h_0 - \partial_G H(y,\nabla_x^2 u_0 + \nabla_y^2 u_2)\right] \vdots \phi dy dx dt
\end{align*}
for all $\phi \in C^{\infty}_0 ((0,T)\times \Omega \times Y_s)^{n\times n \times n}$. This implies 
\begin{align*}
h_0 = \partial_G H(y,\nabla_x^2 u_0 + \nabla_y^2 u_2).
\end{align*}
Especially, we obtain from $\eqref{eq:minty_aux}$ that for every $v_2 \in W_{\per}^{2,p}(Y_s)^n$ and almost everywhere in $(0,T)\times \Omega$
\begin{align}\label{Cell_problem_u_2}
\int_{Y_s} \partial_G H(y, \nabla_x^2 u_0 + \nabla_y^2 u_2) \vdots \nabla_y^2 v_2 dy = 0.
\end{align}
In other words, $u_2$ is a weak solution of the following system
\begin{align}
\begin{aligned}\label{Cell_problem_u_2_strong}
\nabla_y \cdot \left( \nabla \cdot \partial_G H(\nabla_x^2 u_0 + \nabla_y^2 u_2)\right) &= 0 &\mbox{ in }& (0,T)\times \Omega \times Y_s,
\\
- \left[\nabla \cdot \partial_G H (\nabla_x^2 u_0 + \nabla_y^2 u_2)\right] \nu 
\\
-\nabla_S \cdot \left(\partial_G H(\nabla_x^2 u_0 + \nabla_y^2 u_2)\nu \right) &= 0 &\mbox{ on }& (0,T)\times \Omega \times \Gamma,
\\
\partial_G H(\nabla_x^2 u_0 + \nabla_y^2 u_2 ) : (\nu \otimes \nu ) &= 0 &\mbox{ on }& (0,T)\times \Omega \times \Gamma.
\end{aligned}
\end{align}
Let us define for $G \in \R^{n\times n\times n}$ 
\begin{align}\label{def:H_hom}
H_{hom}(G):= \inf_{v_2 \in W_{\per}^{2,p}(Y_s)^n} \int_{Y_s} H(y, G + \nabla_y v_2) dy .
\end{align}
\begin{lemma}
For almost every $(t,x) \in (0,T)\times  \Omega$ it holds that
\begin{align*}
H_{hom}(\nabla_x^2 u_0(t,x)) = \int_{Y_s} H\left(y, \nabla_x^2 u_0(x) + \nabla_y^2 u_2(t,x,y)\right) dy.
\end{align*}
\end{lemma}
\begin{proof}
This follows by standard variational methods. In fact, a minimizer $\tilde{u}_2$ of the right-hand side in $\eqref{def:H_hom}$ is a solution of the Euler-Lagrange equation
\begin{align*}
\int_{Y_s} \partial_G H(y, \nabla_x^2 u_0 + \nabla_y^2 \tilde{u}_2) \vdots \nabla_y^2 v_2 dy = 0
\end{align*}
for all $v_2 \in W_{\per}^{2,p}(Y_s)^n$.
Conversely, if $\tilde{u}_2$ solves the Euler-Lagrange equation above, the convexity of $H$ implies that $\tilde{u}_2$ is a minimizer of $\eqref{def:H_hom}$.
\end{proof}

\begin{remark}
We emphasize that a weak solution of $\eqref{Cell_problem_u_2_strong}$ or equivalent a minimizer of  the right-hand side in $\eqref{def:H_hom}$ is in general not unique. Uniqueness is obtained for $H$ strictly convex.
\end{remark}
Finally, using $\eqref{eq:deriv_macro_aux}$ with $v_2 = 0$ we end up in the  macroscopic variational equation $\eqref{eq:macro_model_var}$.
The initial condition $u_0(0) = u_0^{in}$ follows from the assumption \ref{ass:ueps_in} and the convergence results in Proposition \ref{prop:conv:tueps}.
Hence, $u_0$ is a weak solution of the macroscopic model $\eqref{eq:macro_model}$ and Theorem \ref{thm:main_thm_conv_and_macro_model} is proved.

\begin{appendix}

\section{Auxiliary results}
\label{sec:auxiliary_results}

In this section we give some technical results necessary to show the $\vareps$-uniform estimates for the microscopic solution. Most of the results are well-known in the literature for fixed domains $\Omega$ (independent of $\vareps$). Hence, our aim is to generalize these results to the case of perforated domains. 

\subsection{Poincar\'e and trace inequality}

We have the following uniform Poincar\'e-type inequalities in the perforated domain $\oeps$:
%
\begin{lemma}\label{lem:Poincare-inequality}\
\begin{enumerate}
[label = (\roman*)]
\item Let $p\in (1,\infty)$. Then for all $\veps \in L^1(\oeps)$ with $\nabla \veps \in L^p(\oeps)^n$ it holds that
\begin{align*}
\|\veps \|_{L^p(\oeps)} \le C \left( \| \nabla \veps \|_{L^p(\oeps)} + \|\veps\|_{L^1(\oeps)}\right),
\end{align*}
for a constant $C>0$ independent of $\vareps$.
\item Let $p \in [1,\infty)$. For every $\veps \in W^{1,p}_{\geps^D}(\oeps)$ it holds that
\begin{align*}
\|\veps \|_{L^p(\oeps)} \le C \|\nabla \veps\|_{L^p(\oeps)},
\end{align*}
with $C>0$ independent of $\vareps$.
\end{enumerate}
\end{lemma}
\begin{proof}
It is well-known, see for example \cite[Lemma II.6.1]{Galdi}, that for $\veps \in L^1(\oeps)$ with $\nabla \veps \in L^p(\oeps)^n$ it holds that$\veps \in W^{1,p}(\oeps)$.  Following the proof of \cite[Theorem 2.1]{Acerbi1992} (see also the construction of the local extension operator in the proof of \cite[Lemma 2.6]{Acerbi1992}), we obtain the existence of an extension operator $E^{\ast}_{\vareps}: W^{k,s}(\oeps) \rightarrow W^{k,s}(\Omega)$ for all $s \in [1,\infty)$ and $k=0,1$ (\ie $E^{\ast}_{\vareps}$ is acting on both spaces) with 
\begin{align*}
\|E^{\ast}_{\vareps} \weps\|_{L^s(\Omega)} &\le C_s \|\weps\|_{L^s(\oeps)} &\mbox{ for }& \weps\in L^s(\oeps),
\\
\|\nabla E^{\ast}_{\vareps} \weps\|_{L^s(\Omega)} &\le C_s \|\nabla \weps\|_{L^s(\oeps)} &\mbox{ for }& \weps \in W^{1,s}(\oeps),
\end{align*}
with a constant $C_s >0$ independent of $\vareps$, but depending on $s$. In other words, $E^{\ast}_{\vareps}$ is a  strong $1$-extension operator, see \cite[5.17]{AdamsSobelevSpaces2003}. Now, we obtain with the Poincar\'e inquality on $W^{1,p}(\Omega)$ (for functions with mean value zero)
\begin{align*}
\|\veps\|_{L^p(\oeps)} &\le \| E_{\vareps}^{\ast} \veps\|_{L^p(\Omega)} 
\le 
 C\left(\|\nabla E_{\vareps}^{\ast} \veps \|_{L^p(\Omega)} + \|E_{\vareps}^{\ast} \veps\|_{L^1(\Omega)}\right)
 \\
 &\le C \left(\|\nabla \veps \|_{L^p(\oeps)} + \|\veps\|_{L^1(\oeps)} \right).
\end{align*} 
Next, we consider $\veps \in W^{1,p}_{\geps^D}(\oeps)$. From Lemma \ref{lem:estimate_outer_bd_extension} we get for all $\weps \in W_{\geps^D}^{1,p}(\oeps)$ that
\begin{align}\label{TraceInequalityOuterBoundary}
\|\weps \|_{L^p(\partial \Gamma^D)} \le C \vareps^{1-\frac{1}{p}} \|\nabla \weps\|_{L^p(\oeps)}.
\end{align}

Now, the result follows by a standard contradiction argument, where we only sketch some details (see also the proof of Theorem \ref{TheoremKornContFunct} at the end of Section \ref{sec:Korn_connected} for a similar argument and more details): If the statement is false, we can find a sequence (we use the same index $\vareps$) $\veps \in W^{1,p}(\oeps)$ such that (here we use that the inequality is valid for fixed $\vareps$, but a constant which might depend on $\vareps$)
\begin{align*}
1 = \|\veps\|_{L^p(\oeps)} \geq C_{\vareps}\|\nabla \veps\|_{L^p(\oeps)}
\end{align*}
with $C_{\vareps} \to \infty$ for $\vareps \to 0$.
Using the extension operator from above we define  $\tveps:= E_{\vareps}^{\ast} \veps$ and easily obtain that $\tveps$ is bounded in $W^{1,p}(\Omega)$ with gradient converging to zero in $L^p(\Omega)^n$. Hence, we obtain the existence of a constant $C_0$ such that
\begin{align*}
\tveps &\rightharpoonup C_0 &\mbox{ weakly in }& W^{1,p}(\Omega),
\\
\tveps &\rightarrow C_0 &\mbox{ in }& L^p(\Omega).
\end{align*}
Due to $\eqref{TraceInequalityOuterBoundary}$ the trace of $\tveps$ on $\Gamma^D$ converges to zero  and therefore $C_0 = 0$, which contradicts $\|\veps\|_{L^p(\oeps)} = 1$.
\end{proof}

We have the following well-known trace inequality for heterogeneous domains $\oeps$, which can be obtained by a standard decomposition argument.
\begin{lemma}\label{Lem:TraceInequality}
Let $p \in [1,\infty)$. For every $\theta>0$ there exists a constant $C(\theta)>0$ independent of $\vareps$, such that for all $\ueps \in W^{1,p}(\oeps)$  it holds that
\begin{align*}
   \vareps^{\frac{1}{p}} \|\ueps\|_{L^p(\geps)}  \le C(\theta)\|\ueps\|_{L^p(\oeps)} + \theta \vareps \|\nabla \ueps\|_{L^p(\oeps)}.
\end{align*}
\end{lemma}

\subsection{Embedding results on perforated domains}

To obtain $\vareps$-uniform \textit{a priori} bounds in suitable H\"older spaces, which are important for the homogenization, it is necessary to use embedding results from (parabolic) Sobolev spaces into H\"older spaces. The crucial point is that the embedding constant is independent of $\vareps$ (or with an explicit dependence on $\vareps$).

\begin{lemma}\label{cor:embedding_Sobolev_hoelder}
For every $\veps \in W^{2,p}(\oeps)$ for $p>n$ it holds that $\veps \in  C^{1,\gamma}(\overline{\oeps})$  for $\gamma \in \left(0,1  - \frac{n}{p}\right] $ with 
\begin{align*}
\|\veps\|_{C^{1,\gamma}(\overline{\oeps})} \le C \|\veps\|_{W^{2,p}(\oeps)}
\end{align*}
for a constant $C>0$ independent of $\vareps$.
\end{lemma}
\begin{proof}
The embedding constant for the embedding $W^{2,p}(\Omega)$ into $C^{1,\gamma}(\overline{\Omega})$ is depending on $\Omega$, and therefore independent of $\vareps$. Now the result follows directly from Theorem \ref{thm:main_thm_extension_operator}, in fact we have
\begin{align*}
\|\veps\|_{C^{1,\gamma}(\overline{\oeps})} \le \|E_{\vareps} \veps\|_{C^{1,\gamma}(\overline{\Omega})} \le C \|E_{\vareps} \veps \|_{W^{2,p}(\Omega)} \le C\|\veps\|_{W^{2,p}(\oeps)}.
\end{align*}
\end{proof}

Next we give an embedding result from Bochner spaces into H\"older spaces. A result of this form in a more general context can be found in \cite[Theorem 5.2]{AmannCompactEmbeddingsVectorValuedFunctions}. For our applications the precise dependence on the scaling parameter $\vareps$ is important.
\begin{lemma}\label{lem:embedding_hoelder}
Let $p>n$ and $\theta \in (0,1)$ with
\begin{align*}
\theta < \frac{2 (p-n)}{np + 2(p-n)}.
\end{align*}
Then for $s=(1-\theta) \in (0,1)$ and $q\in (2,p)$  with 
\begin{align*}
\frac{1}{q} = \frac{\theta}{2} + \frac{1-\theta}{p}, 
\end{align*}
we have for $\beta = s - \frac{n}{q} \in (0,1)$ the embedding
\begin{align}
\begin{aligned}\label{embedding_hoelder_anis_sobolev}
L^{\infty}((0,T),W^{1,p}(\oeps)) \cap H^1((0,T),L^2(\oeps)) &\hookrightarrow  C^{0,\frac{\theta}{2}} \left([0,T],W^{s,q}(\oeps)\right)
\\
&\hookrightarrow C^{0,\frac{\theta}{2}} \left([0,T],C^{0,\beta}(\overline{\oeps})\right),
\end{aligned}
\end{align}
such that for a constant $C>0$ independent of $\vareps$ for every $\veps \in L^{\infty}((0,T),W^{1,p}(\oeps)) \cap H^1((0,T),L^2(\oeps))$ it holds that
\begin{align*}
\|\veps\|_{C^{0,\frac{\theta}{2}} \left([0,T],C^{0,\beta}(\overline{\oeps})\right)} &\le C\|\veps\|_{C^{0,\frac{\theta}{2}} \left([0,T],W^{s,q}(\oeps)\right)}
\\
&\le C\|\veps\|_{L^{\infty}((0,T),W^{1,p}(\oeps)) \cap H^1((0,T),L^2(\oeps))}.
\end{align*}

\end{lemma}
\begin{proof}
The result follows by a similar argument as in the proof of Lemma \ref{cor:embedding_Sobolev_hoelder} above by using the extension operator from Theorem \ref{thm:main_thm_extension_operator} and the general embedding $\eqref{embedding_hoelder_anis_sobolev}$ with $\Omega$ instead of $\oeps$. However, for the reader who is not familiar with parabolic Sobolev spaces and associated embedding results we give a detailed proof using more standard techniques.

First of all, the mean value theorem implies for all $t_1,t_2 \in (0,T)$
\begin{align*}
\|\veps(t_1) - \veps(t_2)\|_{L^2(\oeps)} &\le \int_{t_1}^{t_2} \|\dot{v}_{\vareps}\|_{L^2(\oeps)} ds
\\
&\le |t_1 - t_2|^{\frac12} \|\dot{v}_{\vareps}\|_{L^2((0,T),L^2(\oeps))}.
\end{align*}
We define the extension  $\tveps:= E_{\vareps}\veps\in L^{\infty}((0,T),W^{2,p}(\Omega))^n$ with the extension operator from Theorem \ref{thm:main_thm_extension_operator}. We emphasize that we used that the extension operator $E_{\vareps}$ acts pointwise in time. The Gagliardo-Nirenberg interpolation inequality (see for example \cite{brezis2018gagliardo} for an overview) implies
\begin{align*}
\|\veps(t_1) - \veps(t_2)\|_{W^{s,q}(\oeps)} &\le \|\tveps(t_1) - \tveps(t_2)\|_{W^{s,q}(\Omega)}
\\
&\le C\|\tveps(t_1) - \tveps(t_2)\|_{L^2(\Omega)}^{\theta} \|\tveps(t_1) - \tveps(t_2)\|_{W^{1,p}(\Omega)}^{1-\theta}
\\
&\le C\|\veps(t_1) - \veps(t_2)\|_{L^2(\oeps)}^{\theta} \|\veps(t_1) - \veps(t_2)\|_{W^{1,p}(\oeps)}^{1-\theta}
\end{align*}
for a constant $C>0$ independent of $\vareps$. In the last inequality we  again used that $E_{\vareps}$ acts pointwise in time, is a linear operator, and a total extension operator (it can be applied as a mapping on $L^2$ and $W^{1,p}$ with the same embedding constants).
Hence, we obtain
\begin{align*}
\|\veps(t_1) - \veps(t_2)\|_{W^{s,q}(\oeps)} \le C |t_1 - t_2|^{\frac{\theta}{2}} \|\dot{v}_{\vareps}\|_{L^2((0,T),L^2(\oeps))}^{\theta} \|\veps(t_1) - \veps(t_2)\|_{W^{1,p}(\oeps)}^{1-\theta}.
\end{align*}
An elemenatal calculation shows that $n < sq$. In fact, we have  
\begin{align*}
\theta < \frac{2 (p-n)}{np + 2(p-n)} \quad 
\Leftrightarrow \quad n< (1-\theta) \frac{2p }{\theta p + 2 (1-\theta)} = (1-\theta ) q = sq.
\end{align*}
Hence, we have the continuous embedding $W^{s,q}(\Omega) \hookrightarrow C^{0,\beta}(\overline{\Omega})$, see \cite[Theorem 8.2]{di2012hitchhikers}. This implies
\begin{align*}
\|\veps (t_1) - \veps(t_2)\|_{C^{0,\beta}(\overline{\oeps})} &\le \|\tveps(t_1)-\tveps(t_2)\|_{C^{0,\beta}(\overline{\Omega})}
\\
&\le C \| \tveps(t_1) - \tveps(t_2)\|_{W^{s,q}(\Omega)} 
\\
&\le C\|\veps(t_1)  - \veps(t_2)\|_{W^{s,q}(\oeps)}.
\end{align*}
In the last inequality we used that the extension operator $E_{\vareps}$ can be considered as extension operator from $W^{s,q}(\oeps)$ to $W^{s,q}(\Omega)$ with embedding constant independent of $\vareps$, since $W^{s,q}$ is an interpolation space. 
Therefore we get
\begin{align*}
\|\veps (t_1) - \veps(t_2)\|_{C^{0,\beta}(\overline{\oeps})} \le C|t_1 - t_2|^{\frac{\theta}{2}} \|\dot{v}_{\vareps}\|_{L^2((0,T),L^2(\oeps))}^{\theta} \|\veps\|_{L^{\infty}((0,T),W^{1,p}(\oeps))}^{1-\theta}.
\end{align*}
With similar arguments we can show that for every $t \in (0,T)$ it holds that
\begin{align*}
\|\veps(t)\|_{C^{0,\beta}(\overline{\oeps})} \le C\|\veps(t)\|_{W^{s,p}(\oeps)} \le C\|\dot{v}_{\vareps}\|_{L^2((0,T),L^2(\oeps))}^{\theta} \|\veps\|_{L^{\infty}((0,T),W^{1,p}(\oeps))}^{1-\theta}.
\end{align*}
\end{proof}

\section{Two-scale convergence and unfolding}
\label{SectionTwoScaleConvergence}
In this section we briefly summarize the concept of the two-scale convergence and the unfolding operator. These methods provide the basic techniques to pass to the limit $\vareps \to 0$ in the microscopic problem.

\subsection{Two-scale convergence}
\label{SubsectionTwoScaleConvergence}

We start with the definition of  the two-scale convergence, which was first introduced and analyzed in \cite{Nguetseng} and \cite{Allaire_TwoScaleKonvergenz}, see also \cite{LukkassenNguetsengWallTSKonvergenz}.  Here we formulate the two-scale convergence for time-dependent functions. However, the aforementioned references can be easily extended to such problems, since time acts as a parameter. For the initial data we also use the usual two-scale convergence for time-independent functions, which can be defined in the same way.

\begin{definition}\label{DefinitionTSConvergence}
A sequence $\ueps \in  L^q((0,T),L^p( \Omega))$ for $p,q\in [1,\infty)$ is said to converge  in the two-scale sense (in $L^q L^p$) to the limit function $u_0\in L^q((0,T),L^p( \Omega \times Y))$, if for every $\phi \in L^{q'}((0,T);L^{p'}( \Omega, C_{\per}^0(Y)))$  the following relation holds
\begin{align*}
\lim_{\vareps\to 0}\int_0^T \int_{\Omega}u_{\vareps}(t,x)\phi\left(t,x,\frac{x}{\vareps}\right)dxdt = \int_0^T\int_{\Omega}\int_Y u_0(t,x,y)\phi(t,x,y)dydxdt ,
\end{align*}
and we write $\ueps \rightwts{q,p} u_0$. For $q = p$ we use the notation $\ueps \rightwts{p} u_0$.

A  two-scale convergent sequence $u_{\vareps}$ convergences strongly in the two-scale sense to $u_0$ (in $L^q L^p$), if  
\begin{align*}
\lim_{\vareps\to 0}\|u_{\vareps}\|_{L^q((0,T), L^p( \Omega))} =\|u_0\|_{L^q((0,T), L^p( \Omega  \times Y))} ,
\end{align*}
and we write $\ueps \rightsts{q,p} u_0$. For $q = p$ we write $\ueps \rightsts{p} u_0$.
\end{definition}

In \cite{AllaireDamlamianHornung_TwoScaleBoundary,Neuss_TwoScaleBoundary} the method of two-scale convergence was extended to oscillating surfaces:
\begin{definition}\label{DefinitionTSKonvergenzBoundary}
A sequence of functions $\ueps  \in L^q((0,T),L^p(\Gamma_{\vareps}))$ for $p,q \in [1,\infty)$ is said to converge in the two-scale sense on the surface $\Gamma_{\vareps}$ (in $L^qL^p$) to a limit $u_0\in L^q((0,T),L^p(\Omega \times \Gamma))$, if for every $\phi \in C\left([0,T]\times\overline{\Omega},C_{per}^0(\Gamma)\right)$ it holds that
\begin{align*}
\lim_{\vareps \to 0} \vareps \int_0^T\int_{\Gamma_{\vareps}} u_{\vareps}(t,x)\phi\left(t,x,\frac{x}{\vareps}\right)d\sigma dt = \int_0^T\int_{\Omega}\int_{\Gamma} u_0(t,x,y)\phi(t,x,y)d\sigma_ydxdt .
\end{align*}
We write $\ueps \rightwts{q,p} u_0$ on $\geps$. For $q = p$ we use the notation $\ueps \rightwts{p} u_0$ on $\geps$.
\end{definition}

We have the following compactness results (see e.g., \cite{Allaire_TwoScaleKonvergenz,LukkassenNguetsengWallTSKonvergenz,Neuss_TwoScaleBoundary}, with slight modification for time-depending functions):
\begin{lemma}\label{BasicTwoScaleCompactness}\
For every $p,q \in (1,\infty)$ we have:
\begin{enumerate}
[label = (\roman*)]
\item For every bounded sequence $\ueps \in L^q((0,T),L^p(\Omega))$ there exists $u_0 \in L^q((0,T),L^p( \Omega \times Y))$ such that up to a subsequence
\begin{align*}
\ueps &\rightwts{q,p} u_0.
\end{align*}
\item\label{item:lem_basic_two_scale_comp} For every bounded sequence $\ueps \in L^q((0,T),W^{1,p}(\Omega))$ there exist $u_0 \in L^q((0,T),W^{1,p}( \Omega))$ and $u_1 \in L^q((0,T), L^p( \Omega,W^{1,p}_{\per}(Y)/\R))$, such that up to a subsequence
\begin{align*}
\ueps &\rightwts{q,p} u_0,
\\
\nabla \ueps &\rightwts{q,p} \nabla_x u_0 + \nabla_y u_1 .
\end{align*}
\end{enumerate}

\end{lemma}

\subsection{The unfolding operator}

When dealing with nonlinear problems it is helpful to work with the unfolding method which gives a characterization for the weak and strong two-scale convergence, see Lemma \ref{LemmaAequivalenzTSKonvergenzUnfolding} below.
The unfolding operator, in the beginning also called periodic modulation or dilation operator, was first used for homogenization in \cite{VogtHomogenization} and later in \cite{ArbogastDouglasHornung}. Later, this operator was analyzed more closely in \cite{Cioranescu_Unfolding1,Cioranescu_Unfolding2,Cioranescu_BoundaryUnfolding} under the name unfolding operator. For a detailed investigation of the unfolding operator and its properties we refer to \cite{CioranescuGrisoDamlamian2018}. For a perforated domain (here we also allow the case $Y_s = Y$, \ie $\oeps = \Omega$) we define the unfolding operator for $p ,q\in [1,\infty]$ by
\begin{align*}
\teps : L^q((0,T) ,L^p( \oeps)) \rightarrow L^q((0,T),L^p( \Omega \times Y_s)), \\
\teps(\ueps)(t,x,y)= \ueps\left(t,\vareps \left[\fxe\right] + \vareps y\right).
\end{align*}
In the same way, we define the boundary unfolding operator for the oscillating surface $\geps$ via
\begin{align*}
\teps : L^q((0,T), L^p( \geps)) \rightarrow L^q((0,T), L^p( \Omega \times \Gamma)), \\
\teps(\ueps)(t,x,y)= \ueps\left(t,\vareps \left[\fxe\right] + \vareps y\right).
\end{align*}
We emphasize that we use the same notation for the unfolding operator on $\oe$ and the boundary unfolding operator $\geps$. We summarize some basic properties of the unfolding operator, see \cite{CioranescuGrisoDamlamian2018} generalized in an obvious way to time depending functions with different integrability $p$ and $q$:

\begin{lemma}\label{LemmaPropertiesUnfoldingOperator}
Let $q,p \in [1,\infty]$.
\begin{enumerate}[label = (\roman*)]
\item   For $\ueps \in L^q((0,T),L^p(\oeps))$ it holds that
\begin{align*}
\| \teps (\ueps) \|_{L^p((0,T), L^p(\Omega \times Y_s))} = \|\ueps\|_{L^q((0,T),L^p(\oeps))}.
\end{align*}
\item For $\ueps \in L^q((0,T),W^{1,p}(\oeps))$ it holds that
\begin{align*}
\nabla_y \teps (\ueps )= \vareps \teps( \nabla_x \ueps).
\end{align*} 
\item  For $\ueps \in L^q((0,T),L^p(\geps))$ it holds that
\begin{align*}
\|\teps (\ueps) \|_{L^q((0,T),L^p(\Omega \times \Gamma))} = \vareps^{\frac{1}{p}} \|\ueps\|_{L^q((0,T),L^p(\geps))}.
\end{align*}
\end{enumerate}
\end{lemma}

The following lemma gives a relation between the unfolding operator and the two-scale convergence. Its proof is quite standard and we refer the reader to \cite{BourgeatLuckhausMikelic} and \cite{CioranescuGrisoDamlamian2018} for more details (see also \cite[Proposition 2.5]{visintin2006towards}).
\begin{lemma}\label{LemmaAequivalenzTSKonvergenzUnfolding}
Let $p,q \in (1,\infty)$.
\begin{enumerate}
[label = (\alph*)]
\item For   a sequence $\ueps \in L^q((0,T),L^p( \Omega))$, the following statements are equivalent:
\begin{enumerate}[label = (\roman*)]
\item $\ueps \rightwts{q,p} u_0$ ($\ueps \rightsts{q,p} u_0$)
\item $\teps (\ueps )\rightharpoonup u_0$  ($\teps(\ueps) \rightarrow u_0 $) in $L^q((0,T), L^p( \Omega \times Y))$.
\end{enumerate}
\item  For a sequence $\ueps \in L^q((0,T),L^p(\geps))$  the following statements are equivalent:
\begin{enumerate}[label = (\roman*)]
\item $\ueps \rightwts{q,p} u_0$ ($\ueps \rightsts{q,p} u_0$) on $\geps$,
\item $\teps (\ueps) \rightharpoonup u_0$ ($\teps(\ueps) \rightarrow u_0$) in $L^q((0,T),L^p(\Omega \times \Gamma))$.
\end{enumerate}
\end{enumerate}
\end{lemma}

\begin{remark}
In the same way we can define the time-independent unfolding operator (we use the same notations) and the same results are valid. 
\end{remark}

In the following Lemma we show that for a converging sequence in $C^0$ we obtain the convergence of the associated unfolded sequence in $L^{\infty}$:
\begin{lemma}\label{lem:unfolding_c0_linfty}
Let $\veps \in C^0(\overline{\Omega})$ with $\veps \rightarrow v_0$ in $C^0(\overline{\Omega})$. Then we have
\begin{align*}
\teps(\veps) \rightarrow v_0 \qquad\mbox{in } L^{\infty}(\Omega \times Y).
\end{align*}
The same result is valid for the time-dependent case.
\end{lemma}
\begin{proof}
For almost every $(x,y) \in  \Omega \times Y$ it holds that
\begin{align*}
\left|\teps(\veps) (x,y) - v_0(x)\right| \le& \left|\veps\left(\vareps \left[\fxe\right] + \vareps y\right) - v_0 \left(\vareps \left[\fxe\right] + \vareps y \right)\right|
\\
&+ \left|  v_0 \left(\vareps \left[\fxe\right] + \vareps y \right) - v_0(x)\right|
\\
\le& \|\veps - v_0 \|_{C^0( \overline{\Omega})} +  \left|  v_0 \left(\vareps \left[\fxe\right] + \vareps y \right) - v_0(x)\right|. 
\end{align*}
The first term converges to zero, due to the strong convergence of $\veps$ to $v_0$ in $C^0(\overline{\Omega})$. The second term because of the continuity of $v_0$ and $\vareps \left[\fxe\right] + \vareps y \rightarrow x$.
\end{proof}

\subsection{Two-scale compactness for second order Sobolev spaces}
\label{SectionTSConvergenceHigherOrder}

In this section we state a general two-scale compactness result for sequences with second order derivatives, respectively a compactness result for the associated unfolded sequence. The following Proposition was shown in \cite[Theorem 3.6]{Cioranescu_Unfolding2} for the time-independent case and can be easily extended functions depending on time. We use the local averaging operator for $v \in L^p(\Omega)$ defined for almost every $x\in \Omega$ by
\begin{align*}
\meps(v)(x):= \int_Y \teps(v) (x,y) dy
\end{align*}
and similar for vector valued functions. To keep the notation close to the existing literature we use the same notation as for the mechanical energy. However, there should be no confusion, since both are not used in the same sections.
\begin{proposition}\label{prop:ts_higher_order}
Let $p,s \in (1,\infty)$ and $\ueps$ a bounded sequence in $L^s((0,T),W^{2,p}(\Omega))$. Then there exist $u_0 \in L^s((0,T), W^{2,p}(\Omega))$ and $u_2 \in L^s((0,T),L^p(\Omega,W^{2,p}_{\per}(Y)/\R))$ such that up to a subsequence
\begin{align*}
\frac{1}{\varepsilon^2} \big(\teps (\ueps) -  \varepsilon \meps(\nabla \ueps) &\cdot y^c - \meps (\ueps) \big)
\\
&\rightharpoonup \frac12 \nabla_x^2 u_0 y^c \cdot y^c + u_2 \quad\mbox{ weakly in } L^s((0,T),L^p(\Omega,W^{2,p}(Y))).
\end{align*}
Especially it holds up to a subsequence
\begin{align*}
\teps(\ueps) &\rightharpoonup u_0 &\mbox{ weakly in }& L^s((0,T),L^p(\Omega \times Y)),
\\
\teps(\nabla \ueps) &\rightharpoonup \nabla u_0 &\mbox{ weakly in }& L^s((0,T),L^p(\Omega \times Y))^n,
\\
\teps(\nabla^2 \ueps ) &\rightharpoonup \nabla_x^2 u_0 + \nabla_y^2 u_2 &\mbox{ weakly in }& L^s((0,T),L^p(\Omega \times Y))^{n\times n}.
\end{align*}
In the notation of two-scale convergence we have up to a subsequence (see Lemma \ref{LemmaAequivalenzTSKonvergenzUnfolding})
\begin{align*}
\ueps \rightwts{s,p} u_0,\quad \nabla \ueps \rightwts{s,p} \nabla u_0,\quad 
\nabla^2 \ueps \rightwts{s,p} \nabla_x^2 u_0 + \nabla_y^2 u_2 .
\end{align*}
\end{proposition}
For perforated domains we get the following result:
\begin{corollary}\label{cor:two_scale_conv_second_order}
Let $p,s\in (1,\infty)$ and  $(\ueps)$ be a sequence in $L^s((0,T),W^{2,p}(\oeps))$ with
\begin{align*}
\|\ueps\|_{L^s((0,T),W^{2,p}(\oeps))} \le C
\end{align*}
for a constant $C>0$ independent of $\vareps$. Then there exists $u_0\in L^s((0,T),W^{2,p}(\Omega))$ and $u_2 \in L^s((0,T),L^p(\Omega,W^{2,p}_{\per}(Y_s)/\R))$ such that up to a subsequence
\begin{align*}
\chi_{\oeps} \nabla^2 \ueps &\rightwts{s,p} \chi_{Y_s}\left(\nabla_x^2 u_0 + \nabla_y^2 u_2\right).
\end{align*}
\end{corollary}
\begin{proof}
This is a consequence of Proposition \ref{prop:ts_higher_order} and the extension operator from Theorem \ref{thm:main_thm_extension_operator}.
\end{proof}

\section{Korn inequality for $\Omega \setminus \oeps$ connected}
\label{sec:Korn_connected}
Here, we give the detailed proof of the Korn inequality in Theorem \ref{thm:main_thm_Korn_inequality} for domains $\oeps$ where the perforations $\Omega \setminus \oeps$ are also connected, or in other words the inclusions $Y\setminus Y_s $ are not strictly included in $Y$. This leads to two crucial difficulties: The main problem is that the construction of the global extension operator by cluing together the extensions of the restrictions in every micro-cell $\vareps (Y_s + k)$ for $k\in K_{\vareps}$ does not work, because in general we obtain discontinuities across the interfaces of neighboring cells. To overcome this problem we consider local extensions for restriction not only to one micro-cell, but also neighboring cells and then using a partition of unity. This ideas was already used in \cite{Acerbi1992}. This leads to some additional technical calculations and in our case an additional term in the last estimate in Proposition \ref{TheoremGlobalExtension}. 

The second problem is that the Dirichlet-boundary $\geps^D$ is depending on $\vareps$ and changing in every $\vareps$-step. In the proof of Theorem \ref{TheoremKornContFunct}, see inequality $\eqref{ineq:thm_korn_C0_aux}$, we used  \cite[Corollary 4.1]{pompe2003korn} which is only valid for a fixed Dirichlet-boundary condition. To overcome this issue we use that the $L^p$-norm on the boundary of the extension on the perforated part $\Gamma^D\setminus \geps^D$ vanishes for $\vareps \to 0$.

We introduce some additional notations necessary for the construction of a global extension operator with similar properties as the operator $E_{A_{\vareps}}$ in Proposition \ref{TheoremGlobalExtension}. Throughout this section we assume $p\in (1,\infty)$. For $j\in \N$ we define 
\begin{align*}
\E^j := \{-j,\ldots,0,\ldots,j\}^n
\end{align*}
including shifts $l\in \Z^n$ with length $|l|:= \max\{l_1,\ldots,l_n\} \le j$. Further we define
\begin{align*}
\mathcal{Z}^j_s:= \left\{ W^I_s = \mathrm{int} \bigcup_{\alpha \in I} \overline{Y_s} \, : \, I \subset \E^j, \, W_s^I \mbox{ is connected}\right\}.
\end{align*}
Especially, we define 
\begin{align*}
\hY_s := W^{\E^1}_s = \mathrm{int}\bigcup_{\alpha \in \E^1} \overline{Y_s} \in \mathcal{Z}^1,
\end{align*}
and 
\begin{align*}
\hY:= \mathrm{int}\bigcup_{\alpha \in \E^1} \overline{Y}.
\end{align*}
Hence, $\hY_s$ (resp. $\hY$) are reference elements consisting of $Y_s$ (resp. $Y$) with all neighboring elements. Additionally, we set
\begin{align*}
K_{\vareps}^{\ast} := \left\{ \alpha \in \Z^n \, : \, \vareps(\hY^1 + \alpha) \cap \Omega \neq \emptyset\right\}\supset K_{\vareps}.
\end{align*}
 Now, for every $k \in K_{\vareps}^{\ast}$  there exists an $I\subset \E^1$ such that 
\begin{align*}
\vareps (\hY + k) \cap \oeps = \vareps (W_s^I + k).
\end{align*}
For most of the cells on the left-hand side we have $W_s^I = \hY_s$. More precisely, only for $\vareps(\hY + k) \cap \partial \Omega \neq \emptyset$ it holds that $W_s^I \neq \hY_s$.
For the construction of the global extension operator we use local extension operators on the reference elements $W_s^I \in \mathcal{Z}^1$ with $I\subset \E^1$, by transforming micro-cells $\vareps(W_s^I + k)$ for $k \in \Z^n$ to the fixed reference element. By the previous considerations  only for microscopic cells close to the boundary we have $W_s^I \neq \hY_s$.

\begin{lemma}\label{lem:local_extension_A_connected}
Let $p \in (1,\infty)  $ and $A\in \R^{n\times n}$ such that $\det(A)>0$.  Then there exists a constant $C>0$, such that for every $W_s^I \in \mathcal{Z}^1_s$ there exists an extension operator $\tau_A^I: W^{1,p}(W_s^I)^n \rightarrow W^{1,p}(\hY)^n$, such that the inequalities from Lemma \ref{lem:local_extension_A} are valid with $\hY$ instead of $Y$ and $W_s^I$ instead of $Y_s$.
\end{lemma}
\begin{proof}
The proof follows the same lines as the proof of Lemma \ref{lem:local_extension_A}. We only emphasize that the operator $S$ from \cite[Lemma 2]{GahnJaegerTwoScaleTools} used in the proof is also valid  with $\hY$ instead of $Y$ and $W_s^I$ instead of $Y_s$.
\end{proof}
Next, we give the generalization of Proposition \ref{TheoremGlobalExtension} for $\Omega \setminus \oeps$ connected. As already mentioned above, using the definition $\eqref{def:global_extension_E_Aeps}$ in the proof of Proposition \ref{TheoremGlobalExtension}, would in general lead to jumps across the interfaces of neighboring micro-cells. Hence, we follow an approach from \cite{Acerbi1992}.
Let us define for $\alpha \in \Z^n$  the transformation between the macro- and micro-cells by
\begin{align*}
\pi_{\epsilon,\alpha} (y):= \epsilon (y + \alpha).
\end{align*}
Now, let $(\phi^{\alpha})_{\alpha \in \Z^n}$ be a $Y$-periodic partition of unity with respect to the open cover $(\hY + \alpha)_{\alpha \in \Z^n}$, i.e., for all $\alpha, \beta \in \Z^n $ it holds that
\begin{align*}
\phi^{\beta} = \phi^{\alpha} \circ \pi_{1,\alpha - \beta}.
\end{align*}
Especially, we have
\begin{align*}
\sum_{k \in \Z^n} \phi^{\alpha + k} = \sum_{e \in \mathcal{E}^1} \phi^{\alpha + e} = 1 \qquad\mbox{in } Y + \alpha.
\end{align*}
If we define $\phi:= \phi^0$ we have
\begin{align*}
\phi^{\alpha} = \phi \circ \pi_{1,-\alpha} = \phi (\cdot_y - \alpha),
\end{align*}
and 
\begin{align}
\label{eq:sum_phi^e}
\sum_{e \in \mathcal{E}^1} \phi^e  = \sum_{e \in \mathcal{E}^1} \phi(\cdot + e) = 1 \qquad \mbox{in } Y .
\end{align}

\begin{lemma}\label{lem:aux_e_B(tau_B(v))}
For every $A,B\in \R^{n\times n}$ and every $v \in W^{1,p}(W_s^I)^n$ for $p\in (1,\infty)$ with $W_s^I \in \mathcal{Z}_s^1$ it holds that
\begin{align*}
\left\|e_B(\tau_A^I(v))\right\|_{L^p(\hY)} \le C \left\|e_A(v)\right\|_{L^p(W_s^I)} + C |B - A| |A^{-1}| |A| \|\nabla v\|_{L^p(W_s^I)}.
\end{align*}
\end{lemma}
\begin{proof}
An elemental calculation shows 
\begin{align*}
e_B(\tau_A^I(v)) = e_A(\tau_A^I(v)) + \frac12 \left[ (B-A) \nabla \tau_A^I(v) + \nabla \tau_A^I(v)^{\top} (B - A)^{\top}\right].
\end{align*}
Now the desired result follows from Lemma \ref{lem:local_extension_A_connected}.
\end{proof}

\begin{proposition}\label{prop:extension_Aeps_connected}
Let $A_{\vareps} \in L^{\infty}(\R^n)^{n\times n}$ such that $A_{\vareps}$ is constant on every micro-cell $\vareps (Y+ k)$ with $k\in \Z^n$ and fulfills
\begin{align*}
\|A_{\vareps}\|_{L^{\infty}(\R^n)} \le M, \qquad \det A_{\vareps} \geq \mu >0 \quad \mbox{a.\,e. in } \R^n
\end{align*}
for constants $\mu>0$ and $ M>0$ independent of $\vareps$. Then for every $p\in (1,\infty)$ there exists an extension operator $E_{A_{\vareps}} : W^{1,p}(\oeps)^n \rightarrow W^{1,p}(\Omega)^n$, such that  for all $\veps \in W^{1,p}(\oeps)^n$ it holds that
\begin{subequations}
\begin{align}
\label{ineq:glob_ext_Aeps_W1p}\|E_{A_\vareps}(\veps)\|_{L^p(\Omega)} &\le C_G \frac{M^n}{\mu} \left( \|\veps\|_{L^p(\oeps)}  + \vareps \|\nabla \veps \|_{L^p(\oeps)} \right),
\\
\label{ineq:glob_ext_Aeps_Grad}\|\nabla E_{A_{\vareps}}(\veps) \|_{L^p(\Omega)} &\le C_G \frac{M^n}{\mu} \|\nabla \veps \|_{L^p(\oeps)} ,
\\
\label{ineq:glob_ext_Aeps_eAeps}\|e_{A_{\vareps}}(E_{A_{\vareps}}(\veps))\|_{L^p(\Omega)} &\le C_G \|e_{A_{\vareps}} (\veps)\|_{L^p(\oeps)} + C_G \Delta_{\vareps,1} \left(\frac{M^n}{\mu} + 1\right) \|\nabla \veps \|_{L^p(\oeps)}
\end{align}
\end{subequations}
for a constant $C_G>0$ independent of $\vareps$ and
\begin{align*}
\Delta_{\vareps,1}:= \sup_{l\in K_{\vareps}, e \in \E^1} \left| A_{\vareps}^{l + e} - A_{\vareps}^l \right|.
\end{align*}

\end{proposition}

\begin{remark}
We assume that $A_{\vareps}$ is defined on the whole set $\R^n$. In the proof we will see that it is only important to have values in micro-cells in $\Omega$ or touching its lateral boundary $\partial \Omega$ (see the definition of $K_{\vareps}^{\ast}$).
\end{remark}
\begin{proof}
For every $\alpha \in K_{\vareps}^{\ast}$ there exists (a unique) $W_s^{I(\alpha)} \in \mathcal{Z}_s^1$, such that 
\begin{align}\label{RelationZs1}
\vareps(\hY + \alpha) \cap \oeps = \vareps (\hY_s + \alpha ) \cap \Omega   = \vareps( W_s^{I(\alpha)} + \alpha) = \pi_{\vareps,\alpha}\left( W_s^{I(\alpha)} \right).
\end{align}
Hence, for all $\veps \in W^{1,p}(\oeps)^n$  we have
\begin{align*}
\veps\vert_{\vareps(\hY_s + \alpha) \cap \Omega} \circ \pi_{\vareps,\alpha} \in W^{1,p}(W_s^{I(\alpha)})^n .
\end{align*}
In the following, we denote for $W_s^I \in \mathcal{Z}_s^1$ the associated extension operator from Lemma \ref{lem:local_extension_A_connected} by $\tau^I$. For every $\alpha \in K_{\vareps}^{\ast}$ we denote by $W_s^{I(\alpha)} \in \mathcal{Z}_s^1$ the associated set given by $\eqref{RelationZs1}$, and define 
\begin{align*}
v_{\vareps,\alpha} := \veps\vert_{\vareps(\hY_s + \alpha) \cap \Omega} \circ \pi_{\vareps,\alpha}  = \veps\vert_{\vareps(\hY_s + \alpha) \cap \Omega} \left(\vareps (\cdot_y + \alpha)\right)
\end{align*}
 and with $A_{\vareps}^{\alpha}:= A_{\vareps}|_{\vareps(Y + \alpha)}$ the local extension  by 
\begin{align*}
\tilde{v}_{\vareps,\alpha}:= \tau^{I(\alpha)}_{A_{\vareps}^{\alpha}} (v_{\vareps,\alpha}).
\end{align*}
Especially, $\tilde{v}_{\vareps,\alpha}$ fulfills
\begin{align*}
\|\tilde{v}_{\vareps,\alpha}\|_{L^p(\hY)} &\le C |A_{\vareps}^{\alpha}| |(A_{\vareps}^{\alpha})^{-1}|\|v_{\vareps,\alpha}\|_{W^{1,p}(W_s^{I(\alpha)})},
\\
 \|\nabla_y \tilde{v}_{\vareps,\alpha}\|_{L^p(\hY)}&\le C |A_{\vareps}^{\alpha}| |(A_{\vareps}^{\alpha})^{-1}|\|\nabla_y v_{\vareps,\alpha}\|_{L^p(W_s^{I(\alpha)})} ,
 \\
 \|e_{A_{\vareps}^{\alpha}}(\tilde{v}_{\vareps,\alpha})\|_{L^p(\hY)} &\le C\|e_{A_{\vareps}^{\alpha}}(v_{\vareps,\alpha})\|_{L^p(W_s^{I(\alpha)})}.
\end{align*}
Now, we define the global extension operator $E_{A_{\vareps}} : W^{1,p}(\oeps)^n \rightarrow W^{1,p}(\Omega)^n$ by
\begin{align*}
E_{A_{\vareps}}\veps := \sum_{\alpha \in K_{\vareps}^{\ast}} \left(\tilde{v}_{\vareps,\alpha} \circ \pi_{\vareps,\alpha}^{-1}\right) \phi^{\alpha}\left(\frac{\cdot}{\vareps}\right).
\end{align*}
Hence, for almost every $x \in \Omega$ it holds that
\begin{align*}
E_{A_{\vareps}}\veps (x) = \sum_{\alpha \in K_{\vareps}^{\ast}} \tilde{v}_{\vareps,\alpha}\left(\frac{x}{\vareps} - \alpha\right) \phi\left(\frac{x}{\vareps} - \alpha\right).
\end{align*}
It is straightforward to check that $E_{A_{\vareps}}$ is in fact an extension operator.
Next, we consider the $L^p$-norm of $E_{A_{\vareps}}$. This illustrates already the main steps for the more technical estimates for the terms including $\nabla E_{A_{\vareps}}(\veps) $ and $e_{A_{\vareps}}(E_{A_{\vareps}}(\veps))$. For $k_0 \in K_{\vareps}$ we have
\begin{align*}
\|E_{A_{\vareps}}&(\veps)\|_{L^p(\vareps(Y + k_0))}^p = \int_{\vareps (Y + k_0)} \left| \sum_{\alpha \in K_{\vareps}^{\ast}} \tilde{v}_{\vareps,\alpha} \left(\fxe - \alpha \right) \phi \left(\fxe - \alpha\right)\right|^p dx
\\
&= \vareps^n \int_Y \left| \sum_{\alpha \in\E^1 + k_0} \tilde{v}_{\vareps,\alpha} (y + k_0 - \alpha) \phi(y + k_0 - \alpha) \right|^p dy
\\
&= \vareps^n \int_Y \left| \sum_{e \in \E^1} \tilde{v}_{\vareps, e + k_0} (y - e) \phi(y-e) \right|^p dy 
\\
&\le C \vareps^n \sum_{e \in \E^1} \int_{Y + e} \left|\tilde{v}_{\vareps,e +k_0}(y)\right|^p dy
\\
&\le C \vareps^n \sum_{e \in \E^1} \int_{\hY} \left|\tilde{v}_{\vareps,e + k_0}(y) \right|^p dy
\\
&\le C \vareps^n \sum_{e \in \E^1} \left|A_{\vareps}^{e + k_0}\right|^p \left|(A_{\vareps}^{e+ k_0})^{-1}\right|^p \int_{W_s^{I(e+ k_0)}} \left| \veps(y + e  + k_0)\right|^p + \vareps^p \left| \nabla \veps (y + e + k_0)\right|^p dy 
\\
&\le C\left(\frac{M^n}{\mu}\right)^p \sum_{e \in \E^1} \int_{\vareps \left( W_s^{I(e  + k_0)}+ e + k_0\right)} |\veps(x)|^p + \vareps^p \left|\nabla \veps\right|^p dx
\\
&= C \left(\frac{M^n}{\mu}\right)^p \sum_{e \in \E^1} \int_{\vareps (\hY + k_0 + e) \cap \oeps} |\veps(x)|^p + \vareps^p \left|\nabla \veps(x)\right|^p dx
\\
&\le C \left(\frac{M^n}{\mu}\right)^p \int_{\vareps ( \hY^2 + k_0)} |\veps(x)|^p + \vareps^p \left|\nabla \veps(x)\right|^p dx.
\end{align*}
Summing up over $k_0 \in K_{\vareps}$ implies inequality $\eqref{ineq:glob_ext_Aeps_W1p}$.
Next, we consider the norm of $e_{A_{\vareps}}(E_{A_{\vareps}}(\veps))$. We have for all $k_0 \in K_{\vareps}$
\begin{align*}
\big\| &e_{A_{\vareps}}(E_{A_{\vareps}}(\veps))\big\|_{L^p(\vareps(Y + k_0))}^p 
\\
&= \int_{\vareps (Y + k_0)} \left|\frac12 \sum_{\alpha \in K_{\vareps}^{\ast}} \left\{ A_{\vareps}^{k_0} \nabla \left(\tilde{v}_{\vareps,\alpha} \left(\fxe - \alpha\right) \phi\left(\fxe - \alpha\right) \right) + \nabla \left(A_{\vareps}^{k_0} \tilde{v}_{\vareps,\alpha}\left(\fxe-\alpha\right) \phi\left(\fxe -\alpha\right)\right)^{\top} \right\} \right|^p dx
\\
&\le C\vareps^{-p}\bigg\{\int_{\vareps (Y + k_0)} \left| \sum_{\alpha \in K_{\vareps}^{\ast}} \left[A_{\vareps}^{k_0} \nabla_y \tilde{v}_{\vareps,\alpha}\left(\fxe - \alpha\right) + \nabla_y\tilde{v}_{\vareps,\alpha}\left(\fxe -\alpha\right)^{\top}(A_{\vareps}^{k_0})^{\top}\right] \phi\left(\fxe -\alpha\right)  \right|^p dx
\\
& \hspace{3em} + \int_{\vareps (Y + k_0)} \left| \sum_{\alpha \in K_{\vareps}^{\ast}} \left[A_{\vareps}^{k_0} \tilde{v}_{\vareps,\alpha}\left(\fxe - \alpha\right) \otimes \nabla_y \phi\left(\fxe -\alpha\right) + \nabla_y \phi\left(\fxe- \alpha\right) \otimes A_{\vareps}^{k_0}\tilde{v}_{\vareps,\alpha}\left(\fxe-\alpha\right)\right] \right|^p dx \bigg\}
\\
&=: B_{\vareps}^1 + B_{\vareps}^2.
\end{align*}
We estimate the two terms $B_{\vareps}^1$ and $B_{\vareps}^2$ separately.
As in the proof of Proposition \ref{TheoremGlobalExtension} we use the same notation $e_A$ (for arbitrary $A\in \R^{n\times n}$)
for derivatives with respect to $x$ and $y$.
With similar arguments as above we obtain
\begin{align*}
B_{\vareps}^1 &=  C\vareps^{n-p} \int_{Y} \left| \sum_{e \in \E^1}\left[ A_{\vareps}^{k_0} \nabla_y \tilde{v}_{\vareps,e+k_0} (y - e) + \nabla_y \tilde{v}_{\vareps,e+k_0} (y-e )^{\top } (A_{\vareps}^{k_0})^{\top} \right] \phi(y -e) \right|^p dy
\\
&\le C \vareps^{n-p} \sum_{e \in \E^1} \int_{\hY} \left| A_{\vareps}^{k_0} \nabla_y \tilde{v}_{\vareps,e+k_0}(y) + \nabla_y \tilde{v}_{\vareps,e + k_0}(y)^{\top} (A_{\vareps}^{k_0})^{\top} \right|^p dy
\\
&= C \vareps^{n-p} \sum_{e \in \E^1} \int_{\hY} \left|e_{A_{\vareps}^{k_0}} \left( \tau_{A_{\vareps}^{e+k_0}}^{I(e + k_0)} (v_{\vareps,e + k_0}) \right) \right|^p dy.
\end{align*}
Now we apply Lemma \ref{lem:aux_e_B(tau_B(v))} to obtain
\begin{align*}
B_{\vareps}^1 &\le C \vareps^{n-p}\sum_{e \in \E^1} \bigg\{ \int_{W_s^{I(e + k_0)}} \left| e_{A_{\vareps}^{e +k_0}} (v_{\vareps, e + k_0}) \right|^p dy 
\\
&\hspace{4em} + \left|A_{\vareps}^{e + k_0} - A_{\vareps}^{k_0}\right|^p |A_{\vareps}^{e +k_0}|^p \left|(A_{\vareps}^{e + k_0})^{-1}\right|^p \int_{W_s^{I(e + k_0)}}| \nabla_y v_{\vareps,e + k_0}|^p dy \bigg\}
\\
&\le C\vareps^{n-p} \sum_{e \in \E^1} \left\{ \int_{W_s^{I(e + k_0)}} \left| e_{A_{\vareps}^{e +k_0}} (v_{\vareps, e + k_0}) \right|^p dy + \Delta_{\vareps,1}^p \left(\frac{M^n}{\mu}\right)^p   \int_{W_s^{I(e + k_0)}}| \nabla_y v_{\vareps,e + k_0}|^p dy 
\right\}
\\
&=: B_{\vareps}^{1,1} + B_{\vareps}^{1,2}.
\end{align*}
For the first term we get (with similar arguments as in the proof of Lemma \ref{lem:aux_e_B(tau_B(v))})
\begin{align*}
B_{\vareps}^{1,1} &= C  \sum_{e \in \E^1} \int_{\vareps( W_s^{I(e + k_0)} + k_0 + e) } \left| e_{A_{\vareps}^{e + k_0}} (\veps)\right|^p dx
\\
&\le C \sum_{e \in \E^1}  \left\{ \int_{\vareps(W_s^{I(e + k_0)} + k_0 + e)} \left| e_{A_{\vareps}^{k_0}} (\veps)\right|^p + 2^p \left|A_{\vareps}^{k_0} - A_{\vareps}^{k_0 + e} \right|^p |\nabla \veps |^p dx \right\}
\\
&\le C \left\{\int_{\vareps (\hY^2 + k_0) \cap \oeps }  \left| e_{A_{\vareps}^{k_0}} (\veps)\right|^p  dx + \Delta_{\vareps,1}^p \int_{\vareps(\hY^2 + k_0) \cap \oeps} |\nabla \veps|^p dx\right\}.
\end{align*}
Hence, we obtain
\begin{align*}
B_{\vareps}^1 \le  C \left\{\int_{\vareps (\hY^2 + k_0) \cap \oeps }  \left| e_{A_{\vareps}^{k_0}} (\veps)\right|^p  dx + \Delta_{\vareps,1}^p \left( 1+  \left(\frac{M^n}{\mu}\right)^p \right) \int_{\vareps(\hY^2 + k_0)\cap \oeps} |\nabla \veps|^p dx\right\}.
\end{align*}
It remains to estimate the term $B_{\vareps}^2$. Using $|B| = |B^{\top}|$ for all $B \in \R^{n\times n}$ and $\sum_{e \in \E^1} \nabla_y \phi(y-e) = 0$, see $\eqref{eq:sum_phi^e}$, we get
\begin{align*}
B_{\vareps}^2&\le C\vareps^{n-p} \int_Y \left| \sum_{e \in \E^1} A_{\vareps}^{k_0} \tilde{v}_{\vareps,e + k_0} (y - e) \otimes \nabla_y \phi(y-e)\right|^p dy
\\
&= C\vareps^{n-p} \int_Y \left| \sum_{e \in \E^1} A_{\vareps}^{k_0} \left(\tilde{v}_{\vareps,e+k_0}(y - e ) - \tilde{v}_{\vareps,k_0}(y) \right) \otimes \nabla_y \phi(y - e)\right|^p dy
\\
&\le C\vareps^{n-p} \sum_{e \in \E^1} \int_Y \left| A_{\vareps}^{k_0} \left(\tilde{v}_{\vareps,e+k_0}(y - e ) - \tilde{v}_{\vareps,k_0}(y) \right)\right|^p dy.
\end{align*}
We note that for almost every $y \in Y_s$ it holds 
\begin{align}\label{id:v_eps_v_eps}
\tilde{v}_{\vareps,e+k_0}(y - e ) - \tilde{v}_{\vareps,k_0}(y) = 0.
\end{align}
In fact, for $y \in Y_s$ we have $\vareps(y + k_0) \in \vareps(\hY_s + k_0 + e)$ and since $k_0 \in K_{\vareps}$ it holds that $\vareps (y + k_0) \in \oeps$. From this we obtain together with $\eqref{RelationZs1}$ that
\begin{align*}
\vareps(y + k_0) \in \vareps(\hY + e + k_0) \cap \oeps = \vareps\left(W_s^{I(e + k_0)} + e + k_0\right).
\end{align*}
This implies $y- e \in W_s^{I(e + k_0)}$. Hence, for almost every $y \in Y_s$ we have that
\begin{align*}
\tilde{v}_{\vareps,e + k_0}(y-e) = v_{\vareps,e + k_0}(y - e) = \veps(\vareps (y + k_0)) = \tilde{v}_{\vareps,k_0}(y).
\end{align*}
This is $\eqref{id:v_eps_v_eps}$.
Therefore, we can apply the standard Korn inequality for functions vanishing on $\Gamma$ to obtain
\begin{align*}
B_{\vareps}^2 &\le C \vareps^{n-p} \sum_{e \in \E^1} \int_Y \left| e \left(A_{\vareps}^{k_0}\tilde{v}_{\vareps,e+k_0}(\cdot_y -e ) - A_{\vareps}^{k_0} \tilde{v}_{\vareps,k_0}\right)(y) \right|^p dy
\\
&= C \vareps^{n-p} \sum_{e \in \E^1} \int_Y \left|e_{A_{\vareps}^{k_0}}(\tilde{v}_{\vareps,e + k_0})(y - e) - e_{A_{\vareps}^{k_0}}(\tilde{v}_{\vareps,k_0})\right|^p dy
\\
&\le C \vareps^{n-p} \sum_{e \in \E^1} \int_{\hY} \left| e_{A_{\vareps}^{k_0}}(\tilde{v}_{\vareps,e + k_0}) \right|^p dy.
\end{align*}
Now, this term can be estimated in the same way as $B_{\vareps}^1$, to obtain
\begin{align*}
B_{\vareps}^2 \le   C \left\{\int_{\vareps (\hY^2 + k_0) \cap \oeps }  \left| e_{A_{\vareps}^{k_0}} (\veps)\right|^p  dx + \Delta_{\vareps,1}^p \left( 1+  \left(\frac{M^n}{\mu}\right)^p \right) \int_{\vareps(\hY^2 + k_0)\cap \oeps} |\nabla \veps|^p dx\right\}.
\end{align*}
Summing again over $k_0 \in K_{\vareps}$ we get inequality $\eqref{ineq:glob_ext_Aeps_eAeps}$. Estimate $\eqref{ineq:glob_ext_Aeps_Grad}$ follows by similar arguments (using Poincar\'e inequality instead of the Korn inequality).
\end{proof}
In contrast to Proposition \ref{TheoremGlobalExtension}, we obtain in the case $\Omega \setminus \oeps$ connected an additional term in $\eqref{ineq:glob_ext_Aeps_eAeps}$ depending on the gradient. However, we will see that in the following proof of Theorem \ref{TheoremKornContFunct} this term is negligible. The other difficulty in the proof, in contrast to the proof in Section \ref{sec:Korn_inequality}, is that the Dirichlet-boundary $\geps^D$ is depending on $\vareps$. To overcome this problem we use a contradiction argument and the following lemma (which was also obtained for thin layers in \cite[Lemma 4.2]{GahnJaegerTwoScaleTools}):
\begin{lemma}\label{lem:estimate_outer_bd_extension}
Let $p \in (1,\infty)$. For every $\weps \in W_{\geps^D}^{1,p}(\Omega)$ it holds that
\begin{align*}
\|\weps\|_{L^p(\Gamma^D)} \le C \vareps^{1 - \frac{1}{p}} \|\nabla \weps\|_{L^p(\Omega)}
\end{align*}
for a constant $C>0$ independent of $\vareps$.
\end{lemma}
\begin{proof}
 Let 
\begin{align*}
K_{\vareps}^{b,D}:= \left\{k \in K_{\vareps} \, : \, \vareps\left( \overline{Y} + k\right) \cap \partial \Gamma^D \neq \emptyset \right\}.
\end{align*}
In other words, $K_{\vareps}^{b,D}$ are the elements, such that the micro-cells $\vareps(Y + k)$ for $k\in K_{\vareps}^b$ touch the outer boundary $\partial \Omega$. For $k \in K_{\vareps}^b$ we define 
\begin{align*}
\Gamma^D_k:= \left\{y \in \partial Y\, : \, \vareps(y + k) \in \geps^D\right\}.
\end{align*}
Then, the function $\weps(\vareps(\cdot_y + k)) \in W^{1,p}(Y)$ has a zero boundary condition on $\Gamma^D_k$. We emphasize that number of elements  in the set $\{\Gamma^D_k\}_{k\in K_{\vareps}}$ is independent of $\vareps$ and only depending on $n$ (also for domains $\Omega$ given as the union of rectangles). We apply the Poincar\'e and the trace inequality to obtain 
\begin{align*}
\|\weps\|_{L^p(\Gamma^D)}^p &= \sum_{k \in K_{\vareps}^{b,D}} \|\ueps\|^p_{L^p(\vareps (\partial Y + k) \cap \partial \Gamma^D)}
\\
&\le \vareps^{n-1} \sum_{k \in K_{\vareps}^{b,D}} \|\weps(\vareps(\cdot_y + k))\|^p_{L^p(\partial Y)}
\\
&\le C \vareps^{n-1 + p} \sum_{k \in K_{\vareps}^{b,D}} \|\nabla \weps(\vareps(\cdot_y + k))\|^p_{L^p(Y)}
\\
&\le C\vareps^{p-1} \|\nabla \weps\|_{L^p(\Omega)}^p.
\end{align*}
This implies the desired result.
\end{proof}
Now we prove Theorem \ref{TheoremKornContFunct}:
\begin{proof}[Proof of Theorem \ref{TheoremKornContFunct}]
We assume that the inequality from Theorem \ref{TheoremKornContFunct} is not valid for all $\veps \in W^{1,p}_{\Gamma_D^{\vareps}}(\oeps)^n$. Hence, there exists a subsequence $\{\vareps_n\}_{n\in \N} \subset \{\vareps\}$ and elements $v_{\vareps_n} \in W^{1,p}(\Omega_{\vareps_n})^n$ such that
\begin{align*}
1 = \|\nabla v_{\vareps_n} \|_{L^p(\Omega_{\vareps_n})} > n \|e_A(v_{\vareps_n})\|_{L^p(\Omega_{\vareps_n})}.
\end{align*}
Further, we can assume without loss of generality that $\vareps_n \to 0$ for $n\to \infty$. In fact, otherwise we have $\liminf_{n\to \infty} \vareps_n \neq 0$. Since $\vareps$ is a sequence tending to zero, the sequence $\{\vareps_n\}_{n\in \N}$ only consists of finitely many elements which we denote by $\vareps^1,\ldots,\vareps^m$ for $m\in \N$. From  \cite[Corollary 4.1]{pompe2003korn}  we obtain the existence of constants $C_{\vareps^i}>0$ for $i=1,\ldots,m$, such that for all $w_{\vareps^i} \in W^{1,p}(\Omega_{\vareps^i})^n$ it holds that
\begin{align*}
\|\nabla w_{\vareps^i}\|_{L^p(\Omega_{\vareps^i})} \le C_{\vareps^i} \|e_A(w_{\vareps^i})\|_{L^p(\Omega_{\vareps^i})}.
\end{align*}
Choosing $C_{\max}:= \max_{i=1,\ldots,m} \{C_{\vareps^i}\}$, we obtain in particular for the sequence $v_{\vareps_n}$ that for all $n\in \N$
\begin{align*}
n \|e_A(v_{\vareps_n})\|_{L^p(\Omega_{\vareps_n})} <\|\nabla v_{\vareps_n}\|_{L^p(\Omega_{\vareps_n})} \le C_{\max} \|e_A(v_{\vareps_n})\|_{L^p(\Omega_{\vareps_n})}.
\end{align*}
This leads to a contradiction for $n\to \infty$ and we have $\liminf_{n \to \infty} \vareps_n = 0$ and we can assume $\vareps_n\to 0$ for $n\to \infty$.

To simplify the notation we denote the subsequence $\vareps_n$ by $\vareps$. More precisely, we can formulate our assumption in the following way: There exists a sequence $\veps \in W^{1,p}(\oeps)$ and  constants $C_{\vareps}>0$ with $C_{\vareps} \to \infty$ for $\vareps\to 0$ such that
\begin{align*}
1 = \|\nabla \veps \|_{L^p(\oeps)} > C_{\vareps} \|e_A(\veps)\|_{L^p(\oeps)}.
\end{align*}
Especially, we get
\begin{align*}
\|e_A(\veps)\|_{L^p(\oeps)} \overset{\vareps \to 0}{\longrightarrow} 0.
\end{align*}
Let $\Omega^{\delta}:=\{x \in \R^n\, : \, \mathrm{dist}_{\infty}(x,\Omega)<\delta\}$ for $\delta >0$ be the $\delta$-neighborhood of $\Omega$. Due to the Tietze extension theorem, we can extend $A$ to a function in $C^0(\overline{\Omega^{\delta}})^{n\times n}$ (we use the same notation for the extension $A$) for arbitrary $\delta>0$. The Lipschitz continuity of the determinant and the continuity of $A$ imply the existence of $\mu>0$ and $\delta>0$, such that
\begin{align*}
\det(A)\geq \mu >0 \qquad\mbox{in } \overline{\Omega^{\delta}}.
\end{align*}
In the following we assume without loss of generality that $\vareps <\delta$.
Now, let $A_{\vareps} \in L^{\infty}(\Omega^{\delta})^{n\times n}$ be a sequence of matrix-valued functions which are constant on every micro-cell $\vareps(Y+ k)$ for $k\in K_{\vareps}^{\ast}$, and converging to $A$ in $L^{\infty}(\Omega^{\delta})^{n\times n}$ (this is possible since $A \in C^0(\overline{\Omega^{\delta}})^{n\times n}$). We emphasize that $\vareps (Y + k) \subset \Omega^{\delta}$ for $k\in K_{\vareps}^{\ast}$ and we are not interested in the values of $A_{\vareps}$ in micro-cells with $k\in \Z^n\setminus K_{\vareps}^{\ast}$.
There exists a constant $M>0$, such that (for more details we refer to the proof of Theorem \ref{TheoremKornContFunct} in Section \ref{sec:Korn_inequality})
\begin{align*}
\|A_{\vareps}\|_{L^{\infty}(\Omega^{\delta})} \le M, \qquad \det A_{\vareps} \geq \mu >0 \quad \mbox{ a.\,e. in } \Omega^{\delta}.
\end{align*}
Further we define $\tveps := E_{A_{\vareps}} (\veps)$ with the extension operator $E_{A_{\vareps}}$ from Proposition \ref{prop:extension_Aeps_connected}.  From the properties of the extension operator $E_{A_{\vareps}}$ and the Poincar\'e inequality in Lemma \ref{lem:Poincare-inequality} we obtain for a constant $C>0$ independent of $\vareps$
\begin{align*}
\|\tveps\|_{W^{1,p}(\Omega)} \le C \|\nabla \veps \|_{L^p(\oeps)} \le C.
\end{align*}
Hence, $\tveps$ is bounded in $W^{1,p}(\Omega)^n$. Further, we get from inequality $\eqref{ineq:glob_ext_Aeps_eAeps}$ that
\begin{align*}
\|e_{A_{\vareps}}(\tveps)\|_{L^p(\Omega)} &\le C\left( \|e_{A_{\vareps}} (\veps)\|_{L^p(\oeps)} +  \Delta_{\vareps,1} \|\nabla \veps \|_{L^p(\oeps)} \right),
\end{align*}
with 
\begin{align*}
\Delta_{\vareps,1}:= \sup_{l\in K_{\vareps}, e \in \E^1} \left| A_{\vareps}^{l + e} - A_{\vareps}^l \right| \le \sup\left\{|A(x) - A(y)| \, : \, x,y \in \overline{\Omega^{\delta}}, \, |x-y | \le 2\sqrt{n}\vareps\right\}.
\end{align*}
Therefore $\Delta_{\vareps,1} \rightarrow 0$ for $\vareps\to 0$ and we obtain
\begin{align*}
\|e_A(\tveps)\|_{L^p(\Omega)} &\le \|e_{A_{\vareps}}(\tveps)\|_{L^p(\Omega)} + 2 \|A_{\vareps} - A\|_{L^{\infty}(\Omega)}\underbrace{\|\nabla \tveps\|_{L^p(\Omega)}}_{\le C}
\\
&\le C \left( \|e_{A_{\vareps}}(\veps)\|_{L^p(\oeps)} + \|A_{\vareps} - A\|_{L^{\infty}(\Omega)} + \Delta_{\vareps,1}\right)
\\
&\overset{\vareps\to 0}{\longrightarrow} 0.
\end{align*}
From  \cite[Theorem 2.3]{pompe2003korn} we obtain the existence of a constant $C_A>0$ independent of $\vareps$, such that 
\begin{align*}
1 \le \|\nabla \tveps\|_{L^p(\Omega)} \le C_A \left( \|\tveps\|_{L^p(\Gamma^D)} + \|e_A(\tveps)\|_{L^p(\Omega)} \right).
\end{align*}
The second term on the right-hand side tends to zero, due to the calculations above, and the first term tends  to zero because of Lemma \ref{lem:estimate_outer_bd_extension},
what leads to a contradiction and implies the desired result.
\end{proof}

\end{appendix}

\bibliographystyle{abbrv}
\bibliography{literature} 

\end{document}